\documentclass[a4paper,12pt]{amsart}
\usepackage[T2A]{fontenc}
\usepackage{amsmath}
\usepackage{extarrows}
\usepackage{amssymb,amsthm,amscd}
\usepackage{fullpage}
\usepackage{hyperref}
\usepackage{xcolor} 

\usepackage[normalem]{ulem} 
\usepackage[toc,page]{appendix}
\input{diagrams.tex}
\newarrow {Dashto} {}{dash}{}{dash}>

\usepackage{graphicx}
\usepackage{tikz}
\usepackage{tikz-cd}
\usepackage{tikz}
\usetikzlibrary{matrix,arrows,positioning,calc,chains}

\usepackage[lmargin=3cm,rmargin=3cm,tmargin=2cm,bmargin=2cm, marginpar=2cm ]{geometry}

\usepackage{enumerate} 
\usepackage{pict2e}
\usepackage{graphicx,color}
\usepackage{caption}
\usepackage{subcaption}
\usetikzlibrary{calc}

\usetikzlibrary{positioning}
\tikzset{main node/.style={circle,fill=black!20,draw,minimum size=1cm,inner sep=0pt},
}

\tikzstyle{vertex}=[circle, draw, fill=black, inner sep=0pt, minimum size=4pt]
\tikzstyle{edge}=[line width=1.5pt,black!50!white]
\tikzstyle{gridp}=[inner sep=1pt,circle,fill=black!70!white]
\tikzstyle{gridl}=[black!50!white]

\tikzstyle{lnode}=[circle,white,draw=black!60!white,fill=black!60!white,inner sep=1pt]
\tikzstyle{cnode}=[circle,draw=black!60!white,fill=black!60!white,inner sep=1.5pt]
\tikzstyle{redge}=[edge,Red]
\tikzstyle{bedge}=[edge,NavyBlue]

\tikzstyle{nvertex}=[vertex, draw=ncol, fill=ncol]
\tikzstyle{edgeq}=[edge,gray!60,densely dashed]
\tikzstyle{nedge}=[edge,ncol]
\tikzstyle{oedge}=[edge,Red!60!black]

\theoremstyle{plain}
\newtheorem{theorem}{Theorem}[subsection]
\newtheorem{lemma}[theorem]{Lemma}

\newtheorem{conjecture}[theorem]{Conjecture}

\newtheorem{proposition}[theorem]{Proposition}
\newtheorem{graph lemma}[theorem]{Graph Lemma}

\newtheorem{corollary}[theorem]{Corollary}
\theoremstyle{definition}
\newtheorem{definition}[theorem]{Definition}
\newtheorem{example}[theorem]{Example}
\newtheorem{examples}[theorem]{Examples}
\newtheorem{remark}[theorem]{Remark}
\newtheorem{remarks}[theorem]{Remarks}
\newtheorem{sit}[theorem]{}
\newtheorem{nota}[theorem]{Notation}

\def\Aut{\operatorname{Aut}}
\def\Inn{\operatorname{Inn}}

\def\Br{\operatorname{Br}}
\def\Vert{\operatorname{Vert}}
\def\Reb{\operatorname{Edge}}

\def\Ratio{\operatorname{Ratio}}

\def\Tip{\operatorname{Tip}}

\def\Bir{\operatorname{Bir}}
\def\SAut{\operatorname{SAut}}
\def\Supp{\operatorname{Supp}}

\def\Pic{\operatorname{Pic}}

\def\supp{\operatorname{supp}}

\def\Spec{\operatorname{Spec}}

\def\Br{{\mathrm{Br}}}
\def\id{{\mathrm{id}}}
\def\PP{{\mathbb P}}

\def\ZZ{{\mathbb Z}}
\def\CC{{\mathbb C}}
\def\RR{{\mathbb R}}
\def\QQ{{\mathbb Q}}
\def\NN{{\mathbb N}}

\def\G{{\mathbb G}}
\def\Ga{\G_{\rm a}}
\def\Gm{\G_{\rm m}}
\def\A{{\mathbb A}}
\def\AA{{\mathbb A}}

\def\K{{\mathbb K}}
\def\LL{{\mathbb L}}

\def\C{{\mathcal C}}

\def\embed{\hookrightarrow}

\def\Aff{\operatorname{Aff}}

\def\cO{{\mathcal O}}
\def\0{\circ}

\def\isom{\xlongrightarrow{\simeq}} 

\usepackage[final,multiuser,english]{fixme}

\fxsetup{layout=inline}

\newcommand{\slin}{\unitlength1mm\begin{picture}(0,0)
                       \put(0,-0.75){\line(0,-1){8.5}}
                      \end{picture}}
\newcommand{\nolin}{\unitlength1mm\begin{picture}(0,7)
                       \put(0.53,0.53){\line(1,1){8.94}}
                      \end{picture}}
\newcommand{\solin}{\unitlength1mm\begin{picture}(0,0)
                       \put(0.53,-0.53){\line(1,-1){8.94}}
                      \end{picture}}
\newcommand{\nwlin}{\unitlength1mm\begin{picture}(0,7)
                       \put(-0.53,0.53){\line(-1,1){8.94}}
                      \end{picture}}
\newcommand{\swlin}{\unitlength1mm\begin{picture}(0,0)
                       \put(-0.53,-0.53){\line(-1,-1){8.94}}
                      \end{picture}}
\newcommand{\vlin}[1]{\hspace{0.75mm}\unitlength1mm\begin{picture}(#1,0)
                       \put(0,0){\line(1,0){#1}}
                      \end{picture}\hspace{0.75mm}\rule[-3mm]{0mm}{4mm}}

\def\llin{\vlin{11.5}}
\newcommand{\lin}{\vlin{8.5}}

\newcommand{\co}[1]{\unitlength1mm\begin{picture}(0,8)
    \put(0,0){\circle{1.5}}
    \put(0,3){\makebox(0,5)[b]{$#1$}}
                      \end{picture}}
\newcommand{\mybox}{\unitlength1mm\begin{picture}(0,1.5)
    \put(-0.75,-0.75){\line(0,1){1.5}}
    \put(-0.75,-0.75){\line(1,0){1.5}}
    \put(0.75,0.75){\line(0,-1){1.5}}
    \put(0.75,0.75){\line(-1,0){1.5}}
    \end{picture}}
\newcommand{\boxo}[1]{\unitlength1mm\begin{picture}(0,8)
    \put(0,0){\mybox}
    \put(0,3){\makebox(0,5)[b]{$#1$}}
                      \end{picture}}

\newcommand{\cou}[2]{\unitlength1mm\begin{picture}(0,8)
    \put(0,0){\circle{1.5}}
    \put(0,3){\makebox(0,5)[b]{$#1$}}
    \put(0,-7){\makebox(0,4)[t]{$#2$}}
      \end{picture}
      \rule[-7mm]{0mm}{7mm}}

\newcommand{\cshiftup}[2]{\unitlength1mm\begin{picture}(0,9.25)
                       \put(0,10){\cou{#1}{#2}}
                      \end{picture}}

\newcommand{\cshiftdown}[2]{\unitlength1mm\begin{picture}(0,9.25)
                       \put(0,-10){\cou{#1}{#2}}
                      \end{picture}}

\begin{document}

\title[Automorphism groups of rigid surfaces]{Automorphism groups of rigid affine 
surfaces: the identity component}

\author[A. Perepechko]{Alexander Perepechko} 
\address{
HSE University, Faculty of Computer Science,\\
Pokrovsky blvd. 11, Moscow, 109028 Russia
} \email{a@perep.ru}

\author[M. Zaidenberg]{Mikhail Zaidenberg}
\address{Institut Fourier, UMR 5582, 
Laboratoire de Math\'ematiques,\newline\indent
Universit\'e Grenoble Alpes, CS 40700, 38058 
Grenoble cedex 9, France}
\email{mikhail.zaidenberg@gmail.com}

\thanks{
The research of A.~Perepechko was carried out 
at the HSE University at the expense 
of the Russian Science Foundation 
(project no. 22-41-02019).}
\thanks{
\mbox{\hspace{11pt}}
\thanks{\it Key words}:  affine
surface, automorphism group, algebraic group, 
ind-group, group action, 
birational transformation, transformation of a graph.}

\begin{abstract} It is known that the  
identity component
of the automorphism group of a projective 
algebraic variety
is an algebraic group. This is not true in general 
for quasi-projective varieties.
In this note
we address the question: given an affine 
algebraic surface $Y$, as to when
the  identity component $\Aut^\0 (Y)$
of the automorphism group $\Aut (Y)$  
is an algebraic group? We show that this 
occurs if and only if $Y$ admits 
no effective action of the additive group 
$\Ga$ of the field. In the latter case, 
$\Aut^\0 (Y)$ is an algebraic torus of rank $\le 2$.
\end{abstract}
\date{}

\maketitle

\tableofcontents

\section{Introduction}\label{sec:intro}
Let $\K$ be an algebraically closed field 
of characteristic zero
and $\Ga$ (resp. $\Gm$) be the additive 
(resp. the multiplicative) group of $\K$. 
If $Y$ is an affine variety over $\K$,
then the automorphism group $\Aut(Y)$ 
has a canonical structure of an ind-group 
equipped with its Zariski ind-topology, 
see e.g. \cite{FK}. In this paper we concentrate 
on the identity component $\Aut^{\0} (Y)$ of $\Aut(Y)$ 
with respect to the Zariski topology. 
Notice that, similarly to the case of a projective variety, 
the group of components $\Aut(Y)/\Aut^{\0} (Y)$ 
is countable, see \cite[Proposition 1.7.1]{FK}. 
We address the following conjecture.
We say that $Y$ is \textit{rigid} if  $Y$ admits 
no effective $\Ga$-action or, in other words, 
$\Aut^{\0} (Y)$ contains no $\Ga$-subgroup, 
i.e., no nontrivial unipotent element. 
\begin{conjecture}\label{conj}
Let $Y$ be an affine 
algebraic variety over $\K$.
If $Y$ is rigid, then 
the group $\Aut^{\0} (Y)$ 
is an algebraic torus of rank $\le\dim Y$. 
\end{conjecture} 
If Conjecture \ref{conj} is true, 
then the following one is true as well 
\footnote{We thank Hanspeter Kraft 
who suggested the second conjecture.}.
\begin{conjecture}\label{conj1} If an affine variety 
$Y$ admits no effective $\Ga$- and $\Gm$-actions, 
then $\Aut(Y)$ is a discrete group. 
\end{conjecture}
Some partial results on these conjectures can be found 
e.g. in \cite[Theorem 2.1]{AG}, \cite[Proposition 5]{Ii}, 
\cite{Je}, \cite{JL}, \cite[Theorems 4.4 and 4.7]{KPZ}, 
\cite[Theorem 1.3 and Section 7]{Kr}, and \cite{PR1}, \cite{PR2}. 
Conjecture \ref{conj} holds, for instance, 
for toric affine varieties and rational affine varieties 
with a torus action of complexity 1,
see \cite[Theorem 3]{Boldyrev-Gaifullin} 
and \cite[Theorem 6.4]{Borovik-Gaifullin}.
In the present paper
we establish Conjecture \ref{conj}  
in the case $\dim Y=2$.
\begin{theorem}\label{intro-fin-dim} Let $Y$ be 
a normal affine surface over $\K$. 
Then the following hold. 
\begin{enumerate}
\item[(a)]
$\Aut^{\0} (Y)$ is an algebraic group if and only 
if $Y$ admits no effective $\Ga$-action,  if and only 
if $\Aut^{\0} (Y)$ is an algebraic torus of rank $\le 2$.
\item[(b)] 
Let $(X,D)$ be a minimal completion of $Y$ 
by a normal crossing divisor. Then
$\Aut^{\0} (Y)$ is an algebraic group if and only if
every extremal linear segment of the weighted dual graph
$\Gamma(D)$ is admissible.
\end{enumerate}
\end{theorem}
See Definition \ref{def:dual} 
for the notion of a weighted dual graph.  
An \textit{extremal linear segment} 
of a weighted graph 
$\Gamma$ is a maximal linear subgraph 
that carries no branch point of $\Gamma$ and
contains a tip of $\Gamma$, see Definition 
\ref{def:segments}. 
It is \textit{admissible} if all its weights are $\le -2$, 
see Definition \ref{def:admissible}. 

As a corollary, we establish the following combinatorial 
criterion of rigidity, cf. Corollary \ref{cor:bir=inn}. 
\begin{corollary}\label{cor}
In the setup of Theorem \ref{intro-fin-dim} (b)
the surface $Y$ 
is rigid if and only if 
every extremal linear segment of 
$\Gamma(D)$ is admissible.
\end{corollary}

Using Corollary \ref{cor} we show 
that the affine surface 
$Y=\PP^2\setminus \supp D$, where $D$ is 
a reduced effective divisor 
on $\PP^2$ with only nodes as singularities, 
is rigid if and only if $\deg(D)\ge 3$, 
see Example \ref{exa:plane-curve}.

Statement (a) of 
Theorem \ref{intro-fin-dim} 
was announced in \cite[Proposition 4.7]{KPZ}. 
Our approach goes back to the work of V.~I.~Danilov 
and M.~H.~Gizatullin \cite{DG}. 
Namely, the automorphism group $\Aut (Y)$ 
can be realized as a group of 
birational transformations 
between the NC-completions of $Y$, i.e., 
the completions $(X,D)$ 
with a normal projective surface $X$ and 
a normal crossing boundary divisor $D$ 
contained in the smooth locus of $X$. 

The opposite of rigidity is the  property of
generic flexibility.
One says that $Y$ is \textit{generically flexible} 
if the subgroup 
$\SAut(Y)\subset\Aut(Y)$ generated by all the 
$\Ga$-subgroups
of $\Aut(Y)$ acts on $Y$ with an open orbit. 
By a celebrated Gizatullin theorem, see  \cite{Giz71b}, 
extended by Adrien Dubouloz
to normal affine surfaces, see \cite{Dub04}, 
a normal affine surface $Y$ 
non-isomorphic to 
$\A^1\times (\A^1\setminus \{0\})$
is generically flexible
if and only if the dual graph $\Gamma(D)$ of 
a minimal completion $(X,D)$ of $Y$ 
by a normal crossing divisor $D$ is linear.

In the proofs of our main results, 
our tools lie in
the birational geometry
of weighted graphs. 
Different aspects of this subject 
were developed e.g. in 
\cite{Dai1}, \cite{Dai2}, \cite{DG}, \cite{EN}, \cite{FKZ-graphs}, \cite{FZ},  
\cite{Fu}, \cite{Gi2}, \cite{Hi}, \cite{Ne}, \cite{OZ}, \cite{Ra1}, \cite{Ru}, etc.  
We consider, more generally, weighted graphs 
with a distinguished subset of marked vertices 
called \textit{rational}. 
In the case where our graph $\Gamma$ is 
the dual graph of a finite collection of 
curves on a surface, 
the rational vertices of $\Gamma$ correspond 
to rational curves. 

Section \ref{sec:NC-pairs} contains preliminaries 
on NC-completions and their birational transformations.
Section \ref{prelim-graphs}  deals with dual graphs and 
birational transformations of weighted graphs. 
For the sake of completeness we provide 
detailed proofs of some known results. 

In Section \ref{sec:graph-lemma}.2 we discuss 
birationally rigid weighted graphs. 
A weighted graph $\Gamma$ is \textit{minimal} 
if it has no rational vertex of degree $\le 2$ 
and of weight $-1$, and \textit{birationally rigid} 
if it is minimal and any birational transformation 
of weighted graphs
$\Gamma\dasharrow\Gamma'$  is an isomorphism, 
provided that $\Gamma'$ also
is minimal,
see Definition \ref{def:bir-rigid}.

In Section \ref{sec:actions} we show that 
the rigidity of an affine surface can be read 
from the weighted dual graph of a minimal 
NC-completion of this surface. 

The proof of Theorem \ref{intro-fin-dim} is done 
in Section \ref{sec:main}, see Theorem 
\ref{thm:Ga-inn}. In the final Secton \ref{app} 
we classify all connected weighted graphs 
with an infinite number of minimal models, 
see  Theorem \ref{prop:classification}.

Let us indicate the main ingredients 
in the proof of Theorem \ref{intro-fin-dim}.~
\begin{itemize} 
\item
It is well known that any birational transformation 
between smooth surfaces can be decomposed 
into a sequence of blowups and blowdowns. 
The theory of birational transformations of weighted 
graphs is parallel to that of birational transformations 
of NC-pairs, see Definition \ref{def:NC-pairs}.
\item
Namely,  to every birational transformation of an 
NC-pair $(X,D)$ there corresponds a
birational transformation of the dual graph 
$\Gamma(D)$ and vice versa, see 
Proposition \ref{def:repr} and Lemma 
\ref{lem:rel-min-pairs}.
\item 
The group $\Aut(Y)$ of automorphisms of a 
normal affine surface $Y$ contains a $\Ga$-subgroup 
if and only if $Y$ admits a minimal NC-completion 
$(X,D)$ whose dual graph $\Gamma(D)$ 
has a tip of weight zero, see Proposition 
\ref{affine-fibration} 
and its proof. 
\item
An extremal linear segment $S$
of a minimal weighted graph $\Gamma$ is
non-admissible 
if and only if $S$ is birationally equivalent 
to a linear graph 
with a tip of weight zero, see e.g. 
\cite[Examples 2.11]{FKZ-graphs}.  
\item 
These results imply that 
$Y=X\setminus D$ is rigid if and only if 
the dual graph $\Gamma(D)$ contains no non-admissible 
extremal linear segment, see Corollary \ref{cor}.
\item A blowup of an NC-pair $(X,D)$ is called 
\textit{inner} if its center is a node of $D$.  
A blowdown is \textit{inner} if 
it is the inverse of an inner blowup.
A birational transformation is \textit{inner} 
if it is composed of inner blowups, 
inner blowdowns and isomorphisms, see 
Definitions \ref{def:inner}--\ref{def:inner-1}. 
\item Let $\Bir(X,D)$ stand for the subgroup 
of $\Bir(X)$ of birational transformations biregular 
on $Y=X\setminus D$. 
Then the inner birational transformations 
of $(X,D)$ form a subgroup of $\Bir(X,D)$ 
denoted $\Inn(X,D)$.
\item
Similar definitions can be applied for 
birational transformations of weighted graphs, 
see Definitions \ref{def:blowup} and \ref{def:inner1}. 
To every inner birational transformation of 
an NC-pair $(X,D)$ there corresponds 
a unique inner
birational transformation of the dual graph 
$\Gamma(D)$ and vice versa, see 
Proposition \ref{def:repr}.  
\item If a minimal weighted graph $\Gamma$ 
has only admissible
extremal linear segments,
then any birational transformation 
of $\Gamma$ can be replaced by 
an inner birational transformation, 
see Proposition \ref{cor:graph}.3. 
\item
Let  $(X,D)$ be a minimal  NC-completion 
of a rigid affine surface $Y$. 
Then the identity component 
$\Aut^\0 (Y)=:\Bir^\0(X,D)$ 
of $\Aut (Y)$ coincides with 
the identity component
$\Inn^\0(X,D)$ of the group of
inner  birational transformations of the pair, 
see  Definitions 
\ref{def:inner}, \ref{def:inner-1} 
and Theorem \ref{th:bir=inn}. 
\item
In turn, the identity component $\Inn^\0 (X,D)$ 
coincides with
the affine algebraic 
subgroup $\Aut^{\0} (X,D)$, see Proposition 
\ref{prop:ind-sbgrp}. 
\item
The absence of $\Ga$-subgroups forces 
the connected affine algebraic group
$\Aut^{\0} (Y)=\Aut^{\0} (X,D)$ 
to be an algebraic torus. 
\end{itemize}

\smallskip

This proves Theorem \ref{intro-fin-dim}.
As a side result we  obtain
that for a rigid normal affine surface $Y$ 
the group $\Aut^{\0} (X,D)$ does not depend on the choice 
of a minimal NC-completion $(X,D)$ of $Y$. 

\medskip

 {\bf Acknowledgments.} We are grateful to Ivan Arzhantsev, 
 Hanspeter Kraft and Alvaro Liendo for some remarks 
 and indications in the literature. 
 Our thanks also go to an anonymous reviewer 
 for his or her insightful remarks.
 
\medskip


\section{NC-pairs}\label{sec:NC-pairs}

\subsection{Birational transformations of NC-pairs}
\label{sec:NC-pair}
\begin{definition}\label{def:NC-pairs}
Let $X$ be a  normal projective surface  over $\K$ 
and $D$ be a  reduced effective divisor on $X$. 
The pair $(X,D)$ is called an {\em $NC$-pair} if $D$ 
is a normal crossing divisor (i.e., the only singularities of 
$D$ are nodes) contained in the smooth locus of $X$.
An NC-pair $(X,D)$ is called an {\em SNC-pair} if $D$ 
is a simple normal crossing divisor, i.e., each component 
of $D$ is smooth and any two components of $D$ 
either are disjoint or intersect at a single point. 
An NC-pair $(X,D)$ is called \textit{minimal} 
if the contraction of a $(-1)$-curve on $X$ that is 
a component of $D$ does not lead to a new NC-pair.
\end{definition}
 Clearly, an NC-pair $(X,D)$ is not an SNC-pair 
 if and only if $D$ has either  a nodal component
 or a pair of components with more than one point 
 in common. As an example, one can consider the pair 
 $(\PP^2,C)$, where $C\subset\PP^2$ is a nodal cubic curve. 
 Any NC-pair is dominated by an SNC-pair 
 via a birational morphism. 
 For instance, in the example above such a domination 
 is obtained after blowing $\PP^2$  up in the node of $C$ 
 and performing another blowup in a node of the resulting curve, 
 cf.\ Example \ref{deformations-remark}. 
\begin{definition}\label{def:bir-NC-pairs}
Let $(X,D)$ and $(X',D')$ be NC-pairs. We denote by
\begin{enumerate}
    \item 
$\Bir((X,D),(X',D'))$ the set of 
birational transformations  
$X\dashrightarrow X'$ that restrict 
to biregular isomorphisms of the complements 
$X\setminus \supp D\isom X'\setminus \supp D'$;
  \item $\Bir(X,D)$ the group $\Bir((X,D),(X,D))$;
  \item $\Aut(X,D)$ the group of  biregular 
  automorphisms of $X$ that preserve $D$;
    \item\label{Aut-nat-class} $\Aut^\natural(X,D)$ 
    the subgroup of $\Aut(X,D)$ of automorphisms 
   preserving every component and  
   every node of $D$;
  \item $\Aut^{\0}(X,D)$ the identity component 
  of $\Aut(X,D)$.
\end{enumerate}
\end{definition}
\begin{remark} 
The group $\Bir(X,D)$ is naturally isomorphic 
to the automorphism 
group $\Aut(Y)$ of the complement 
$Y=X\setminus \supp D$.
Indeed, for $g\in\Bir(X,D)$  the restriction $g|_Y$ 
is an automorphism of $Y$, and any 
$f\in\Aut (Y)$ 
extends to a unique birational transformation 
$(X, D) \dasharrow (X,D)$. 

Thus, the group $\Bir(X,D)=\Aut (Y)$ 
does not depend  on the choice of 
an NC-completion $(X,D)$ 
of $Y=X\setminus \supp D$. 
However,  in general,
 the group $\Aut (X,D)$ depends on the chosen 
NC-completion. For instance, 
for the completion $\AA^2\embed \PP^2$  
of the affine plane we have 
$\Aut (\PP^{2},\PP^{1})\simeq\Aff(\AA^2)$, whereas 
for the completion $\AA^2\embed
\PP^1\times\PP^1$ we have
  \[\Aut(\PP^1\times\PP^1,\PP^{1}\times \{\infty\} \cup 
  \{\infty\} \times \PP^{1})\cong
  (\Aff(\AA^{1}))^2\rtimes  \ZZ_2.\]
 \end{remark}
\subsection{Inner transformations}
\begin{definition}\label{def:inner}
 A blowup of an NC-pair $(X,D)$ with center $x\in \supp D$ 
 gives again an NC-pair. Both pairs provide
 NC-completions of $Y=X\setminus\supp D$.
 Such a blowup is called \textit{inner} 
if $x$ is a node of $D$ and \textit{outer} otherwise. 
We say that a blowdown of a component of $D$, 
producing a new NC-pair, is \textit{inner} 
(resp. \textit{outer}) if the inverse blowup 
is inner (resp. outer). 
\end{definition}
\begin{definition}\label{def:inner-1}
A birational transformation 
$(X,D)\dashrightarrow(X',D')$ between two NC-pairs 
is called \textit{inner} 
if it can be decomposed into a sequence of inner blowups, 
inner blowdowns and isomorphisms of NC-pairs.
The subset of inner birational transformations 
forms a subgroup $\Inn(X,D)\subset\Bir (X,D)=\Aut(Y)$.
\end{definition}
\begin{lemma}\label{lem:inclusions}
We have \[\Aut^{\0}(X,D)\subset\Aut^\natural(X,D)
\subset\Aut(X,D)\subset\Inn(X,D)\subset\Bir(X,D) =\Aut(Y).\]
\end{lemma}
\begin{proof}
This is immediate from our definitions.
\end{proof}
\begin{lemma}\label{lem:connected-natural}
$\Aut^{\0}(X,D)=(\Aut^\natural(X,D))^{\0}$ 
is a (connected) algebraic group.
\end{lemma}
\begin{proof}  It is well known that $\Aut^{\0}(X)$ 
is a connected algebraic group, see e.g. \cite{MO, Ma, Ra}, 
and so is its closed connected subgroup $\Aut^{\0}(X,D)$.  
\end{proof}
\begin{lemma}\label{lem:conjug}
Every $\phi\in\Inn(X,D)$ can be written as 
\begin{equation}\label{eq:isomorphism} \phi
=\alpha\sigma_n^{\pm 1}\cdots\sigma_1^{\pm 1}
={\sigma'}_n^{\pm 1}\cdots{\sigma'}_1^{\pm 1}\beta
\end{equation}
where $\alpha, \beta\in\Aut (X,D)$ and the 
$\sigma_i, \sigma_i'$ are inner blowdowns. 
\end{lemma}
\begin{proof}
Consider a commutative diagram
\[
\begin{tikzpicture}
	\node (empty) at (-3,0) {};
  \node  (INN) at (0,2) {$(X_1,D_1)$};
  \node  (Inn) at (0,0)  {$(X'_1,D'_1)$};
  \node (iNN)  at (4,2)  {$ (X_2,D_2)$};
  \node  (inn) at (4,0)  {$(X'_2,D'_2)$};
  \draw[->, dashed] (INN) to node[above] {$\sigma^{\pm 1}$}  (iNN);
  \draw[->] (INN) to node[left] {$\alpha_1$} (Inn);
  \draw[->] (INN) to node[right] {$\simeq$} (Inn);
  \draw[->, dashed] (Inn) to node[above] {${\sigma'}^{\pm 1}$}  (inn);
  \draw[->] (iNN) to node[right] {$\alpha_2$}   (inn);
  \draw[->] (iNN) to node[left] {$\simeq$}   (inn);
\end{tikzpicture}
\]
where $\sigma$ and $\sigma'$ are inner blowdowns. 
Given a pair $(\sigma',\alpha_1)$ there exists 
a unique pair $(\sigma,\alpha_2)$ 
fitting in the diagram, and vice versa. 
Now the claim follows. 
\end{proof}
\begin{lemma}\label{diez-comm}
Every inner birational transformation 
$\phi\colon(X,D)\dashrightarrow(X',D')$ 
of NC-pairs 
induces an isomorphism
\[\phi_*\colon\Aut^\natural (X,D)
\stackrel{\simeq}{\longrightarrow} \Aut^\natural (X',D'),\quad
 \alpha\mapsto \phi\circ\alpha\circ\phi^{-1}.\] 

In particular, $\Aut^\natural (X,D)$ is a normal 
subgroup of $\Inn(X,D)$.
\end{lemma}
\begin{proof} The statement is definitely 
true for inner blowups and blowdowns. 
Hence, it holds in general. 
\end{proof}
\begin{remark}
The group $\Aut^\natural(X,D)$ is the kernel 
of the  natural homomorphism 
$\Aut(X,D)\to\Aut(\Gamma(D))$, where $\Gamma(D)$ 
is the dual graph of $D$, see Definition \ref{def:dual}.
\end{remark}
\begin{proposition}\label{prop:countable} 
The group $\Inn(X,D)/\Aut^\natural (X,D)$ is countable. 
\end{proposition}
\begin{proof} Since $\Aut (X,D)/\Aut^\natural (X,D)$ 
is a finite group, it suffices to show that $\,$
the collection $\Inn(X,D)/\Aut (X,D)$  
of left cosets is countable. 

By Lemma \ref{lem:conjug} any left coset in $\Inn(X,D)/\Aut (X,D)$ 
has the form $\psi\cdot\Aut (X,D)$ for some 
$\psi={\sigma'}_n^{\pm 1}
\cdots{\sigma'}_1^{\pm 1}\in\Inn(X,D)$. 
These cosets form a countable set. 
\end{proof}
\subsection{The ind-structure on $\Bir(X,D)$}
\label{ss:degree-function}
Let an NC-pair $(X,D)$  be  a completion of 
a normal affine surface
$Y=X\setminus\supp D$. 
By Goodman's criterion of affiness for surfaces
\cite[Theorem~2]{Go69} 
we can find an ample divisor $H\subset X$ such that 
$\supp H = \supp D$. This defines 
an ample polarization on $X$.
Replacing $H$ by its suitable 
high multiple, the very ample 
linear system $|H|$ defines an 
embedding $\phi_{|H|}\colon X\hookrightarrow\PP^n$ 
that sends the divisor $H$ 
to a hyperplane section of $X$. 
The restriction of $\phi_{|H|}$ to $Y$ 
yields a proper embedding
$Y=X\setminus \supp H\hookrightarrow\AA^n=
\PP^n\setminus \PP^{n-1}$. 
\begin{definition}[{\rm cf. e.g. \cite[(1.1)]{Cantat-Xie}}]
For $\Phi\in \Bir(X,D)$ we define its degree as 
$\deg\Phi=H\cdot \Phi^*(H)$. 
For $\Phi\in\Aut(Y)$ we understand $\deg\Phi$ 
as the degree of  the natural extension $\hat\Phi\in\Bir(X,D)$. 
\end{definition}
\begin{remarks} 1. 
The degree function $\deg\Phi$ on $\Bir(X,D)$
coincides 
with the usual degree of $\Phi|_Y\in\Aut(Y)$ viewed as 
an automorphism of the closed subvariety
$Y\subset\AA^n$.
Indeed, let $\Phi$ be written in homogeneous coordinates 
on $\PP^n$ as 
\[\Phi=(p_0\colon\ldots\colon p_n)|_X\] 
where the $p_i$ are  homogeneous polynomials in $n+1$ 
variables of the same degree $d$, 
that do not belong simultaneously to the homogeneous 
ideal of $X$. 
Then $\deg\Phi=\min\{d\}$ 
where the minimum is taken over all such presentations. 

2. We have $\Bir(X,D)=\varinjlim \Bir(X,D)_{\le d}$ where
\[\Bir(X,D)_{\le d}=\big\{\Phi\in \Bir(X,D)\,|\,
\max\{\deg\Phi,\,\deg\Phi^{-1}\}\le d\big\}\] 
is an affine algebraic variety and 
$\Bir(X,D)_{\le d}\subset\Bir(X,D)_{\le d+1}$ 
is a closed embedding.
This yields an affine ind-structure on 
$\Bir(X,D)=\Aut(Y)$. 
It is easily seen that this ind-structure coincides 
with the usual ind-structure on $\Aut(Y)$ defined via
the induced closed embedding $Y\hookrightarrow\A^n$, 
see e.g. \cite[Section 5.1]{FK} or \cite[Section 2.1]{KPZ}. 
On the other hand, in the case of 
 a rational surface $X$ the group $\Bir(X)$ 
carries no natural ind-structure, see \cite{BF13}.

3. The very ample linear system $|H|$ 
is stable under 
the action of $\Aut^\natural(X,D)$ on $(X,D)$.
Hence the embedding
\[\phi_{|H|}\colon (X, \supp D)=
(X, \supp H)\hookrightarrow (\PP^n, \PP^{n-1})\]
induces an embedding of
$\Aut^\natural(X,D)$ 
onto a closed subgroup of  
$\Aut(\PP^n, \PP^{n-1})$.
Therefore, $\Aut^\natural(X,D)$ 
is an affine algebraic group.
\end{remarks}
Recall the following definitions; see e.g. \cite{FK}. 
One says that 
a subgroup $G\subset \Bir(X,D)$ is connected 
in the Zariski ind-topology if every
element $g\in G$ belongs to a connected algebraic 
subset of $G$ 
that contains the identity. 
The identity component $G^\0$ of a  subgroup 
$G\subset \Bir(X,D)$ is the maximal connected 
subgroup of $G$.
\begin{proposition} \label{prop:ind-sbgrp}
We have $\Inn^\circ(X,D)=\Aut^\circ(X,D)$. 
In particular, $\Inn^\circ(X,D)$ is 
a connected affine algebraic group.
\end{proposition}
\begin{proof} Assume first that 
the ground field $\K$ is uncountable. 
Recall that  $\Aut^\natural(X,D)$ is 
a normal affine algebraic subgroup 
of $\Inn(X,D)$ of countable index closed  
in $\Bir(X,D)$, see Lemma \ref{diez-comm} 
and Proposition \ref{prop:countable}. 
Therefore, $\Inn(X,D)$
is a countable union
of mutually disjoint closed algebraic subvarieties 
of $\Bir(X,D)$ of the form 
$h_k\circ\Aut^\natural(X,D)$ 
with $h_k\in\Inn(X,D)$ for $k\in\NN$. 
Let $A$ be an irreducible algebraic subvariety  of 
$\Bir(X,D)=\Aut(Y)$ of positive dimension contained in 
$\Inn(X,D)$. Since the set of $\K$-points of $A$ 
is uncountable, 
$A$ is actually contained in 
one of the left cosets $h_k\circ\Aut^\natural(X,D)$, 
see \cite[Lemma~1.3.1]{FK}. 
If $\id_{(X,D)}\in A$ then $A\subset 
(\Aut^\natural(X,D))^\circ=\Aut^\circ(X,D)$. Now
the equality $\Inn^\circ(X,D)=\Aut^\circ(X,D)$ follows. 

Suppose now that $\K$ is countable. 
We embed $\K$ in 
the uncountable algebraically closed field $\LL$; for instance, 
one can take for $\LL$ the algebraic closure of the function field 
$\K(x_t\,|\,t\in\RR)$. Let ${\rm Gal}={\rm Gal}(\LL/\K)$ 
stand for the Galois group 
of the field extension 
$\K\subset\LL$. Then $\K=\LL^{\rm Gal}$. Moreover, 
for any algebraic $\K$-variety $Z$ we have 
$Z=Z(\LL)^{\rm Gal}$ and $Z(\LL)$ is irreducible 
if and only if $Z$ is, 
see \cite[Sec.~1.11]{FK} and 
\cite[\href{https://stacks.math.columbia.edu/tag/037K}{Lemma~037K}]
{SP-CA}. 
In particular,
$(X,D)=(X(\LL),D(\LL))^{\rm Gal}$. Furthermore,
\[\Aut^\natural(X(\LL),D(\LL))^{\rm Gal}=
(\Aut^\natural(X,D)(\LL))^{\rm Gal}=\Aut^\natural(X,D),\]
and likewise for $\Aut^\circ(X,D)=(\Aut^\natural(X,D))^\circ$.
Since $\Inn(X,D)=\bigcup_k h_k\circ\Aut^\natural(X,D)$
we have $\Inn(X(\LL),D(\LL))^{\rm Gal}=\Inn(X,D)$.
 
 Applying the first part of the proof we get
 $\Inn^\circ(X(\LL),D(\LL))=\Aut^\circ(X(\LL),D(\LL))$.
Passing to  the ${\rm Gal}$-fixed point loci in the latter equality 
we obtain $\Inn^\circ(X,D)=\Aut^\circ(X,D)$. 
\end{proof}
\section{Birational geometry 
of weighted graphs}\label{prelim-graphs} 
The birational geometry of  NC-pairs has its 
combinatorial counterpart, namely, the 
birational geometry of weighted graphs, 
see the references in the Introduction. 
\subsection{Weighted graphs}
\begin{definition}\label{def:rat-vert}
Let $\Gamma$ be a (nonempty)  finite 
weighted graph with the set of vertices 
$\Vert(\Gamma)$ 
and the set of edges $\Reb(\Gamma)$, 
where vertices are weighted by integer numbers, 
and let $\Ratio(\Gamma)\subseteq\Vert(\Gamma)$ 
be a distinguished subset of vertices  
which we call \textit{rational vertices}. 
A rational vertex of weight $-1$ 
(resp. of weight $0$) is called 
($-1$)\textit{-vertex} 
(resp. ($0$)\textit{-vertex}). 
The graph $\Gamma$ may be disconnected 
and may contain loops,
cycles and multiple edges, that is, several 
edges joining the same pair of vertices.
\end{definition}
\begin{nota} Following \cite{FKZ-graphs}
we let $[[a_1,\ldots,a_n]]$ be 
a linear weighted graph
 with $n$ ordered
vertices of weights $a_1,\ldots,a_n$.
We also let $((a_1,\ldots,a_n))$ be 
a circular weighted graph 
with $n$ cyclically ordered 
vertices of weights $a_1,\ldots,a_n$. 
\end{nota}
\begin{definition}\label{def:vertices} $\,$
\begin{itemize}
\item
Given a weighted graph $\Gamma$, 
the \textit{degree} 
$\deg_{\Gamma}(v)$ of a vertex 
$v\in\Vert(\Gamma)$ is the
number of incident edges of $\Gamma$ at $v$, 
where every loop $[v,v]$ is counted
twice. We simply write $\deg(v)$ 
when the graph 
$\Gamma$ has been fixed.
\item A rational vertex $v$ of $\Gamma$ 
is called a \textit{tip} or an
\textit{end vertex} if $\deg(v)=1$ 
and
\textit{at most linear vertex} if $\deg(v)\le 2$ 
and no loop of $\Gamma$ is incident with $v$. 
We let $\Tip(\Gamma)$ be the set 
of tips of $\Gamma$.
\item A vertex $v$ is called 
a \textit{branch vertex}
if either $v$ is non-rational or $\deg(v)\ge 3$.
We let $\Br(\Gamma)$ be the set of 
all branch vertices of $\Gamma$. 
Thus, $\Vert(\Gamma)\setminus\Ratio(\Gamma)
\subseteq\Br(\Gamma)$.
\item
For a proper subset $V$ of $\Vert(\Gamma)$ we 
let $\Gamma\ominus V$ 
be the subgraph of $\Gamma$ obtained by deleting 
$V$ and all its incident edges, 
including the incident loops. 
\item Given a vertex $v\in \Gamma$, 
the connected components of
$\Gamma\ominus\{v\}$ linked to $v$ 
are the \textit{branches}
of $\Gamma$ at $v$. A branch linked to $v$ 
via exactly one edge
is called \textit{simple}. 
\end{itemize}
\end{definition}
\begin{definition}\label{def:segments}$\,$
\begin{itemize}
\item In the case where $V= \Br(\Gamma)$ is the set of 
branch vertices of $\Gamma$,
the connected
components of $\Gamma\ominus \Br(\Gamma)$ 
are called
\textit{segments}. 
A connected graph 
with no branch vertex coincides with its only segment. 

\item A segment can be either linear or circular. 
A circular segment 
is a cycle in $\Gamma$ which contains
no branch vertex; it is a connected component 
of $\Gamma$. 
A linear segment can consist of 
a single vertex. For example, this is 
the case for a vertex of degree 2 linked 
to a branch point 
by two edges.
\item A linear segment $L$ is called
\textit{extremal} if $L$ contains one or two 
tips of
$\Gamma$ and
 \textit{inner} if $L$ contains no  
  tip of $\Gamma$. 
\end{itemize}
\end{definition}
\begin{definition}\label{def:intersection-form}  
Let $\Gamma$ be a weighted graph  
with vertices $v_1,\ldots,v_n$. 
The \textit{intersection form} $I(\Gamma)$ 
is a bilinear form on $\RR^n$ given by 
the symmetric square matrix $A(\Gamma)=(a_{i,j})$, 
where $a_{i,i}$ is the weight of $v_i$ and for 
$j\neq i$, $a_{i,j}$ equals the number of 
edges joining $v_i$ and $v_j$. 
The \textit{discriminant} ${\rm discr}(\Gamma)$
is the determinant $\det(-A(\Gamma))$. 
We let $i_+(\Gamma), i_-(\Gamma)$ and 
$i_0(\Gamma)$ be the inertia indices of 
the associate quadratic form $I(\Gamma)$. 
\end{definition}
\begin{definition}\label{def:blowup}
  Let $[v_1,v_2]$ (where possibly $v_1=v_2$) 
  be an edge of 
  $\Gamma$.
 The \textit{inner blowup} in $[v_1,v_2]$ consists 
 in adding a new ($-1$)-vertex $v$,
 replacing the edge  $[v_1,v_2]$ by two new edges 
 $[v_1,v]$ and $[v,v_2]$ and
 decreasing the weights of $v_1$ and $v_2$  
 by $1$ if $v_1\neq v_2$ and by $4$ if 
 $v_1=v_2$ (i.e. $[v_1,v_2]$ is a loop).

The \textit{outer blowup} at 
a vertex $v$ in $\Gamma$ consists in 
introducing a new ($-1$)-vertex
$v^\prime$ along with a new edge 
$[v,v^\prime]$ and decreasing
the weight of $v$ by $1$.

The modifications inverse to inner and 
outer blowups are called
\textit{inner} and \textit{outer blowdowns}, respectively.
\end{definition}
\begin{remark}
The definition of inner birational transformations  
of weighted graphs  in 
\cite[Definitions 2.3 and 2.8]{FKZ-graphs} 
has a broader meaning. Indeed, 
the inverse of an outer blowup is
considered in {\em loc.cit.} to be inner. 
\end{remark}
\begin{proposition}[{\rm cf. \cite[Proposition 1.1]{Ne}, 
\cite[[Proposition 1.14]{Ru}, \cite[Lemma 4.6]{FKZ-graphs}}] 
\label{lem:inertia}
A blowup of a weighted graph $\Gamma$ adds $1$ to 
$i_-(\Gamma)$, while $i_+(\Gamma)$ and 
$i_0(\Gamma)$ remain unchanged. 
\end{proposition}
\begin{proof} 
Letting $\Vert(\Gamma)=\{v_1,\ldots,v_n\}$ 
consider the $\RR$-vector space $V=\bigoplus_i \RR v_i$  
with the inner product given by 
the intersection form $I(\Gamma)$. 
The symmetric matrix $M_n$ of $I(\Gamma)$ is diagonal 
in a suitable orthogonal  basis $\{e_1,\ldots,e_n\}$ of $V$.

A blowup $\delta\colon\Gamma'\dashrightarrow\Gamma$ 
creating a vertex  
$v_{n+1}$ induces a map $\delta^*\colon V\embed V'
=\bigoplus_{i=1}^{n+1} \RR v_i$ defined by 
\[\delta^*\colon v\mapsto v+I(\Gamma')(v, v_{n+1})v_{n+1}.\] 
The vector $v_{n+1}$ is an eigenvector of the matrix 
$M_{n+1}$ of $I(\Gamma')$  with eigenvalue $-1$. 
Indeed, one can check that $I(\Gamma')$ restricted on 
$\delta^*(V)\subset V'$ 
coincides with $\delta^*I(\Gamma)$ 
and $v_{n+1}$ is orthogonal to $\delta^*(V)$. 
The orthogonal basis in $V'$ of the eigenvectors 
\[\delta^*(e_1),\ldots,\delta^*(e_n),v_{n+1}\] 
is diagonalizing for $M_{n+1}$, which proves the claim. 
\end{proof} 
\begin{remark}
In the case where $\Gamma=\Gamma(D)$ 
is the dual graph of an NC-divisor $D$ on 
a smooth surface $X$ and $\delta\colon X'\to X$ 
is the inverse of the blowup 
of a point on $D$, the homomorphism $\delta^*$ 
sends  an $\RR$-divisor $N$ supported on $\supp D$ 
to its total transform $\delta^*(N)$ on $X'$, and 
$v_{n+1}$ represents the exceptional $(-1)$-curve of $\delta$. 
The properties of $\delta^*$ 
used in the proof follow in this case from the Projection Formula. 
\end{remark}
\begin{definition}\label{def:inner1}
A {\em birational  transformation of weighted graphs} 
$\phi:\Gamma\dashrightarrow \Gamma'$ 
is a finite sequence of blowups, blowdowns 
and isomorphisms. 
 If all these blowups and blowdowns are inner, 
 then $\phi$ is called {\em inner}. 
 Two weighted graphs are called 
 \textit{birationally equivalent} if one can 
 be transformed into the other by means of 
 a birational transformation. 
 \end{definition}
 From Proposition \ref{lem:inertia} 
 we deduce the following 
 \begin{corollary}\label{cor:inertia-indices}
 The inertia indices $i_+(\Gamma)$ and 
 $ i_0(\Gamma)$ are birational invariants. 
 \end{corollary}
\begin{remarks} \label{rems} 1.
Consider a weighted graph $\Gamma$ 
and a sequence $\phi$ of blowups, blowdowns 
and isomorphisms resulting in a final graph $\Gamma'$.
If, besides isomorphisms, $\phi$ includes also 
blowups or blowdowns, then one can ignore
 the isomorphisms that participate in $\phi$ 
 by simply renaming the created graphs, including 
 the final one. Anyway, one can reduce $\phi$ to a sequence 
 of blowups and blowdowns preceded or followed by 
 a single isomorphism,
see  formula \eqref{eq:isomorphism} in the proof of Proposition
\ref{prop:countable}. 

2. A birational transformation of weighted graphs 
$\phi\colon\Gamma\dasharrow\Gamma'$ 
is not a mapping of graphs: 
in general, it is not well defined neither 
on the sets of vertices nor on the sets of edges. 
Nonetheless, given a second birational transformation 
$\phi'\colon\Gamma'\dasharrow\Gamma''$ the composition 
$\phi'\circ\phi\colon \Gamma\dasharrow\Gamma''$ 
is a well defined birational transformation of weighted graphs. 
Also the identity ${\rm id}\colon \Gamma\to\Gamma$ 
and the inverse birational transformation 
$\phi^{-1}\colon\Gamma'\dasharrow\Gamma$ are  well defined. 
\end{remarks}
\subsection{Contractible graphs}
 \begin{definition}\label{def:morphism} A birational transformation 
 $\phi:\Gamma\dasharrow \Gamma'$ is 
 a {\em birational morphism} (or {\em a domination}) 
 if no blowup participates in $\phi$.
In this case we write $\phi:\Gamma\to \Gamma'$.  
If, moreover, $\phi$ is composed only of blowdowns 
then we say that $\phi$ is a \textit{contraction}. 
So,  
contractions are birational morphisms, but an isomorphism 
of graphs is not a contraction.
 \end{definition}
\begin{remark}\label{rem:composition}
The composition of birational morphisms (resp., contractions) 
of weighted graphs is a birational morphism (resp., a contraction). 
\end{remark}
\begin{definition}
A weighted graph $\Gamma$ is called \textit{minimal}
if any birational morphism $\Gamma\to\Gamma'$ 
is in fact an isomorphism. A graph $\Gamma$  
is minimal if and only if it has no
at most linear $(-1)$-vertex. Clearly, any weighted 
graph $\Gamma$ dominates a minimal one called 
a \textit{minimal model} of $\Gamma$.  
\end{definition}
A simple example of a minimal weighted graph is 
a graph consisting of a single vertex $v$ and a loop, 
cf. Example \ref{deformations-remark}. 
Indeed, an at most linear $(-1)$-vertex has 
no incident loop, see Definition \ref{def:vertices}. 
\begin{definition} \label{def:contractible}
 A connected weighted graph $\Gamma$
 with at least two vertices is said 
 to be {\em contractible} if there exists a contraction of
$\Gamma$
to a graph  $\Gamma'$ that consists of 
a single $(-1)$-vertex with 
no incident loop; the latter graph $\Gamma'$
also is considered to be contractible.
\end{definition}
The following lemma is well known. 
\begin{lemma}\label{lem:contractibility}  
Let $\Gamma$ be a weighted graph.
\begin{enumerate}
\item[(a)] If $\Gamma$ is contractible, 
then $\Gamma$ is a tree with only rational vertices. 
Furthermore, the intersection form $I(\Gamma)$ 
is negative definite of discriminant 1. 
\item[(b)] Conversely, if $\Gamma$ is a tree with 
only rational vertices and negative definite  
intersection form $I(\Gamma)$ of discriminant 1, 
then  $\Gamma$ is contractible.
\item[(c)] Contracting an at most linear $(-1)$-vertex 
of a contractible
graph  with at least two vertices yields a contractible 
graph and so does a blowup of a contractible graph. 
Thus, the contractibility is a birational invariant. 
\end{enumerate}
\end{lemma}
\begin{proof} To show the first assertion in (a) 
it suffices to look at the reconstruction process 
of blowups 
starting from the graph which consists of 
a single isolated $(-1)$-vertex. 
The second assertion follows from \cite[Section 3, 
Proposition and Theorem]{Hi}. 
We address \cite[Proposition 1.20]{Ru} 
and \cite[Remark 4.8, formula (17)]{FKZ-graphs} 
for statement (b). 
 Statement (c) follows e.g. from (b) due to 
 \cite[Remark 4.8, formula (17)]{FKZ-graphs}; 
 cf. also 
 \cite[Section 3, Proposition]{Hi}.
\end{proof}
The next corollary is immediate.
\begin{corollary}\label{cor:contractibility}
For a contractible weighted graph $\Gamma$  
the following hold.
\begin{enumerate}
\item[(a)] The weight of every vertex $v$ 
of $\Gamma$ is negative.
\item[(b)]  No two $(-1)$-vertices 
of $\Gamma$ are neighbors.
\item[(c)] Every $(-1)$-vertex $v$ of $\Gamma$ 
is at most linear. 
\item[(d)] If $\Gamma$ contains at least two vertices, 
then every 
$(-1)$-vertex of $\Gamma$ is contracted in 
any process of contraction 
of $\Gamma$ to a graph that consists of 
a single $(-1)$-vertex 
with no incident loop. 
\end{enumerate}
\end{corollary}
\subsection{Contraction of a subgraph} 
\label{ss:subgraph}
We use below the following notions.
Let $\Gamma$ be a weighted graph and $\Delta$ 
be a proper \textit{induced subgraph} of $\Gamma$.
The latter means that every edge of 
$\Gamma$ linking two vertices of $\Delta$  
is an edge of $\Delta$.
 \begin{definition}\label{def:contr-sbgrph}
 Let  $\Delta$  be an induced subgraph of $\Gamma$
 and $p\colon\Gamma\to\Gamma'$ be a contraction.
 We say that  $p$
 \textit{contracts} 
$\Delta$
if every vertex of $\Delta$ is contracted under $p$. 
We say that $p$ is a \textit{contraction of} 
$\Delta$ if $p$ contracts every vertex of $\Delta$ 
and no other vertices.
\end{definition}
\begin{remarks}\label{rem:induced-contr}
1.
The fact that $p$ contracts 
$\Delta$ does not imply, 
in general, that $\Delta$ is contractible, 
see e.g. Example \ref{exa:nontip}. On the other hand, 
 if $\Delta$ is contractible, this does not imply, 
 in general, 
 that there exists a contraction 
 $p\colon\Gamma\to\Gamma'$  
 that contracts $\Delta$. 
 For instance, a $(-1)$-vertex $v$ of 
 a minimal graph $\Gamma$ 
survives every contraction $\Gamma\to\Gamma'$.
Indeed, a minimal graph 
admits no contraction, and so $v$ cannot be 
an at most linear vertex.

2. It follows by recursion from Lemma 
\ref{lem:contractibility}(c)
that a contraction of a proper induced subgraph 
of a contractible graph results in a contractible graph. 
\end{remarks}
  \begin{definition}\label{def:image} 
    Consider the blowdown 
    $\delta\colon\Gamma\to\Gamma'$ of
    an at most linear
   vertex $v_0\in\Vert(\Gamma)$. 
    The image $\delta(\Gamma\ominus \{v_0\})$ in 
    $\Gamma'$ is a well defined subgraph isomorphic to 
    $\Gamma\ominus \{v_0\}$
    up to a change of weights. 
    Likewise, there is a well defined image 
    of any subgraph $\Delta$ of $\Gamma\ominus \{v_0\}$.
    For a connected subgraph $\Delta$ of $\Gamma$ such that 
    $v_0\in\Vert(\Delta)$ and $\Vert(\Delta)\neq \{v_0\}$, 
    the image of $\Delta$ in 
    $\Gamma'$ coincides with $\delta(\Delta\ominus \{v_0\})$
     if $v_0$ is a tip of $\Delta$;
    otherwise $\delta(\Delta)=\delta(\Delta\ominus 
    \{v_0\})\cup [\delta(u),\delta(v)]$ 
    where $u$ and $v$ are the vertices of $\Delta$ 
    linked to $v_0$ in $\Delta$
    (it is possible that $u=v$). 
     
     We extend by recursion the notion of image 
     of a subgraph to 
     any birational morphism $\phi\colon\Gamma\to \Gamma'$. 
  \end{definition}

  \begin{remark}
    If $\delta$ is an inner blowdown, 
 then $v_0$ has  in $\Gamma$ degree 2
and exactly two incident edges, say, 
$[u,v_0]$ and $[v_0,v]$
that are replaced by a single edge 
$e'=[\delta(u),\delta(v)]$ of $\Gamma'$,
see Definition \ref{def:blowup}.
If $\delta$ is outer, then $v_0$ 
is a tip of $\Gamma$, and 
 the unique edge 
$[v_0,v]$ of $\Gamma$  at $v_0$ 
disappears under $\delta$. 
  \end{remark}
For a birational morphism
    $\delta\colon\Gamma\to\Gamma'$ and 
    an induced subgraph 
    $\Delta' \subset \Gamma'$ we define 
    the partial preimage
    $\delta^{-1}(\Delta')\subset \Gamma$ 
    and the total preimage
    $\delta^*(\Delta')\subset\Gamma$ as follows.
\begin{definition}\label{def:preimage} 
Let $\Delta' \subset \Gamma'$ 
be an induced subgraph and 
$\delta\colon\Gamma\to\Gamma'$ 
be the blowdown of an at most linear 
vertex $v_0\in\Vert(\Gamma)$. 
If $\delta$ 
is inner (resp. outer) and $v_0$ is linked
to vertices $v_1$ and $v_2$ of $\Gamma$ 
(resp. to a vertex $v_1\in\Vert(\Gamma)$), 
then we define the 
\textit{partial preimage}
\[\delta^{-1}(\Delta'):=
\delta_0^{-1}(\Delta'\setminus[v_1,v_2]) 
\quad (\text{resp.  } \delta^{-1}(\Delta'):=
\delta_0^{-1}(\Delta'))\] 
where
\[\delta_0\colon \Gamma\ominus\{v_0\}\stackrel{\simeq}
{\longrightarrow}\Gamma'\setminus[v_1,v_2] 
\quad (\text{resp.  } 
\delta_0\colon\Gamma\ominus\{v_0\}\stackrel{\simeq}
{\longrightarrow}\Gamma')\]
stands for the induced isomorphism of 
unweighted graphs, cf. Definition \ref{def:image}. 

If $\delta$ is inner (resp. outer) then 
we define the \textit{total preimage}
\[\delta^*(\Delta'):=\begin{cases} 
\delta^{-1}(\Delta')\cup
[v_1,v_0]\cup\{v_0\}\cup [v_0,v_2] & \text{if }
[v_1,v_2]\in\Reb(\Delta'),\\
\delta^{-1}(\Delta')\cup\{v_0\}\cup
[v_1,v_0] & \text{if }   v_1\in \Vert(\Delta') \text{ and }
v_2\notin \Vert(\Delta'),\\
\delta^{-1}(\Delta') & \text{if } 
v_1,v_2\notin \Vert(\Delta'),
\end{cases}\]
resp.
\[\delta^*(\Delta'):=\begin{cases} 
\delta^{-1}(\Delta')\cup [v_1,v_0]\cup\{v_0\}  & \text{if }   
v_1\in\Vert(\Delta'),\\
\delta^{-1}(\Delta') & \text{if }   
v_1\notin\Vert(\Delta').\end{cases}\] 

We extend by recursion the notions of partial and 
total preimages of an induced subgraph
to any birational morphism 
$\phi\colon\Gamma\to\Gamma'$.
\end{definition}
\begin{remark}\label{rem:preimage}
For a vertex $v\in \Vert(\Gamma')$ 
its partial preimage $\delta^{-1}(v)$ 
is a vertex of $\Gamma$ 
that is often denoted by the same letter $v$.
Given an induced subgraph  $\Delta' \subset \Gamma'$, 
the subgraph $\delta^{-1}(\Delta')$ 
is not necessarily induced,
while $\delta^*(\Delta')$ is. So, $\delta^*(\Delta')$ 
is uniquely defined by the set of its vertices, 
that is, by the subset 
$\Vert(\delta^{-1}(\Delta'))\subset \Vert(\Gamma)$ 
plus the vertex $v_0$ in certain cases. 

If $\Delta'$ is connected, then also 
$\delta^*(\Delta')$ is. Moreover, $\delta^*(\Delta')$ 
is a connected component of the subgraph 
$\Gamma\ominus\bigcup_{v\in \Vert(\Gamma')
\setminus\Vert(\Delta')} \{\delta^{-1}(v)\}$. 
\end{remark}
 We will usually apply Definitions~\ref{def:image} 
 and \ref{def:preimage} 
 to induced subgraphs of $\Gamma$.
\begin{sit}\label{convention} 
\textbf{Convention.}  
Given a birational morphism 
$\phi\colon\Gamma\to\Gamma'$
we consider  $\Vert(\Gamma')$ 
as a subset of $\Vert(\Gamma)$ 
identifying $v'\in\Vert(\Gamma')$ 
with $\phi^{-1}(v')\in\Vert(\Gamma)$.
\end{sit}
\begin{remark}\label{rem:homeo}
A path $\gamma$ in $\Gamma$ is 
called \textit{simple} 
if $\gamma$ contains no loop and has 
no self-intersection. 
It is easily seen that, under a  birational morphism 
 $p\colon\Gamma\to\Gamma'$, the partial preimage of 
a simple path $\gamma$ 
 connecting vertices $u',v'\in\Vert(\Gamma')$ 
 is a simple path in $\Gamma$ 
 connecting $p^{-1}(u')$ and $p^{-1}(v')$. 
Moreover, the partial preimage $p^{-1}(\Delta')$ 
of a subgraph $\Delta'\subset\Gamma'$ 
is homeomorphic to $\Delta'$, 
i.e. the associated simplicial $1$-complexes 
viewed as topological spaces
are homeomorphic.  Also, $\Delta'$ and 
the total preimage $p^{*}(\Delta')$ are
homotopically equivalent, 
cf. \cite[Sec. 2.2]{FKZ-graphs}. 
\end{remark}
\begin{definition}\label{def:morphism-equiv}
Given a weighted graph $\Gamma$ we say 
that two birational morphisms 
$p_i\colon\Gamma\to\Gamma_i$, $i=1,2$, 
are \textit{equivalent} if 
there is a (uniquely defined) isomorphism 
of weighted graphs 
$\iota\colon\Gamma_1\to\Gamma_2$ 
such that $p_1^{-1}(\Delta)=p_2^{-1}(\iota(\Delta))$ 
for any subgraph $\Delta\subset\Gamma_1$.
If this holds for $\Gamma_1=\Gamma_2$ 
with $\iota=\id$ then we say that $p_1$ 
and $p_2$ are \textit{equal} 
and we write $p_1=p_2$. 
We say that a diagram of weighted graphs
\[\label{triangular-diagram}
\begin{tikzpicture}[baseline=(current  bounding  box.center)]
\node (c) {};
\draw (0,0.5cm) node (G) {$\Gamma$};
\node[below left of=c] (G1){$\Gamma_1$};
\node[below right of=c] (G2){$\Gamma_2$};
\draw[->] 
(G) edge[->] node[label=left:$p_1$] {} (G1)
(G1)  edge[->] node[label=below:$p_3$] {} (G2)
(G) edge[->]node[label=right:$p_2$] {} (G2);
\end{tikzpicture}
\]
where the $p_i$ are birational morphisms of graphs, 
commutes if $p_3\circ p_1=p_2$. 
In the latter case the diagram
\[\begin{tikzpicture}[baseline=(current  bounding  box.center)]
\node (c) {};
\draw (0,0.5cm) node (G) {$\Gamma$};
\node[below left of=c] (G1){$\Gamma_2$};
\node[below right of=c] (G2){$\Gamma_1$};
\draw[->] 
(G) edge[->] node[label=left:$p_2$] {} (G1)
(G1)  edge[dashed] node[label=below:$p_3^{-1}$] {} (G2)
(G) edge[->]node[label=right:$p_1$] {} (G2);
\end{tikzpicture}\]
is also called commutative. Square commutative 
diagrams of birational morphisms and more complicated 
commutative diagrams are defined in a similar way. 
\end{definition}
\begin{nota} \label{nota:p-v}
Given a birational morphism
$\phi\colon \Gamma\to\Gamma^{\prime}$ 
and a vertex $v$ of $\Gamma$, 
we let $\phi^v$ be the maximal initial subsequence 
of blowdowns 
and isomorphisms in $\phi$ preserving $v$. 
Thus, $\phi^v=\phi$ if and only if $v$ 
is not contracted under $\phi$.
In particular, $\phi^v=\phi$ 
if $v$ is non-rational. 
\end{nota}
\begin{lemma}\label{lem:contraction} $\,$
\begin{enumerate}[(a)]
\item Given a birational morphism 
$\phi\colon\Gamma\to\Gamma'$ 
that is not an isomorphism, 
there is a contraction $p\colon\Gamma\to\Gamma''$ 
equivalent to $\phi$.
\item Consider a composition 
$\Gamma\overset{\phi}{\to}\Gamma_1
\overset{\sigma}{\to}\Gamma_2$ 
of a birational morphism 
$\phi$ and a blowdown $\sigma$ 
of a vertex $v$ that is not adjacent in $\Gamma$ 
to any vertex contracted by $\phi$. 
Then there exists a birational morphism 
$\phi'\colon\Gamma\to\Gamma_2$ 
equal to $\sigma\circ\phi$ and starting 
with the blowdown of $v$.
\item Consider a birational morphism 
$\phi\colon\Gamma\to\Gamma_1$ 
and at most linear $(-1)$-vertex 
$v\in\Vert(\Gamma)$ 
contracted by $\phi$. 
Then there exists a birational morphism 
$\phi'\colon\Gamma\to\Gamma_2$ 
equal to $\phi$ and starting 
with the blowdown $\sigma_v$ of $v$.
\end{enumerate}
\end{lemma}
\begin{proof}
(a) We obtain the desired contraction 
$p\colon\Gamma\to\Gamma''$
by blowing down the vertices of $\Gamma$ 
blown down by $\phi$ 
in the same order as this is done under $\phi$.

(b) The weight of $v$ is $-1$ both in $\Gamma$ 
and $\Gamma_1$, and the weight of any 
vertex contracted by $\phi$ 
is the same in $\Gamma$ and after the blowdown 
of $v$ in $\Gamma$. 
Thus, we obtain $\phi'$ by contracting $v$ 
first and then contracting all the vertices 
contracted by $\phi$ in the same order. 
The weights of the non-contracted vertices 
are the same under $\sigma\circ\phi$ and $\phi'$. 
Now the assertion follows.

(c)  The vertex $v$ does not change 
its weight under $\phi^v$, hence it is adjacent 
to no vertex contracted by $\phi^v$. 
The claim follows now from (b) applied to 
$\phi^v$ and $\sigma_v$.
\end{proof}
\begin{proposition}\label{prop:morphism-equiv} 
Two birational morphisms 
$\phi_i\colon\Gamma\to\Gamma_i$, $i=1,2$ 
are equivalent if and only if the subsets of 
vertices of $\Gamma$ contracted by 
$\phi_1$ and $\phi_2$ coincide.
\end{proposition}
\begin{proof}
The `only if' part is clear. To show the `if' part
assume that each vertex contracted by 
$\phi_1$ is contracted by $\phi_2$ and vice versa. 
If there are no contracted vertices in $\Gamma$, 
then both $\phi_1$ and $\phi_2$ 
are isomorphisms and it suffices to 
set $\iota=\phi_2\circ \phi_1^{-1}$, 
see Definition \ref{def:morphism-equiv}.
Let now $v\in\Vert(\Gamma)$ be a $(-1)$-vertex  
contracted by the $\phi_i$.
By Lemma~\ref{lem:contraction}(c), up to 
equivalence we may suppose that both 
$\phi_1$ and $\phi_2$ start with the blowdown 
$\sigma_v$ of $v$.
This reduces the assertion to the one for 
the birational morphisms 
$\phi_i^\prime\colon \sigma_v(\Gamma)\to\Gamma_i$. 
Now the induction by the number 
of contracted vertices ends the proof.
\end{proof}
\begin{lemma}[{\rm see 
\cite[Lemma 2.4(a)]{FKZ-graphs}}]
\label{contract-branch} Consider a contraction 
$p\colon\Gamma\to\Gamma^{\prime}$. $\,$
\begin{itemize}
\item[(a)] 
Assume that a branch $W$ of $\Gamma$ 
at a vertex $v\in\Vert(\Gamma)$
is contracted under 
$p^v\colon\Gamma\to\Gamma^{v}$. 
Then $W$ is contractible and
there is a factorization $p^v=r\circ~\!\!q$
where $q\colon \Gamma\to\Gamma''$ is 
a contraction of $W$ and 
$r\colon\Gamma''\to\Gamma'$
is a birational morphism. 
Furthermore, $W$ is either simple 
(that is, linked to $v$ by a single edge, 
see Definition \ref{def:vertices}) 
or linked to 
 $v$ by exactly two edges. 
In the former case, $p^v$ decreases 
the degree of $v$ by $1$,
while in the latter case
$p^v$ replaces the branch
$W$ by a loop at $v$, so that
the degree of $v$ does not change. 
The latter case cannot happen provided 
$p$ contracts $v$. 
\item[(b)] 
Assume that a branch vertex 
$v\in \Br(\Gamma)$ is blown under $p$ down.
Then $v$ 
is rational and has no incident loop.
Furthermore, at least $\deg(v)-2>0$ 
branches of  
$\Gamma$ at $v$ are simple 
and contracted by $p^v$.
\end{itemize}
\end{lemma}
\begin{proof} (a) The first assertion is 
immediate from Lemma \ref{lem:contraction}(b) and (c).
Consider the contraction 
$q\colon\Gamma\to\Gamma''$
of $W$
induced by $p^v$ and its decomposition 
$q=\sigma_n\circ\ldots\circ\sigma_1$ 
into a sequence of blowdowns. 
Then $\sigma_1^{-1}$  is a blowup
of $\Gamma''$ either  at $v$, or 
at an incident loop at $v$.
So, $W$ is a tree linked to $v$ by 
a single edge in the former 
case and by two edges in the latter case, 
cf. Lemma \ref{lem:contractibility}(a). 

(b) Consider the contraction 
$p^v\colon\Gamma\to\Gamma^v$.
Since $p$ contracts $v$, $p^v(v)$ is at most linear 
$(-1)$-vertex of $\Gamma^v$. In particular,
$v$ is rational  and has no incident loop, 
see Definition \ref{def:vertices}.   

The contraction of a non-simple branch 
of $\Gamma$ at $v$ creates 
a loop of $\Gamma^v$
incident to $p^v(v)$, see (a). 
Since $p^v(v)$ has no incident loop, 
all contracted branches at $v$ are simple. 
Since $p^v(v)$ is at most linear, 
there are at least $\deg(v)-2$ 
such branches. 
\end{proof}
\begin{lemma}\label{lem:contract-segment}$\,$
\begin{itemize}
\item[(a)]
Let $p\colon\Gamma\to\Gamma'$ be a contraction
and $L$ be an extremal linear segment of $\Gamma$
that is not contracted under $p$. Then the image 
$p(L)$ is contained in an extremal linear 
segment of $\Gamma'$
and contains a tip of $\Gamma'$.
\item[(b)]
Let $\Gamma$ be a contractible graph. 
Then  $\Gamma$ contains a contractible 
extremal linear segment. 
\end{itemize}
\end{lemma}
\begin{proof} (a) The assertion is true if $p$ is a blowdown. 
Now the general case follows by induction 
on the number of blowdowns in a decomposition of $p$.

(b) 
The assertion is evident if $\Gamma$ is linear. Otherwise 
$\Gamma$ contains branch points. 
Let $\Br(\Gamma)=\{v_1,\ldots,v_k\}$. 
The number $\nu(\Gamma)=\sum_{i=1}^k (\deg(v_i)-2)$ is called 
the \textit{total branching number} of $\Gamma$, 
see \cite[Sec. 2.1]{FKZ-graphs}. Clearly,
$\nu(\Gamma)=0$ if and only if $\Gamma$ is linear.
On each step of the contraction the total branching number 
can only decrease. 
 At the first moment when $\nu(\Gamma)$ actually drops,
 an extremal linear branch of $\Gamma$ has been contracted. 
\end{proof}
\subsection{Relatively minimal diagrams} 
Recall the following well known fact.
\begin{lemma}\label{lem:bir-NC-surfaces}
Any birational transformation of NC-pairs 
$\Phi\colon (X_1,D_1)\dasharrow (X_2,D_2)$ 
fits in a commutative diagram
\begin{equation}\label{diagr:bir-pairs}
\begin{tikzpicture}[baseline=(current  
bounding  box.center),scale=1.3]
\node (c) {};
\node[] (G) at (1,1) {$(\widetilde X,\widetilde D)$};
\node[] (G1) at (0,0) {$(X_1,D_1)$};
\node[] (G2) at (2,0) {$(X_2,D_2)$};
\draw[->] 
(G) edge[->] node[label=left:$P_1$] {} (G1)
(G1) edge[dashed] node[label=below:$\Phi$] {} (G2)
(G) edge[->]node[label=right:$P_2$] {} (G2);
\end{tikzpicture}
\end{equation}
\noindent where every $P_i$ is 
a composition of an isomorphism and blowdowns 
of smooth rational $(-1)$-components 
of boundary divisors
and  all the intermediate pairs
are NC-pairs.
\end{lemma}
\begin{proof}
Any birational transformation of 
smooth projective surfaces $\Phi\colon X_1\to X_2$ 
fits in a commutative diagram
\[\begin{tikzpicture}
\node (c) {};
\draw (0,0.5cm) node (G) {$\widetilde X$};
\node[below left of=c] (G1){$X_1$};
\node[below right of=c] (G2){$X_2$};
\draw[->] 
(G) edge[->] node[label=left:$P_1$] {} (G1)
(G1) edge[dashed] node[label=below:$\Phi$] {} (G2)
(G) edge[->]node[label=right:$P_2$] {} (G2);
\end{tikzpicture}
\]
where $P_i$ is a composition of an isomorphism 
and blowdowns of smooth rational $(-1)$-curves, see e.g. 
\cite[Corollaries (8.10)-(8.11)]{Mu} or 
\cite[Corollary II.12]{Be}. This can be applied as well 
in our setting after a simultaneous resolution of 
singularities of $X_1$ and $X_2$. Indeed, 
$\Phi|_{X_1\setminus \supp D_1}\colon 
X_1\setminus \supp D_1\to 
X_2\setminus \supp D_2$ is  a biregular isomorphism 
and $X_i$ is smooth near $D_i$, 
see Definitions \ref{def:NC-pairs} and \ref{def:bir-NC-pairs}.

Note that $P_i^{-1}$ can be 
written as a composition of an isomorphism 
and a sequence of blowups with smooth centers 
located at points of $\supp D_i$ and 
infinitesimally near points.
A blowup of an NC-pair $( X,D)$ in a point of 
$\supp D$ yields an NC-pair. 
Hence all the intermediate  pairs in the decomposition 
\[P_i\colon (\widetilde X,\widetilde D)
\stackrel{\sigma_{i,n_i}}{\longrightarrow} 
(X_{i, n_i-1},D_{i, n_i-1})\longrightarrow\cdots\longrightarrow 
(X_{i,2},D_{i,2})\stackrel{\sigma_{i,2}}{\longrightarrow} 
(X_{i,1},D_{i,1})\simeq (X_{i},D_{i})\] 
are NC-pairs. 
\end{proof}
To formulate an analog of this lemma for 
birational transformations of weighted 
graphs one needs 
to explain the meaning of 
the corresponding commutative diagram. 
\begin{definition}\label{def:fit-diagram}
Consider a diagram of weighted graphs
\begin{equation}\label{triangular diagram}
\begin{tikzpicture}[baseline=(current  
bounding  box.center)]
\node (c) {};
\draw (0,0.5cm) node (G) {$\Gamma$};
\node[below left of=c] (G1){$\Gamma_1$};
\node[below right of=c] (G2){$\Gamma_2$};
\draw[->] 
(G) edge[->] node[label=left:$p_1$] {} (G1)
(G) edge[->]node[label=right:$p_2$] {} (G2);
\end{tikzpicture}
\end{equation}
where the $p_i$ are birational morphisms of graphs.
Consider  also a birational transformation 
$\phi\colon\Gamma_1\dasharrow\Gamma_2$
 written as
\begin{equation}\label{eq:decomposition} 
\phi\colon \Gamma_1=\Delta_1\stackrel{\psi_1}
{\dasharrow}\Delta_2\stackrel{\psi_2}{\dasharrow}
\cdots\stackrel{\psi_{n-2}}{\dasharrow}\Delta_{n-1}
\stackrel{\psi_{n-1}}{\dasharrow}\Delta_{n}=
\Gamma_2\end{equation}
where the $\Delta_i$ are weighted graphs
and exactly one of the 
$\psi_i$ and $\psi_{i+1}$ is a birational morphism 
and the other one is the inverse of a birational morphism. 
We say that $\phi$ \textit{fits in} 
\eqref{triangular diagram} 
if 
\eqref{eq:decomposition} is included 
in a sequence of commutative diagrams 
\begin{equation}\label{triangular diagram-1}
\begin{tikzpicture}[baseline=(current  
bounding  box.center)]
\node (c) {};
\draw (0,0.5cm) node (G) {$\Gamma$};
\node[below left of=c] (G1){$\Delta_i$};
\node[below right of=c] (G2){$\Delta_{i+1}$};
\draw[->] 
(G) edge[->] node[label=left:$\eta_i$] {} (G1)
(G1) edge[dashed] node[label=below:$\psi_i$] {} (G2)
(G) edge[->]node[label=right:$\eta_{i+1}$] {} (G2);
\end{tikzpicture}
\end{equation}
where $\eta_1=p_1$ and $\eta_n=p_2$; 
cf. Definition \ref{def:morphism-equiv}. 
\end{definition}
\begin{remark}\label{rem:bir-equiv}
The graphs $\Gamma_1$ and $\Gamma_2$
 in  diagram \eqref{triangular diagram} are 
 birationally equivalent via the birational 
 transformation $\phi'=p_2\circ p_1^{-1}
 \colon\Gamma_1\dasharrow\Gamma_2$. 
 Thus, starting with a diagram 
 \eqref{triangular diagram}  
 we come to a birational transformation $\phi'$. 
 In the next proposition we show that, conversely, 
 given a birational transformation of weighted graphs 
 one can construct a corresponding diagram 
 \eqref{triangular diagram}.
\end{remark}
\begin{proposition}[{\rm see 
\cite[Remark A1(1)]{FZ}}]\label{prop:domination}
Let $\phi\colon\Gamma_1\dasharrow\Gamma_2$ 
be a birational transformation of weighted graphs. 
Then $\phi$ fits in a diagram 
\eqref{triangular diagram} for some $\Gamma,p_1,p_2$.
\end{proposition}
\begin{proof} The assertion is evidently true if 
$\phi$ is an isomorphism. 
Assume this is not the case. 
Up to an automorphism of $\Gamma_2$, 
we can write $\phi$ as in \eqref{eq:decomposition}, 
where $\psi_i$ for $1\le i\le n-2$ is either 
a product of blowdowns following by 
a product $\psi_{i+1}$ of blowups, 
or such a product in the reverse order, 
see Remark \ref{rems}.1. 
Let us proceed by recursion. 
For $n=1$ our assertion is obvious. 

Assume now that $n=2$ and 
$\phi=\sigma_2^{-1}\sigma_1$ 
where the $\sigma_i:\Gamma_i\to\Gamma_0$ 
are blowdowns.
If both $\sigma_1^{-1}$ and $\sigma_2^{-1}$ 
are inner blowups of the same edge of $\Gamma_0$, 
then the graph $\Gamma=\Gamma_1=\Gamma_2$ 
dominates both $\Gamma_1$ and $\Gamma_2$ 
via isomorphisms. 
Otherwise, applying simultaneously both 
of them one transforms $\Gamma_0$ 
into a graph $\Gamma$ 
with two distinct vertices $v_1$ and $v_2$ 
such that $\Gamma$ dominates $\Gamma_1$ 
(resp., $\Gamma_2$) 
via the contraction $p_1$ of $v_2$ (resp., $p_2$ 
of $v_1$). 
This produces a desired diagram \eqref{triangular diagram}.

\begin{figure}[h]
\begin{tabular}{l@{\hspace{30pt}}l@{\hspace{30pt}}l}
    \pgfdeclarelayer{background}
		\pgfsetlayers{background,main}
\begin{diagram}[notextflow,height=2em]
&& & & & & \Gamma_{0,2,3}=\Gamma & & & & &\\
&& & & & \ldTo_{} &                                           & \rdTo_{} & & & &\\
&& & &  \Gamma_{0,2,2} & &  & \rDashto & \Gamma_{0,1,3}  & & & \\
&& & \ldTo_{}  & & \rdTo_{}  & & \ldTo_{}   & & \rdTo_{}     &  &  \\
&&  \Gamma_{0,2,1} & & & \rDashto  &\Gamma_{0,1,2}    & & & \rDashto & \Gamma_{2,3} =\Gamma_{2}  &  \\
& \ldTo_{}  & & \rdTo_{}  & & \ldTo_{}  &  & \rdTo_{}     &  & \ldTo_{\sigma_{2,3}} & &  & \\
\Gamma_1 =\Gamma_{1,2} &    & & \rDashto  & \Gamma_{0,1,1}   &    &\rDashto  &  & \Gamma_{2,2}  &   &     \\
      & \rdTo_{\sigma_{1,2}} & &  \ldTo_{}  &  &\rdTo_{} &     & \ldTo_{\sigma_{2,2}}  & & &  \\
      && \Gamma_{1,1}  &  & & \rDashto &  \Gamma_{2,1} &  &  &  & &     \\
    &&  & \rdTo_{\sigma_{1,1}}  &  & \ldTo_{\sigma_{2,1}} &  &  &  &  &  &  &  &\\
      &&    &   & \Gamma_{0}    &      &  &    &  & & &  \\
      \end{diagram}
      \end{tabular}
  \caption{Construction of diagram 
  \eqref{triangular diagram} for 
  $m_1=2$ and $m_2=3$.}\label{fig:romb}
\end{figure}
Let further $n=2$ and $\phi=\psi_2^{-1}\psi_1$, 
where this time the $\psi_i$ stand 
for birational morphisms. We can write
\[\psi_i=\sigma_{i,1}\cdots\sigma_{i,m_i}\colon \Gamma_i
=\Gamma_{i,m_i}\to\Gamma_0=\Gamma_{i,0}\] 
where 
$\sigma_{i,j}\colon \Gamma_{i,j}\to\Gamma_{i,j-1}$ 
is a blowdown, $i=1,2$. 
Applying the preceding case to the composition 
$\sigma_{2,1}^{-1}\sigma_{1,1}\colon \Gamma_{1,1}
\dasharrow\Gamma_{2,1}$ 
we can find a weighted graph $\Gamma_{0,1,1}$ 
which dominates both $\Gamma_{1,1}$ and $\Gamma_{2,1}$ 
making the square diagram commutative, 
see Figure \ref{fig:romb}. 
The same procedure applied to the pairs 
$\{\Gamma_{2,2},\,\Gamma_{0,1,1}\}$ and 
$\{\Gamma_{0,1,1},\,\Gamma_{1,2}\}$ 
yields two new graphs $\Gamma_{0,1,2}$ 
and $\Gamma_{0,2,1}$ which dominate 
the corresponding pairs. 
Continuing in this way we fill in a commutative
diagram  in the form of a lattice parallelogram 
consisting of $(m_1+1)(m_2+1)$
 weighted graphs and their morphisms. 
 This parallelogram has the pairs of opposite 
 vertices $\{\Gamma_1,\,\Gamma_2\}$ and 
 $\{\Gamma_0,\,, \Gamma\}$, 
 where $\Gamma=\Gamma_{0,m_1,m_2}$ 
 dominates both $\Gamma_1$ and $\Gamma_2$ 
 via sequences of blowdowns inducing $\phi$, 
 see Figure \ref{fig:romb} for $m_1=2$ and $m_2=3$. 
 This yields the assertion for $n=2$. 

Proceeding by induction on $n$ we assume that 
$n\ge 3$ and we choose $i\ge 2$ such that 
$\psi_i^{-1}$ and $\psi_{i-1}$ in 
\eqref{eq:decomposition} are morphisms. 
By the preceding,  the product $\psi_i\psi_{i-1}$ 
can be replaced by $\rho_i\rho_{i-1}$ where this time
$\rho_i$ and $\rho_{i-1}^{-1}$ are morphisms. 
Replacing now $\psi_{i+1}\rho_i$ by 
$\psi'_{i+1}$ provided $i\le n-1$ and $\rho_{i-1}\psi_{i-2}$ 
by $\psi'_ {i-2}$ provided $i\ge 3$
we  transform \eqref{eq:decomposition} into 
a shorter decomposition of the same type.
This gives the inductive step. 

The graph $\Gamma$ in diagram \eqref{triangular diagram} 
constructed in the proof dominates every pair of 
intermediate graphs obtained under our procedure. 
In particular, $\phi$ fits in this diagram, 
see Definition \ref{def:fit-diagram}.
\end{proof}
\begin{definition} A diagram \eqref{triangular diagram} 
with two minimal weighted graphs $\Gamma_1$
and $\Gamma_2$ is called {\em relatively minimal} 
if no ($-1$)-vertex of $\Gamma$ is contracted 
in both $\Gamma_1$ and $\Gamma_2$. 
\end{definition}
\begin{remark} 
 Given a birational map 
 $\phi\colon\Gamma_1\dasharrow\Gamma_2$ 
fitting in a relatively minimal diagram
 \eqref{triangular diagram}, 
the graph $\Gamma$ is not uniquely defined, 
in general, see Figure \ref{fig:2-graphs} 
for a simple example;
cf. also \cite[Example 2.6]{FKZ-graphs}.
\begin{figure}[h]
\begin{tabular}{l@{\hspace{30pt}}l@{\hspace{30pt}}l}
    \pgfdeclarelayer{background}
		\pgfsetlayers{background,main}
    \begin{tikzpicture}[scale=1]
      
       \node[] (G) at (0,2.4) {$\Gamma$};
       \node[] (-1) at (0.5,2.7) {\footnotesize $-1$};
       \node[vertex] (a) at (0.5,2.4) {};
       \node[] (1) at (1.5,2.7) {\footnotesize $1$};
        \node[vertex] (b) at (1.5,2.4) {};
        \node[] (G1) at (-1, 1.4) {$\Gamma_1$};
        \node[] (-1) at (-0.5,1.7) {\footnotesize $-1$};
        \node[vertex] (a1) at (-0.5,1.4) {};
        \node[] (1) at (0.5,1.7) {\footnotesize $1$};
        \node[vertex] (b1) at (0.5,1.4) {};
        \node[] (G2) at (3, 1.4) {$\Gamma_2$};
        \node[] (-1) at (1.5,1.7) {\footnotesize $-1$};
        \node[vertex] (a2) at (1.5,1.4) {};
         \node[] (1) at (2.5,1.7) {\footnotesize $1$};
        \node[vertex] (b2) at (2.5,1.4) {};
        \node[] (G0) at (0.5, 0.4) {$\Gamma_0$};
         \node[] (2) at (1,0.7) {\footnotesize $2$};
        \node[vertex] (b0) at (1,0.4) {};
        
\node[] (tG) at (4.5,2.4) {$\Gamma'$};
       \node[] (-1) at (5,2.7) {\footnotesize $-1$};
       \node[vertex] (A) at (5,2.4) {};
       \node[] (-1) at (7,2.7) {\footnotesize $-1$};
        \node[vertex] (B) at (7,2.4) {};
        \node[vertex] (A) at (5,2.4) {};
       \node[] (0) at (6,2.7) {\footnotesize $0$};
        \node[vertex] (C) at (6,2.4) {};
        \node[] (tG1) at (4, 1.4) {$\Gamma_1$};
        \node[] (-1) at (4.5,1.7) {\footnotesize $-1$};
        \node[vertex] (a11) at (4.5,1.4) {};
        \node[] (1) at (5.5,1.7) {\footnotesize $1$};
        \node[vertex] (b11) at (5.5,1.4) {};
        \node[] (tG2) at (8, 1.4) {$\Gamma_2$};
        \node[] (-1) at (6.5,1.7) {\footnotesize $-1$};
        \node[vertex] (a21) at (6.5,1.4) {};
         \node[] (1) at (7.5,1.7) {\footnotesize $1$};
        \node[vertex] (b21) at (7.5,1.4) {};
        \node[] (tG0) at (5.5, 0.4) {$\Gamma'_0$};
         \node[] (2) at (6,0.7) {\footnotesize $2$};
        \node[vertex] (b01) at (6,0.4) {};
     
      \draw[edge] 
      (a)edge(b) (a1)edge(b1) (a2)edge(b2) 
      (A)edge(C) (C)edge(B) (a11)edge(b11) (a21)edge(b21);  
     \draw[->](0.8,2.2) -- (0,1.6); 
      \node[] (p1) at (0.3, 2.0) {\footnotesize $p_1$};
      \draw[->](1,2.2) -- (1,1); 
       \node[] (s1) at (0.3, 0.8) {\footnotesize $\sigma_1$};
      \node[] (s2) at (1.7, 0.8) {\footnotesize $\sigma_2$};
      \node[] (p2) at (1.7, 2) {\footnotesize $p_2$};
      \draw[->](1.2,2.2) -- (2,1.6);
      \draw[->](0,1.2) -- (0.8,0.6);
      \draw[->](2,1.2) -- (1.2,0.6);

       \draw[->](5.8,2.2) -- (5,1.6); 
      \node[] (p11) at (5.2, 2) {\footnotesize $p'_1$};
      \draw[->](6,2.2) -- (6,1); 
      \node[] (s11) at (5.3, 0.8) {\footnotesize $\sigma_1$};
      \node[] (s21) at (6.8, 0.8) {\footnotesize $\sigma_2$};
      \node[] (p21) at (6.8, 2) {\footnotesize $p'_2$};
      \draw[->](6.2,2.2) -- (7,1.6);
      \draw[->](5,1.2) -- (5.8,0.6);
      \draw[->](7,1.2) -- (6.2,0.6);

         \end{tikzpicture}
\end{tabular}
  \caption{Two different dominations
 of the same weighted graphs.}
  \label{fig:2-graphs}
\end{figure}

\noindent In both diagrams, the linear graphs $\Gamma_1$ 
and $\Gamma_2$ are isomorphic, 
the $\sigma_i$ are blowdowns of $(-1)$-vertices 
and $\phi=\sigma_2^{-1}\circ\sigma_1\colon\Gamma_1
\dasharrow\Gamma_2$ 
is a birational transformation.
In the left diagram 
$\Gamma\simeq \Gamma_1\simeq \Gamma_2$,
while in the right diagram $p_1'$ and $p_2'$ 
 are blowdowns of 
distinct $(-1)$-vertices of $\Gamma'$ and  
$\Gamma_1\simeq\Gamma_2\not\simeq\Gamma$.
The vertical arrows in both diagrams correspond 
to the left arrow in  diagram 
\eqref{triangular diagram-1} with $i=2$.
\end{remark}
According to the following lemma,  
 for any pair of birationally equivalent minimal 
 weighted graphs $\Gamma_1$ and $\Gamma_2$ 
 there is a weighted graph $\Gamma$ and birational 
 morphisms $p_i\colon\Gamma\to\Gamma_i$ 
 that form a relatively minimal diagram \eqref{triangular diagram}.  
\begin{lemma}[{\rm cf. \cite[Remark A1(1)]{FZ}}]\label{lem:rel-min}
Given a diagram \eqref{triangular diagram} 
with minimal weighted graphs $\Gamma_1$ and $\Gamma_2$ 
there exist birational morphisms 
\[\Gamma\stackrel{\psi}{\longrightarrow} 
\Gamma'\stackrel{p'_i}{\longrightarrow} \Gamma_i,\quad i=1,2\] 
such that $p_i$ is equivalent to $p'_i \circ \psi$ for $i=1,2$ 
and $\Gamma', p'_i, \Gamma_i$, $i=1,2$ fit in
a relatively minimal diagram \eqref{triangular diagram}.
\end{lemma}
\begin{proof}
Let $v\in\Vert(\Gamma)$ be a $(-1)$-vertex contracted 
by both $p_1$ and $p_2$. Replacing the $p_i$ 
by  equivalent birational morphisms $\Gamma\to\Gamma_i$ 
as in Lemma~\ref{lem:contraction}(c)
we may suppose that both $p_1$ and $p_2$ 
start with the same 
blowdown $\sigma_v\colon\Gamma\to\Gamma_v$ of $v$ 
and there are the factorizations
\[p_i\colon\Gamma\stackrel{\sigma_v}{\longrightarrow} 
\Gamma_v\stackrel{p'_{v,i}}{\longrightarrow} \Gamma_i\] 
where $\Gamma_v, p_{v,1}'$ and $p_{v,2}'$ form 
a diagram  \eqref{triangular diagram}.
Now the recursion on the number of vertices  of $\Gamma$ 
ends the proof. 
\end{proof}
\subsection{Dual graphs of NC-pairs}
 Let us recall the correspondence between 
 NC-pairs and weighted graphs.
\begin{definition}\label{def:dual}
Given an NC-pair $(X,D)$ the {\em dual graph} 
$\Gamma(D)$  is a  weighted graph
 whose vertices are in bijection with the irreducible 
 components of $D$ 
 and edges are in bijection with the nodes of  $D$. 
The loops of $\Gamma(D)$ at a vertex $C$ 
are in bijection with the  self-intersection points 
of the component $C$ of $D$, and 
the edges joining two different vertices $C_1$ 
and $C_2$ of $\Gamma(D)$ are in bijection with 
the points of $C_1\cap C_2$.
The weight of $C$ in $\Gamma(D)$ is 
the self-intersection  index $C^2$ of 
the component $C$ in $X$.  
The rational vertices in $\Ratio(\Gamma(D))$ 
correspond to the rational components of $D$.  
\end{definition}
 An NC-pair $(X,D)$ is 
 an SNC-pair if and only if 
 $\Gamma(D)$ contains no loops and multiple edges, 
 cf.  Definition \ref{def:NC-pairs}. 
 Clearly,  any NC-pair is dominated by an SNC-pair.  
 An NC-pair $(X,D)$ is minimal
 if the dual graph $\Gamma(D)$ is minimal. 
 
The notions of inner and outer blowups of 
an NC-pair are consistent with the notions 
of inner and outer blowups 
of the dual graph $\Gamma(D)$, respectively.
\begin{example}\label{deformations-remark} 
Recall that any non-complete normal  
algebraic surface $Y$ admits an SNC-completion $(X,D)$,  
where
$Y=X\setminus \supp D$,
and an NC-completion with a minimal dual graph.

For example, if $X=\PP^2$ and $D$ is a nodal cubic, 
then  $\Gamma(D)$ is minimal and consists 
of a single vertex of weight $9$ and a loop. 
So the pair $(\PP^2,D)$ is an NC-completion  
of the  affine surface 
$Y=\PP^2\setminus D$ with a minimal dual graph. 
There exists a non-minimal SNC-completion $(X,D_1)$ 
whose dual graph $\Gamma(D_1)$ is a cycle 
with three rational vertices and a cyclically ordered 
sequence of weights $((-2, -1, 4))$. 
Moreover, there exists a minimal such completion 
with the sequence of weights $((0,0,-2, -2, -2,-2,-3))$.  
Thus, two minimal cyclic graphs with rational vertices 
and the cyclically ordered sequences of weights $((9))$ 
and $((0,0,-2, -2, -2,-2,-3))$, respectively, 
are birationally equivalent. 

Similarly, one can show that for any $a, b>0$ 
the minimal linear graphs with sequences of weights 
\[[[a]],\quad [[\underbrace{-2,-2,\ldots,-2,0,0}_{a+1}]] \,\,\,\text{and}\,\,\, 
[[\underbrace{-2,\ldots,-2,-3}_{a-1},\underbrace{-2,\ldots,-2,0,1}_{b+1}]]\]
are  birationally equivalent, cf. Theorem~\ref{prop:classification}. 
\end{example}
The assumption that the surface $Y=X\setminus\supp D$ 
is affine imposes severe restrictions 
on the dual graph $\Gamma(D)$. 
\begin{lemma}\label{lem:Hartogs} 
Let an NC-pair $(X,D)$ be  a completion 
of a normal affine surface  
$Y=X\setminus\supp D$. Then the following hold.
\begin{enumerate}
\item[(a)] $\supp D$ is connected 
and supports an effective ample divisor. 
\item[(b)] The dual graph $\Gamma(D)$ is connected, 
the intersection form $I(\Gamma(D))$ 
is not seminegative definite, 
and $D$ and $\Gamma(D)$ 
are not contractible. 
\item[(c)]
If $\supp D=C$ is an irreducible, reduced curve, 
then $C^2>0$ and the linear system $|C|$ is ample. 
If, moreover, $C$ is smooth and rational,
then $X$ is rational. 
\end{enumerate}
\end{lemma}
\begin{proof} According to the Goodman criterion 
of affiness for surfaces \cite[Theorem 2]{Go69} 
(cf. also \cite[Lemma 2]{Gi1}),  
$\supp D$ coincides with the support of 
an ample divisor $H$ on $X$.
The connectedness of $\supp D=\supp H$ follows now 
from the Lefschetz hyperplane section theorem, 
see e.g. \cite[p. 166, Corollary]{Go69} 
or \cite[Corollary II.6.2]{Har70}. 
Hence the graph $\Gamma(D)$ is connected. 
This proves statement (a). 

Since $H^2> 0$,
the intersection form $I(\Gamma(D))$ 
cannot be seminegative definite. 
By Lemma \ref{lem:contractibility} 
the dual graph $\Gamma(D)$ is not contractible, 
hence also $D$ is  not contractible. This proves (b).

Assume now that $\supp D=C$ 
is a reduced  irreducible curve. 
Then $C^2>0$, 
and the linear system $|C|$ 
is ample by the Nakai-Moishezon criterion. 
If $C$ is smooth and rational, then $X$ 
is rational as well, see 
\cite[Remarks 2 and 3]{Gi1} or 
\cite[Proposition~V.4.3]{BHPV}.  This shows (c).
\end{proof}
Lemma \ref{lem:Hartogs} justifies the following convention. 
\begin{sit}\label{convention-bis} \textbf{Convention.} 
{\em In the following, we consider only
weighted graphs $\Gamma$ with no contractible
connected component.}
\end{sit}
\begin{definition}\label{def:repr}
Every decomposition of $\Phi\in\Bir((X_1,D_1),(X_2,D_2))$ 
into a sequence of isomorphisms, blowups and blowdowns 
induces a  birational transformation of dual graphs 
$\phi\colon\Gamma(D_1)\dashrightarrow\Gamma(D_2)$ 
called a \textit{representation} of $\Phi$.
\end{definition}
In the opposite direction,
we have the following result.
\begin{proposition}[{\rm cf. \cite[Proposition 3.34]
{FKZ-graphs}}]\label{prop:transf-repr}
Given an NC-pair $(X,D)$ and a birational 
transformation of the dual graph 
$\phi\colon\Gamma(D)\dashrightarrow \Gamma'$ 
consisting of a sequence of blowups 
and blowdowns, one can find a new NC-pair 
$(X',D')$ and a birational transformation of NC-pairs 
$\Phi\colon (X,D)\dashrightarrow (X',D')$
 such that $\Phi|_{X\setminus \supp D}\colon 
 X\setminus \supp D\to X'\setminus \supp D'$ 
is an isomorphism and $\Phi$ is represented 
by $\phi$; in particular, $\Gamma(D')=\Gamma'$. 
If $\phi$ contains no outer blowup,
then $\Phi$ is  uniquely defined up to 
isomorphism of NC-pairs.
\end{proposition}
\begin{proof} The first statement clearly holds
for a single blowup and a single blowdown 
of $\Gamma(D)$. By recursion, it holds 
in the general case. 
The uniqueness follows by recursion 
from the facts that
to an inner blowup of $\Gamma(D)$ 
there corresponds the blowup of $X$ 
at a uniquely defined node of $D$, 
and to a blowdown of a $(-1)$-vertex 
of $\Gamma(D)$ there corresponds the 
contraction of a uniquely defined 
$(-1)$-component of $D$. 
\end{proof}
\begin{remark}\label{rem:up-and-down} 
Let $(X,D)$ be an NC-pair, and let $C$ 
be a smooth component of $D$ that
corresponds to a  
$(0)$-vertex $v$ of $\Gamma(D)$ of degree 1 or 2. 
Fix a point $P\in C$, which is a node of $D$ 
in the case where $\deg_{\Gamma(D)}(v)=2$.
Blowing $P$ up and 
contracting the proper transform of $C$ yields 
a birational transformation 
$\Phi\colon (X,D)\dasharrow (X',D')$ called an 
 \textit{elementary transformation}. 
 The corresponding birational transformation $\phi$ 
of $\Gamma(D)$ is also called \textit{elementary}; 
it affects the weights of the 
neighbors of $C$ in $\Gamma(D)$.
See the proof of Lemma \ref{lem:min-mod} for an example.

Let now $D=C$ be a smooth rational curve, and so 
$\Gamma(D)$ consists of a single isolated vertex 
$C$ of weight $a=C^2$ with no incident loop. 
Take two distinct points $P_1$ and $P_2$ of $C$. 
Blowing  $P_i$ up produces an NC-pair $(X_i,D_i)$ 
whose dual graph $\Gamma(D_i)$ 
has a sequence of weights $[[-1,a-1]]$. 
Let $\sigma_i\colon X_i\to X$ be the blowdown 
of the exceptional curve. The pairs $(X_1,D_1)$ 
and $(X_2,D_2)$ 
are not isomorphic, in general, while the dual 
graphs $\Gamma(D_1)$ and $\Gamma(D_2)$ are. 
The corresponding birational transformation 
$\Phi=\sigma_2^{-1}\sigma_1\colon (X_1,D_1)\dasharrow (X_2,D_2)$ 
is not an isomorphism of pairs. 
It restricts to the identity on $Y=X\setminus \supp D$. 
On the level of dual graphs,  the birational map 
$\phi=\sigma_2^{-1}\sigma_1\colon \Gamma(D_1)
\dasharrow\Gamma(D_2)$ in the induced representation 
of $\Phi$ fits in the commutative diagram
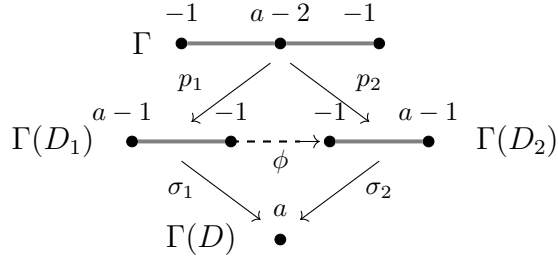
\[
\begin{tabular}{l@{\hspace{30pt}}l@{\hspace{30pt}}l}
    \pgfdeclarelayer{background}
		\pgfsetlayers{background,main}
    \begin{tikzpicture}[scale=1.3]
      
       \node[] (G) at (0,1.4) {$\Gamma(D)$};
       \node[] (-1) at (0.3,2.7) {\footnotesize $-1$};
       \node[vertex] (a1) at (-0.7,2.4) {};
       \node[] (a-1) at (2.3,2.7) {\footnotesize $a-1$};
        \node[vertex] (b1) at (0.3,2.4) {};
          \node[] (-1) at (1.3,2.7) {\footnotesize $-1$};
       \node[vertex] (a2) at (1.3,2.4) {};
       \node[] (a-1) at (-0.8,2.7) {\footnotesize $a-1$};
        \node[vertex] (b2) at (2.3,2.4) {};
        \node[] (G1) at (-1.5, 2.4) {$\Gamma(D_1)$};
        \node[] (c) at (0.8,1.7) {\footnotesize $a$};
        \node[] (G2) at (3.2, 2.4) {$\Gamma(D_2)$};
        \node[vertex] (c) at (0.8,1.4) {};
        \node[] (a-2) at (0.8,3.7) {\footnotesize $a-2$};
        \node[vertex] (d) at (0.8,3.4) {};
        \node[] (-1) at (-0.2,3.7) {\footnotesize $-1$};
        \node[] (G) at (-0.6, 3.4) {$\Gamma$};
        \node[vertex] (d1) at (-0.2,3.4) {};
        \node[] (-2) at (1.6,3.7) {\footnotesize $-1$};
        \node[vertex] (d2) at (1.8,3.4) {};
        \draw[thick, dashed](0.3,2.4) -- (1,2.4);
        \draw[->](1,2.4) -- (1.2,2.4);
        \node[] (phi) at (0.8, 2.2) {\footnotesize $\phi$};
   
      \draw[edge] 
      (a1)edge(b1);    
      \draw[edge]   
      (a2)edge(b2);  
      \draw[edge] 
      (d1)edge(d);   
      \draw[edge] 
      (d)edge(d2);   
     \draw[->](-0.2,2.2) -- (0.6,1.6); 
      \node[] (p1) at (-0.2, 1.9) {\footnotesize $\sigma_1$};
      \node[] (p2) at (1.8, 1.9) {\footnotesize $\sigma_2$};
      \draw[->](1.8,2.2) -- (1,1.6);
           \draw[->](0.9,3.2) -- (1.7,2.6); 
            \node[] (p1) at (-0.1, 3) {\footnotesize $p_1$};
      \node[] (p2) at (1.7, 3) {\footnotesize $p_2$};
           \draw[->](0.7,3.2) -- (-0.1,2.6);
     
         \end{tikzpicture}
\end{tabular}
\]
where $\Gamma$ is the dual graph of the pair 
$(\tilde X, \tilde D)$ obtained from $(X,D)$ 
by blowing up the points 
$P_1$ and $P_2$ on $X$. Varying the positions 
of $P_1$ and $P_2$ on $C$ yields nontrivial deformations 
of the participating pairs.  
Thus, $\Phi$ cannot be reconstructed in general 
from its representation $\phi=p_2\circ p_1^{-1}$, 
once outer blowups are involved.
\end{remark}
We have the following geometric analog of 
Lemma \ref{lem:rel-min}. 
\begin{lemma}\label{lem:rel-min-pairs} 
Any birational transformation 
$\Phi\colon (X_1,D_1)\dasharrow (X_2,D_2)$ 
between minimal NC-pairs $(X_i,D_i)$
fits in a commutative diagram
\eqref{diagr:bir-pairs} such that a representation 
$\phi$ of $\Phi$ fits into the corresponding 
relatively minimal diagram \eqref{triangular diagram}. 
\end{lemma}
\begin{proof} 
Starting with an arbitrary  commutative diagram
\eqref{diagr:bir-pairs} we repeat the procedure 
from the proof of Lemma \ref{lem:rel-min} 
on the geometric level. 
Namely, to each at most linear 
$(-1)$-vertex of $\Gamma(\widetilde D)$ 
there corresponds a $(-1)$-component $C$ of 
$\widetilde D$ 
which meets the union of other components 
transversally in at most two points.  
If $C$ is contracted under the $P_i\colon X\to X_i$, 
$i=1,2$ then there are factorizations 
\[P_i\colon \widetilde X\stackrel{\sigma_C}{\longrightarrow} 
\widetilde X'\stackrel{P'_i}{\longrightarrow} X_i,\] 
where $\sigma_C$ stands for the contraction of $C$ in 
$\widetilde X$, see \cite[Proposition II.8]{Be}. 
This produces a new NC-pair $(\widetilde X',\widetilde D')$ 
which still dominates both $(X_i,D_i)$ fitting 
in a new diagram 
\eqref{diagr:bir-pairs} with the same 
$\Phi\colon (X_1,D_1)\dasharrow (X_2,D_2)$. 
A simultaneous resolution of singularities of 
$X_1,X_2$ and $\tilde X$ does not affect our procedure.
Hence, we may assume that all these surfaces 
are smooth. We have $\rho(\widetilde X')=
\rho(\widetilde X)-1$
where $\rho$ stands for the Picard rank, 
see e.g. \cite[Proposition II.3]{Be}. 
The induction on the Picard rank shows 
that this procedure ends 
with a relatively minimal diagram. Finally, we 
contract simultaneously the exceptional divisors 
of the resolutions
and arrive at the same conclusion 
for the original surfaces. 
\end{proof}
%
\section{The Graph Lemma and birational rigidity 
of weighted graphs} \label{sec:graph-lemma}
In this section we elaborate criteria as to when 
$\Bir(X,D)=\Inn(X,D)$ 
and, on the combinatorial level,  as to when  
every relatively minimal birational 
transformation between two weighted graphs 
 is inner, see Theorem \ref{th:bir=inn} and Proposition 
 \ref{cor:graph}, respectively. 
To this end, we apply an important tool 
called the Graph Lemma. 
\subsection{The Graph Lemma}
\label{graph-lemma-subsection}
We use the following terminology.
\begin{definition}\label{def:admissible}
Let $\Gamma$  be a minimal  weighted graph. 
A circular segment $S$ of $\Gamma$ is called
\textit{admissible} if either the weights of 
its vertices are $\le-2$,
or it consists of a single vertex of weight $\le 2$ 
and a loop.
A linear segment $L$ of $\Gamma$  is
called \textit{admissible} if the weights 
of its vertices are
$\le -2$. 
\end{definition}
The following Graph Lemma \ref{graph-lemma} 
extends
\cite[Proposition A1]{FZ} and 
\cite[Lemma 2.4(b)]{FKZ-graphs}. 
For the reader's convenience we provide a proof. 
 Recall that for a contraction 
$p\colon\Gamma\to\Gamma'$ 
we consider $\Vert(\Gamma')$ 
as a subset of $\Vert(\Gamma)$, that is, 
we identify every vertex of  
$\Gamma'$ 
with its proper preimage in $\Gamma$, 
see Definition \ref{def:preimage}.
For a vertex $v\in\Gamma$,
$p^v\colon\Gamma\to\Gamma^v$ stands for the
maximal subsequence of blowdowns in $p$ 
that preserve $v$, see Notation \ref{nota:p-v}.
\begin{graph lemma}
\label{graph-lemma} 
Consider a relatively minimal diagram 
\eqref{triangular diagram}
with connected 
weighted graphs  $\Gamma$, $\Gamma_1$ and 
$\Gamma_2$ 
and with birational morphisms 
$p_i\colon\Gamma\to\Gamma_i$, 
where the $\Gamma_i$ are minimal for $i=1,2$. 
Then the following hold.
\begin{enumerate}
\item[(a)] For the  sets 
of branch vertices we have
\[\Br(\Gamma_1)=\Br(\Gamma)\cap \Vert(\Gamma_1) 
= \Br(\Gamma)\cap \Vert(\Gamma_2)
=\Br(\Gamma_2).\]
Furthermore, the degrees 
of vertices in 
$B:=\Br(\Gamma_1)=\Br(\Gamma_2)$
do not change under the contractions 
$p_1$ and $p_2$.
\item[(b)]
Let $\mathcal{C}$ be a connected component  
of $\Gamma\ominus B$.
Then for $i=1,2$ the image $S_i=p_i(\mathcal{C})$
is a (nonempty) segment of $\Gamma_i$.
Conversely, for any segment $S_i$ 
of $\Gamma_i$ the total preimage 
$p_i^*(S_i)$ is a connected component of 
$\Gamma\ominus B$.
The induced correspondence 
$S_1\leftrightsquigarrow S_2$
between the
segments of $\Gamma_1$ and 
of $\Gamma_2$
is a bijection that preserves,   respectively,
the sets of circular segments, 
inner linear segments and 
extremal linear segments.
\item[(c)]
Assume that  $\Br(\Gamma)\neq B$. 
 Then for every vertex 
 $v \in  \Br(\Gamma)\setminus B$ the following hold.
 \begin{itemize}
\item[($c_1$)] 
$v$ is rational with no incident loop
and all branches of $\Gamma$ 
at $v$ are simple.
There are two distinct
branches 
 $W_1$ and $W_2$ of 
$\Gamma$ at $v$ such that $p_1$ 
contracts  $W_1$ 
and $v$ and sends  $W_2$ into
a non-admissible extremal linear segment 
of $\Gamma_1$, 
and symmetrically for $p_2$. 
\item[($c_2$)]
We have $\deg_{\Gamma}(v)\in\{3,\,4\}$.

$(i)$ If $\deg_{\Gamma}(v)=3$, then 
$p^v_i$ contracts at least one branch 
$W_i$ at $v$, $i=1,2$. 

$(ii)$ If $\deg_{\Gamma}(v)=4$, then exactly
two branches of $\Gamma$ at $v$, 
say $W_1$ and $W_1'$,
are contracted in $\Gamma_1$ 
and the other two, say $W_2$ and $W_2'$,
are contracted in $\Gamma_2$.
The branches $W_i, W_i'$, $i=1,2$ 
are linear and 
the  minimal graphs $\Gamma_i$
are linear and non-admissible. 
\footnote{See  Proposition \ref{prop:4branches} 
below
for a more detailed description.}
\end{itemize}
\end{enumerate}
\end{graph lemma}
\begin{proof}
($a$) We start with the following claims.

\smallskip 

\noindent {\bf Claim 1.} \textit{
Let $v\in \Vert(\Gamma)$. 
\begin{itemize}
\item[(i)]
Assume that a simple branch $W_1$ 
of $\Gamma$ at $v$ is
contracted by $p^v_1$. Then $v$ 
is a rational vertex
contracted by 
$p_2$. Furthermore, $p^v_2$ 
induces an isomorphism 
of weighted graphs
$W_1\cong p^v_2(W_1)$, and all branches of 
$\Gamma$ at $v$ are simple.
\item[(ii)]
Let $W_1$ and $W_2$ be 
simple branches  of $\Gamma$ at $v$ 
such that $W_i$ is
contracted by $p^v_i$ for $i=1,2$. 
Then  
$W_1\neq W_2$ and $v$ is contracted in 
both $\Gamma_1$ and $\Gamma_2$. 
\end{itemize}}
\begin{proof} (i) 
All $(-1)$-vertices of $W_1$ are contracted in 
$\Gamma_1$. 
By the assumption of relative minimality
no $(-1)$-vertex of $W_1$
is contracted in $\Gamma_2$. 
Since 
$\Gamma_2$ is supposed to be minimal,
$p_2(W_1)$ contains no $(-1)$-vertex. 
This can only happen provided $p_2$
contracts $v$ 
prior to the contraction of
any vertex of $W_1$.  
Since $p_2$ contracts no $(-1)$-vertex of $W_1$, 
$p^v_2$ contracts no vertex of $W_1$,
and  so $W_1\cong p^v_2(W_1)$ 
as weighted graphs.

Suppose to the contrary that 
there is a non-simple branch $W$ 
of $\Gamma$ at $v$. Then $p^v_2(W)$
is either  a non-simple branch 
of $p^v_2(\Gamma)$ at $v$
different from the simple branch 
$p^v_2(W_1)\cong W_1$, or 
a loop incident to $v':=p^v_2(v)$,
see Lemma \ref{contract-branch}(a). 
It follows that $v'$ is a  branch vertex 
 of $p^v_2(\Gamma)$. Hence, 
$v'$ cannot be contracted  by $p_2$ on the next steps, 
that is, $p_2=p_2^v$ does not contract $v$, 
contrary to the preceding conclusion. 
Now statement (i) follows. 

(ii) By (i) we have $W_1\neq W_2$
and $v$ is contracted in 
both $\Gamma_1$ and $\Gamma_2$.
\end{proof}
\noindent {\bf Claim 2.} \textit{We have 
$\Br(\Gamma_i)=\Br(\Gamma)
\cap \Vert(\Gamma_i)$. }
\begin{proof}
Clearly, $\Br(\Gamma_i)\subset
\Br(\Gamma)\cap \Vert(\Gamma_i)$. 
Let us show the opposite inclusion. 
By symmetry, we can take $i=1$. 
According to Definition \ref{def:vertices}, 
for a non-rational vertex 
$v\in\Br(\Gamma)$ we have 
$v\in \Br(\Gamma_i)$ for $i=1,2$.
Suppose to the contrary that 
there is a rational vertex 
$v\in \Br(\Gamma)\cap \Vert(\Gamma_1)
\setminus\Br(\Gamma_1)$.
Then $\deg_{\Gamma}(v)\ge 3$,
$v$ is not contracted in $\Gamma_1$, and 
$\deg_{\Gamma_1}(v)\le 2$. 
It follows that  $p_1=p_1^v$ 
contracts at least $\deg_{\Gamma}(v)-2>0$ 
simple branches of $\Gamma$ at $v$, 
see Lemma \ref{contract-branch}(b). 
In particular, 
there is a simple branch $W_1$ of 
$\Gamma$ at $v$ 
contracted by $p^v_1$. 
By Claim 1(i), $v$  is contracted in $\Gamma_2$. 
Therefore, also $p_2^v$ contracts at least 
$\deg_{\Gamma}(v)-2>0$ branches 
of $\Gamma$ at $v$. So, there is
 a simple branch $W_2$ of $\Gamma$ at $v$ 
contracted in $\Gamma_2$. Due to Claim 1(ii), 
$p_1$ contracts $v$ too, contrary to 
our assumption that $v\in \Vert(\Gamma_1)$. 
\end{proof}
\noindent {\bf Claim 3.} \textit{We have $\Br(\Gamma_1)
=\Br(\Gamma_2)$. }
\begin{proof}
Assume to the contrary that $\Br(\Gamma_1)
\neq\Br(\Gamma_2)$. 
Using Claim 2, we can consider by symmetry that
there exists a rational vertex
$v\in \Br(\Gamma)\cap\Vert(\Gamma_1)
\setminus \Vert(\Gamma_2)$.  
Since $v$ is contracted in $\Gamma_2$,
some simple branch $W_2$ of $\Gamma$ at $v$
is contracted in $\Gamma_2$. 
According to Claim 1(i), $v$ is contracted in 
$\Gamma_1$.  
This contradicts our choice of $v\in\Vert(\Gamma_1)$. 
\end{proof}
Claims 2 and 3 prove the equalities in ($a$). 
To show the last assertion in ($a$), we recall
that the degree of a branch vertex $v$ 
of $\Gamma$ 
can drop only under
 a contraction of a simple branch of 
 $\Gamma$ at $v$, see Lemma 
\ref{contract-branch}(a). According to Claim 1(i), 
if a simple branch at $v\in B$ is contracted
under $p^v_1$ (resp., $p^v_2$),
then $v\notin \Vert(\Gamma_2)$ 
(resp., $v\notin \Vert(\Gamma_1)$). 
The latter contradicts our choice of 
$v\in B=\Br(\Gamma_1)=\Br(\Gamma_2)$. 
\qed$_{(a)}$.

\smallskip

($b$)  Let  $\mathcal{C}$ 
be a  connected component 
of $\Gamma\ominus B$ and 
$S_i=p_i(\mathcal{C})$ for $i=1,2$.
We proceed with the following claim.

\smallskip

\noindent {\bf Claim 4.} \textit{$S_i$ 
is a (nonempty) connected component of 
$\Gamma_i\ominus \Br(\Gamma_i)$  for $i=1,2$.}

\begin{proof}
Every vertex 
$v\in \Vert(\Gamma)\setminus \Vert(\mathcal{C})$ 
linked to $\mathcal{C}$ 
belongs to $B=\Br(\Gamma_i)$ for $i=1,2$ by (a). 
In particular, 
$v$ is not contracted in $\Gamma_i$, $i=1,2$. 
It follows that there is a decomposition 
$p_i=p_i'\circ p_i''=p_i''\circ p_i'$, where 
$p_i'$ (resp. $p_i''$) contracts only 
vertices from $\Vert(\mathcal{C})$ (resp. from
 $\Vert(\Gamma)\setminus \Vert(\mathcal{C})$). 

Assume on the contrary that, say 
 $p_1'$ 
contracts $\mathcal{C}$. Then $S_2\neq\emptyset$
due to the relative minimality assumption.
 Since $\mathcal{C}$ is contractible, 
so is $S_2=p_2(\mathcal{C})$, see 
Remark \ref{rem:induced-contr}.2. 
The latter contradicts the minimality of $\Gamma_2$.
Thus, $S_i\neq\emptyset$ for $i=1,2$. 

By the argument above,
every vertex of 
$\Vert(\Gamma_i)\setminus\Vert(S_i)$
linked to a vertex of $S_i$ belongs to 
$\Br(\Gamma_i)$.
Since $S_i$ is connected, it follows that $S_i$ 
is a connected component of 
$\Gamma_i\ominus \Br(\Gamma_i)$.
\end{proof}
Let us return to the proof of (b).
Since $B=\Br(\Gamma_i)$, $S_i$ contains 
no branch vertex of $\Gamma_i$. 
According to Claim 4, $S_i$ 
is a segment of $\Gamma_i$. 
To show the converse, take a segment 
$S_i$ of $\Gamma_i$. By definition, $S_i$ 
is a connected component of 
$\Gamma_i\ominus \Br(\Gamma_i)$. 
Therefore, $\mathcal{C}:=p_i^*(S_i)$ is 
a connected component of $\Gamma\ominus B$,
see Remark \ref{rem:preimage}.

Thus, $p_i$ induces 
a one to one correspondence 
between the segments of $\Gamma_i$ 
and the connected components of 
$\Gamma\ominus B$.
This yields a one to one correspondence 
between 
the sets of segments of 
$\Gamma_1$  and $\Gamma_2$. 

The number of tips of $S_i$ linked to 
$\Br(\Gamma_i)$ equals the number
 of vertices of $\mathcal{C}$ linked to $B$.
So, it is the same for $i=1,2$. 
This number equals $0$ if $S_i$ 
($\mathcal{C}$, respectively) 
is a connected component of $\Gamma_i$ 
(of $\Gamma$,  respectively), equals $1$ if $S_i$ 
is a non-isolated extremal linear segment of 
$\Gamma_i$, and equals $2$ if $S_i$ is 
an inner linear segment of $\Gamma_i$.
Therefore, the above correspondence 
preserves the subsets of circular, inner linear 
and extremal linear segments. \qed$_{(b)}$.
\smallskip

($c_1$) 
Every non-rational vertex 
of $\Gamma$ belongs to $B=\Br(\Gamma_i)$, 
see Definition \ref{def:vertices}. 
Given a vertex $v\in \Br(\Gamma)\setminus B$, $v$ 
is rational and there are simple branches $W_1$ and 
$W_2$ of $\Gamma$  at $v$
contracted by $p_1^v$ and $p_2^v$, 
respectively, see Lemma \ref{contract-branch}(b). 
According to Claim 1(i) and (ii), 
$v$ has no incident loop and
is contracted by both $p_1$ and $p_2$.
Besides, $W_1\neq W_2$
and all branches of $\Gamma$ at $v$ 
are simple. 

Since $W_i$ is contracted in 
$\Gamma_i$, it contains no vertex of 
$B=\Br(\Gamma_1)=\Br(\Gamma_2)$, see (a).
Therefore, for $j\neq i$ the image $p_j(W_i)$ 
in $\Gamma_j$ contains no branch vertex
of $\Gamma_j$.  
It follows that $p_j(W_i)$ 
is contained in a segment, say, 
$S_i$ of $\Gamma_j$. 

According to Lemma \ref{lem:contract-segment}(b)
the contractible tree $W_i$ 
has a contractible extremal linear branch, say $L_i$. 
Let $v_i$ be a $(-1)$-vertex of $L_i$. 
By the relative minimality assumption, $v_i$ 
is not contracted in $\Gamma_j$ for $j\neq i$. 
In particular, $p_j(W_i)$ is nonempty.
For $j\neq i$ 
the image $p_j(L_i)$ contains 
a tip $t$ of $\Gamma_j$, see 
Lemma \ref{lem:contract-segment}(a). 
Since $t\in S_i$,
it follows that $S_i$ is an extremal 
linear segment of $\Gamma_j$.

The weight of $v_i$ in $\Gamma$ can 
only increase under the contraction $p_j$, 
and it must increase indeed because 
of the minimality of $\Gamma_j$. Therefore,
the weight of $v_i$ in $\Gamma_j$ is $\ge 0$, 
and so
the segment $S_i$ is not admissible. 
\qed$_{c_1}$
\smallskip

($c_2$) We use the notation from ($c_1$). 
According to Claim 1(ii),
the branches of $\Gamma$
at $v$ contracted by $p_1^v$ and 
those contracted by $p_2^v$ 
are distinct.  In total, 
there are at least 
$2\deg_{\Gamma}(v)-4\le \deg_{\Gamma}(v)$  
contractible simple
branches of $\Gamma$ at $v$, see Lemma 
\ref{contract-branch}(b). 
It follows that $3\le \deg_{\Gamma}(v)\le 4$ and
there are exactly $\deg_{\Gamma}(v)$ (simple) branches
of $\Gamma$ at $v$. 

(i) Suppose that  $\deg_{\Gamma}(v)=3$.
By (a), $p_i^v(v)$
is not a branch point of $\Gamma_i^v$  for $i=1,2$.
Hence $p^v_i$ contracts at least one branch, say $W_i$ 
of $\Gamma$ at $v$.
By the preceding, $W_1\neq W_2$.

(ii) Suppose now that  $\deg_{\Gamma}(v)=4$.
Since $p_i^v$ contracts at least 2 
branches for $i=1,2$
and these 4 branches are distinct,
two branches, say $W_1$ and $W_1'$  
of $\Gamma$
at $v$ are contracted by $p_1^v$ 
and two others, 
say $W_2$ and $W_2'$  
are contracted by $p_2^v$.

According to ($c_1$), for $j\neq i$
the images
$p_j(W_i)$ and $p_j(W_i')$ are contained 
in extremal linear segments of $\Gamma_j$.
Hence, these images are linear chains, and 
their union equals $\Gamma_j$. 
None of these chains contain a branch point of 
$\Gamma_j$, see the proof of ($c_1$). Therefore,
both $\Gamma_1$ 
and $\Gamma_2$ are linear  graphs. 
Since the $(-1)$-vertices
of $W_i$ and $W_i'$ are not contracted in $\Gamma_j$,
their images in $\Gamma_j$ have non-negative weights.
So, $\Gamma_1$ 
and $\Gamma_2$ are not admissible.

Suppose to the contrary that there exists 
a second branch point, say, $v'\neq v$ 
of $\Gamma$,
and let $V_1,\ldots,V_k$, $k\ge 3$ be the branches of 
$\Gamma$ at $v'$. We may assume that $v'\in W_1$
and $W_1'\subset V_1$. 
Then we have $V_2\cup\cdots\cup V_k\subset W_1$. 

Since the branch $W_1'$ is contractible and 
is contracted under $p_1$, it contains a 
$(-1)$-vertex that survives the contraction $p_2$.
Hence the branch $V_1$ of $\Gamma$ at $v'$
cannot be contracted under $p_2$. 

Since $V_2,\ldots, V_k$ are contracted 
in $\Gamma_1$ together with $W_1$,
no $(-1)$-vertex of $V_2\cup\cdots\cup V_k$ 
can be contracted under 
$p_2$.  It follows that there is an 
isomorphism of weighted graphs 
$V_i\cong p_2^{v'}(V_i)$ for $i=2,\ldots,k$. 
Since also $V_1$ is not contracted under 
$p_2^{v'}$, the image of $v'$ in $p_2^{v'}(\Gamma)$
is a branch vertex of degree $k\ge 3$. Hence, 
it cannot be contracted on the 
next step of contraction $p_2$.  It follows that 
$p_2=p_2^{v'}$, and so $v'\in \Br(\Gamma_2)$.
However, by the preceding $\Gamma_2$ is a linear graph.
This gives a contradiction. 
\qed$_{c_2}$
\end{proof}
\begin{remark} For a vertex  
$v \in  \Br(\Gamma)\setminus B$ 
with $\deg_{\Gamma}(v)=3$  
it can happen that both $p_1^v$ 
and $p_2^v$ contract exactly 
one branch  of $\Gamma$
at $v$, as it occurs in Example 
\ref{exa:non-linear domination} below. 
It can also happen that, say
$p_1^v$ contracts a single (non-linear) branch 
of $\Gamma$
at $v$, while two other branches of $\Gamma$
at $v$ are contracted by 
$p_2^v$, see Example \ref{exa:nontip}.
\end{remark}
\begin{example}\label{exa:non-linear domination} 
We present  on Figure \ref{fig:graphs-1} a graph $\Gamma$ 
that dominates two minimal linear graphs $\Gamma_1$ and 
$\Gamma_2$. 
 For $i=1,2$ the birational morphism $p_i$ 
contracts the branch $W_i$ of $\Gamma$ 
and the vertex $v$. We have $B=\emptyset$ 
and $\deg_{\Gamma}(v)=3$.
The simple branch $T$ of $\Gamma$ at $v$ 
that consists of 
a single vertex $a$ is not contractible. In fact, it can 
be replaced by any minimal weighted graph $T$  
linked to $v$ by a single edge.
\begin{figure}[h]\label{fig:graphs}
\begin{tabular}{l@{\hspace{10pt}}l@{\hspace{10pt}}l}
    \pgfdeclarelayer{background}
		\pgfsetlayers{background,main}
    \begin{tikzpicture}[scale=1.0]
      
       \node[] (G) at (2.7,1) {$\Gamma\colon$};
       \node[vertex] (a) at (3.1,1) {};
       \node[] (A) at (3.1,0.7) {\footnotesize $a$};
       \node[] (0) at (3.1,1.3) {\footnotesize $0$};
       \node[vertex] (v) at (4.1,1) {};
       \node[] (V) at (4.1,0.7) {\footnotesize $v$};
        \node[] (-3) at (4,1.3) {\footnotesize $-3$};
        \node[vertex] (b1) at (5.1,1.6) {};
        \node[] (B1) at (5.1,1.3) {\footnotesize $v_1$};
       \node[] (-1) at (5.1,1.9) {\footnotesize $-1$};
        \node[vertex] (b2) at (6.1,2.1) {};
       \node[] (B2) at (6.1,1.8) {\footnotesize $b_1$};
        \node[] (-2) at (6.1,2.4) {\footnotesize $-2$};
        \node[vertex] (c1) at (5.1,0.4) {};
        \node[] (C1) at (5.1,0.1) {\footnotesize $v_2$};
       \node[] (-1) at (5.1,0.7) {\footnotesize $-1$};
       \node[vertex] (c2) at (6.1,-0.2) {};
       \node[] (C2) at (6.1,-0.5) {\footnotesize $b_2$};
       \node[] (-2) at (6.1,0.1) {\footnotesize $-2$};
        
        \node[] (W1) at (6.8,2) {\footnotesize $W_1$};
        \node[] (W2) at (6.8,-0.3) {\footnotesize $W_2$};
     
       \node[] (G1) at (9,2.9) {$\Gamma_1\colon$};
       \node[vertex] (a1) at (9.4,2.9) {};
       \node[] (A1) at (9.4,2.6) {\footnotesize $a$};
       \node[] (1) at (9.4,3.2) {\footnotesize $1$};
       \node[vertex] (b11) at (10.4,2.9) {};
       \node[] (B11) at (10.4,2.6) {\footnotesize $v_2$};
        \node[] (0) at (10.4,3.2) {\footnotesize $0$};
        \node[vertex] (b21) at (11.4,2.9) {};
        \node[] (B21) at (11.4,2.6) {\footnotesize $b_2$};
       \node[] (-2) at (11.4,3.2) {\footnotesize $-2$};
       
       \node[] (G2) at (9,-1) {$\Gamma_2\colon$};
       \node[vertex] (a2) at (9.4,-1) {};
       \node[] (A2) at (9.4,-1.3) {\footnotesize $a$};
       \node[] (1) at (9.4,-0.7) {\footnotesize $1$};
        \node[vertex] (c11) at (10.4,-1) {};
       \node[] (C11) at (10.4,-1.3) {\footnotesize $v_1$};
        \node[] (0) at (10.4,-0.7) {\footnotesize $0$};
        \node[vertex] (c21) at (11.4,-1) {};
        \node[] (C21) at (11.4,-1.3) {\footnotesize $b_1$};
       \node[] (-2) at (11.4,-0.7) {\footnotesize $-2$};
     
      \draw[edge] 
      (a)edge(v) (v)edge(b1) (b1)edge(b2) (v)edge(c1) (c1)edge(c2)
      (a1)edge(b11) (b11)edge(b21) 
      (a2)edge(c11) (c11)edge(c21);
      
      \draw[->](7.1,1.3) -- (8.7,2);
      \node[] (pi1) at (7.9,2) {\footnotesize $p_1$};
      \draw[->](7.1,0.7) -- (8.7,-0.4);
      \node[] (pi2) at (7.9,-0.3) {\footnotesize $p_2$};

         \end{tikzpicture}
\end{tabular}
  \caption{A weighted graph $\Gamma$  
  with branches $W_1$ and $W_2$ at the vertex $v$ that 
 are contracted 
 in $\Gamma_1$ and $\Gamma_2$, respectively. }
  \label{fig:graphs-1}
\end{figure}
\end{example}
\begin{example}\label{exa:4-branches} 
Figure \ref{fig:graphs-next} presents a graph $\Gamma$ 
that dominates two minimal linear graphs $\Gamma_1$ and 
$\Gamma_2$. 
The birational morphism $p_1$ (resp. $p_2$)
contracts the vertices  $b_1$  and $b_2$ (resp. $a_1$ and $a_2$) of $\Gamma$ 
and the vertex $v$. We have $B=\emptyset$ and $\deg_{\Gamma}(v)=4$.

\begin{figure}[h]\label{fig:graphs-next}
\begin{tabular}{l@{\hspace{10pt}}l@{\hspace{10pt}}l}
    \pgfdeclarelayer{background}
		\pgfsetlayers{background,main}
    \begin{tikzpicture}[scale=1.0]
      
       \node[] (G) at (2.7,1) {$\Gamma\colon$};
       
       \node[vertex] (a1) at (3.1,1.7) {};
       \node[] (-1) at (3.1,1.9) {\footnotesize $-1$};
       \node[vertex] (a2) at (3.1,0.4) {};
       \node[] (-1) at (3.1,0.7) {\footnotesize $-1$};
       \node[] (A1) at (3.1,1.3) {\footnotesize $a_1$};
        \node[] (C1) at (3.1,0.1) {\footnotesize $a_2$};
       \node[vertex] (v) at (4.1,1) {};
       \node[] (V) at (4.1,0.7) {\footnotesize $v$};
        \node[] (-3) at (4,1.3) {\footnotesize $-3$};
        \node[vertex] (b1) at (5.1,1.6) {};
        \node[] (B1) at (5.1,1.3) {\footnotesize $b_1$};
       \node[] (-1) at (5.1,1.9) {\footnotesize $-1$};
        \node[vertex] (b2) at (5.1,0.4) {};
        \node[] (B2) at (5.1,0.1) {\footnotesize $b_2$};
       \node[] (-1) at (5.1,0.7) {\footnotesize $-1$};
       \node[] (G1) at (9,2.9) {$\Gamma_1\colon$};
       \node[vertex] (a10) at (9.4,2.9) {};
       \node[] (A11) at (9.4,2.6) {\footnotesize $a_1$};
       \node[] (1) at (9.4,3.2) {\footnotesize $0$};
       \node[vertex] (b11) at (10.4,2.9) {};
       \node[] (B11) at (10.4,2.6) {\footnotesize $a_2$};
        \node[] (0) at (10.4,3.2) {\footnotesize $0$};
       
       \node[] (G2) at (9,-1) {$\Gamma_2\colon$};
       \node[vertex] (a20) at (9.4,-1) {};
       \node[] (A2) at (9.4,-1.3) {\footnotesize $b_1$};
       \node[] (1) at (9.4,-0.7) {\footnotesize $0$};
        \node[vertex] (c11) at (10.4,-1) {};
       \node[] (C11) at (10.4,-1.3) {\footnotesize $b_2$};
        \node[] (0) at (10.4,-0.7) {\footnotesize $0$};

      \draw[edge] 
      (a1)edge(v) (a2)edge(v) (v)edge(b1) (v)edge(b2)
      (a10)edge(b11) 
      (a20)edge(c11); 
      
      \draw[->](7.1,1.3) -- (8.7,2);
      \node[] (pi1) at (7.9,2) {\footnotesize $p_1$};
      \draw[->](7.1,0.7) -- (8.7,-0.4);
      \node[] (pi2) at (7.9,-0.3) {\footnotesize $p_2$};

         \end{tikzpicture}
\end{tabular}
  \caption{A weighted graph $\Gamma$  
  with $4$ branches at the vertex $v$ that 
  are contracted, two by two, in
  $\Gamma_1$ and $\Gamma_2$, respectively.} 
  \label{fig:graphs-next}
\end{figure}
\end{example}
\begin{example}[\rm{see \cite[Example~2.6]{FKZ-graphs}}]
\label{exa:nontip} The graph

\vspace*{5mm}
$$
 \Gamma:\quad\qquad
 \cshiftup{-1}{a}\cshiftdown{-1}
 {b}\hspace{9.8mm}
 \nwlin\cou{-3}{v}\swlin\co{}
\lin\co{-2}\slin\cshiftdown{-2}{e} 
\lin\cou{-3}{v'} \solin \nolin\hspace{9.8mm}
\cshiftup{-1}{c}\cshiftdown{-1}{d}
\quad
$$
\vspace*{10mm}

\noindent admits three different contractions 
$p_i\colon \Gamma\to \Gamma_i$, $i=1,2,3$ 
to the minimal linear graphs 
$\Gamma_1=\Gamma_2:=[[0,0]]$ and 
$\Gamma_3:=[[0,-2]]$, respectively. The morphisms
$p_1$ and $p_2$ contract the non-linear
subgraphs $\Gamma\ominus \{a,b\}$ and
$\Gamma\ominus\{c,d\}$, 
respectively, 
and $p_3$ contracts two linear subgraphs of $\Gamma$ 
with tips $a,b$ and $c,d$,
respectively. 

The contractions $p_1$ and $p_3$
(resp. $p_2$ and $p_3$) fit in a diagram 
\eqref{triangular diagram} that is not relatively minimal, 
while the diagram 
\eqref{triangular diagram} with $p_1$ and $p_2$ is. 
In the latter case we have $B=\emptyset$ 
and $\Br(\Gamma)$ consists of $3$ vertices
of degree $3$, including $v$ and $v'$ 
of weight $-3$ and 
the third vertex, say $v_0$ of weight $-2$. 
The morphism $p^v_1$ contracts the branch 
$\Gamma\ominus \{a,b\}$ of $\Gamma$ 
at $v$, while $p^v_2$ contracts two other branches 
 of $\Gamma$ at $v$ that consist of the tips $a$ and $b$. 
The branch $T$ of 
$\Gamma$ at $v_0$ consisting
of the single vertex $e$ is not contractible. Nevertheless, 
$T$ is contracted by the $p_i$ for $i=1,2$, hence
$e\notin p_1^{-1}(\Gamma_1)\cup 
p_2^{-1}(\Gamma_2)$.
\end{example}
In Proposition \ref{prop:4branches} below
we provide
a description of all possible graphs $\Gamma$ as in
Graph Lemma \ref{graph-lemma}($c_2$)(ii).  
We use the following notation and results.
\begin{nota}
For a sequence of integer numbers 
$k_1,\ldots,k_n$ 
with $k_1,\ldots,k_{n} \ge 2$
we let $[k_1, \ldots, k_n]$ 
be the minus
continued fraction defined inductively via
$[k_1] = k_1$ and $[k_1,\ldots,k_n] 
= k_1 - \frac{1}{[k_2,\ldots,,k_n]}$.
If $[k_1,\ldots, k_n]=m/e$, where
$0< e<m$ and $\gcd(e,m)=1$, then
we let
\begin{equation}\label{eq:boxes}
\qquad\boxo{e/m}\qquad =\quad
[[-k_1,\ldots,-k_n]]\quad\text{ and }\qquad
\quad\boxo{(e/m)^*} \qquad=\quad 
[[-k_n,\ldots,-k_1]].
\end{equation}
For the subgraph 
$L'=[[-k_2,\ldots,-k_n]]$ of  
$L=[[-k_1,\ldots,-k_n]$ we have 
$m = {\rm discr}(L)$ 
and $e = {\rm discr}(L')$, 
see \cite[Lemma 4.4]{FKZ-graphs}.
\end{nota}
\begin{lemma}[\rm{\cite[Proposition 4.9(b)]
{FKZ-graphs}}]
\label{3prop} 
A contractible weighted linear graph $L$
has a unique $(-1)$-vertex, 
and the weights of other 
vertices are $\,\le -2$.

Let now $L=[[a_1,\ldots,a_n]]$ be a weighted 
linear graph with $a_{k}=-1$ for some
$k\in\{1,\ldots,n\}$ and $a_i\le-2$ $\forall i\neq k$. 
If $1<k<n$, then $L$
is contractible if and only if ${\rm discr} (L)=1$ 
or, equivalently,
\begin{equation}\label{eq:equality}
\frac{e_1}{m_1} + \frac{e_2}{m_2}
=1-\frac{1}{m_1m_2}\,,
\end{equation}
where
\[\frac{m_1}{e_1}=[-a_{k-1},\ldots,-a_1]
\quad\mbox{and}\quad
\frac{m_2}{e_2}=[-a_{k+1},\ldots,-a_n]\,\] 
with $e_i, \,m_i>0$ and
$\gcd (e_i,m_i)=1$ for $i=1,2$. 

If $k=1$ (resp. $k=n$), 
then $L$
is contractible if and only if $L=[[-1,-2,\ldots,-2]]$ 
(resp. $L=[[-2,\ldots,-2,-1]]$).
\end{lemma}
\begin{proposition} \label{prop:4branches}$\,$
\begin{itemize}\item[(a)]
Consider a relatively minimal diagram
\eqref{triangular diagram} with a connected 
weighted graph $\Gamma$ and
minimal weighted graphs $\Gamma_1$ and $\Gamma_2$
such that $\Gamma$ has a vertex  
$v \in  \Br(\Gamma)\setminus B$ of degree 4. 
Let  $L_{3-i}= p_i^v(\Gamma)$ for $i=1,2$ be 
the induced weighted linear subgraph
of $\Gamma$ spanned by $W_i, W_i'$ and $v$, 
where the weight of $v$ is replaced by $-1$. 
Then $L_{3-i}$
has the form 
\begin{equation}\label{eq:lin}
L_{3-i}=\qquad\quad
\boxo{(e_{1,i}/m_{1,i})^*}\llin\cou{}{-1}\llin
\boxo{e_{2,i}/m_{2,i}}\llin\cou{v}{-1}\llin
\boxo{(e_{3,i}/m_{3,i})^*}\llin\cou{}{-1}
\llin\boxo{e_{4,i}/m_{4,i}}
\end{equation}
where 
\begin{equation}\label{eq:4-conditions}
\,\,\,\qquad\frac{e_{j,i}}{m_{j,i}} + \frac{e_{j+1,i}}{m_{j+1,i}}
=1-\frac{1}{m_{j,i}m_{j+1,i}}\quad\text{ for }\quad i=1,2 
\quad\text{and}\quad j=1,2,3.
\end{equation}

\item[(b)]
Conversely, let $\Gamma$ 
be a weighted graph 
with a unique branch point $v$
of degree $4$ and with four contractible
 linear branches  at $v$, say, 
$W_1, W_1', W_2, W_2'$.
Supppose that for $i=1,2$, 
after the contraction 
of $W_i$ and $W_i'$ the weight of $v$ 
becomes equal to $-1$ 
and the resulting graph $L_{3-i}$
satisfies \eqref{eq:lin} and \eqref{eq:4-conditions}. 
Then $\Gamma$  fits in 
a relatively minimal diagram 
\eqref{triangular diagram}
with non-admissible minimal linear graphs  
$\Gamma_1$ and 
$\Gamma_2$ 
and with birational morphisms 
$p_i\colon\Gamma\to\Gamma_i$ 
such that $p_i^v$ 
contracts $W_i$ and $W_i'$
for $i=1,2$.
\end{itemize}
\end{proposition}
\begin{proof} (a)
After the contraction  of $W_2$ and $W_2'$ 
under $p_2^v$, 
the weight of $v$ becomes $-1$, because $v$ is 
contracted on the next step, see ($c_1$) 
of Graph Lemma \ref{graph-lemma}.
The image $L_1$ of $\Gamma$ under $p_2^v$
is a linear graph with contractible linear branches 
$W_1$ and $W_1'$ at $v$ that remain unchanged
under $p_2^v$, see $(c_1)$. Since a contractible linear graph 
has a unique $(-1)$-vertex, $L_1$ has the 
desired form \eqref{eq:lin}. 
Both $W_1$ and $W_1'$
are contracted under $p_1$, hence we have 
the equalities in \eqref{eq:4-conditions} 
for $i=1$, $j=1$ and $j=3$. 
By the relative minimality assumption, the contraction 
$p_2$ of $\Gamma$ to a minimal graph $\Gamma_2$ 
does not contract the $(-1)$-vertices of 
$W_1$ and $W_1'$, but changes their weights. 
Therefore, the connected subgraph of $L_1$ 
that connects
the leftmost and rightmost $(-1)$-vertices of $L_1$
is also contractible. By  Proposition \ref{3prop}
 this implies the last equality 
in \eqref{eq:4-conditions}
for $i=1$ and $j=2$. 
By symmetry, the same conclusions
holds for the linear graph $L_2$.

(b) The converse assertion follows easily
from Lemma \ref{3prop}.
\end{proof}
\begin{remark}
Consider a normal affine surface $Y$ 
and its NC-completion $(X,D)$ 
with dual graph $\Gamma(D)$.
Suppose $\Gamma(D)$ satisfies 
the assumptions of 
Proposition \ref{prop:4branches}(a).  
According to Graph Lemma  
\ref{graph-lemma}($c_2$)(ii),
in this case the dual graph of a minimal
NC-completion of $Y$ is linear. 
 By Gizatullin's characterization 
of generic flexibility (\cite{Giz71b}) 
extended by Dubouloz (\cite{Dub04}),
$Y$ is
generically flexible, see the definition in
the  Introduction. 
In particular, $Y$ admits two 
$\Ga$-actions with distinct general orbits. 

Moreover, $Y$ 
admits an NC-completion $(X_0,D_0)$
with 
an irreducible and smooth boundary divisor $D_0$.
Indeed, it suffices to contract the four branches 
of $D$ at the branch vertex $v$ 
of $\Gamma(D)$, see \ref{graph-lemma}($c_2$)(ii). 
If $Y$ is smooth, then by \cite[Proposition~1]{Gi1}
$(X_0,D_0)$
is isomorphic to one of the following pairs:
\begin{equation}\label{eq:Ciz}
(\PP^2,L),\quad (\PP^2,C)\quad\text {and } 
(\mathbb{F}_n, S),\quad n\ge 0,
\end{equation}
where $L$ (resp. $C$) is a line (resp. a conic) 
in $\PP^2$ and $S$ is 
an ample section of the Hirzebruch surface 
$\mathbb{F}_n\to\PP^1$.
By a Danilov-Gizatullin theorem 
\cite[Part II, Theorem~5.8.1]{DG}, 
see also \cite{FKZ-DGthm},
two affine surfaces 
$Y=\mathbb{F}_n\setminus S$ and 
$Y'=\mathbb{F}_m\setminus S'$
are isomorphic if and only if $S^2={S'}^2$. 
Thus, up to isomorphism, 
the smooth affine surfaces that admit 
a minimal completion 
with the dual graph as in Proposition 
\ref{prop:4branches}(a) 
form an infinite countable set. 

Actually, every smooth affine surface 
from Gizatullin's list \eqref{eq:Ciz}, and, more generally, 
 every
normal affine surface $Y$ that can be
completed by a single smooth rational curve $D_0$,
admits also a completion as in 
 \ref{graph-lemma}($c_2$)(ii), that is, 
 with a dual graph $\Gamma$ that has 
 two pairs of contractible 
 linear branches at a single branch vertex $v$ 
 such that the corresponding minimal 
 graphs $\Gamma_i$, $i=1,2$ are linear and fit in a relatively 
 minimal diagram \eqref{triangular diagram}.
Indeed, let $D_0^2=k>0$. Take four distinct points 
$P_1,\ldots,P_4$ on $D_0$. 
Blowing up the $P_j$ and their infinitesimally near points 
successively
we can transform the pair $(X_0,D_0)$ into 
an SNC-pair  $(X,D)$
with $Y=X\setminus D$ such that the weighted dual graph 
$\Gamma(D)$ has a single branch vertex $v$ of degree 4, 
that corresponds to the proper transform 
of $D_0$ on $X$, and four
linear branches $W_1, W_1',W_2,W_2'$ at $v$ of the form
\[W_1\cong W_2=[[-1,-2,\ldots,-2]] \quad
\text{ of length}\quad k \quad \text{ and }\quad
W_1'\cong W_2'=[[-1]],\]
where all four $(-1)$-vertices are linked to $v$. 
After the contraction of all four branches, the weight 
of $v$ becomes equal to $k$. For $i=1,2$, 
after the contraction of two branches 
$W_i$ and $W_i'$ the weight of $v$ 
becomes equal to $-1$; the further blowing of $v$ down 
yields a minimal linear graph $\Gamma_i$. 
Hence, the graph $\Gamma=\Gamma(D)$
verifies the assumptions of 
Proposition \ref{prop:4branches}(a).
\end{remark}
\subsection{Birationally rigid graphs} 
\label{minimality-subsection}
\begin{definition}\label{def:bir-rigid}
A minimal weighted graph $\Gamma_1$ 
is said to be \textit{birationally rigid} if in any 
relatively minimal diagram 
\eqref{triangular diagram} 
with a minimal graph 
$\Gamma_2$, no vertex of 
$\Vert(\Gamma_1)\subset\Vert(\Gamma)$ 
is contracted in $\Gamma_2$.
In other words, 
$\Vert(\Gamma_1)\subset\Vert(\Gamma_2)$ in $\Gamma$.

We say that $\Gamma_1$ is \textit{admissible} 
if all its segments are admissible, 
see Definition \ref{def:admissible}.
\end{definition}
\begin{examples}\label{exa:couples} The graph 
$\Gamma_1=[[0,0]]$
is neither admissible, nor
birationally rigid, see Examples \ref{exa:4-branches} 
and \ref{exa:nontip}. 
The same holds for the circular graph 
$\Gamma_1=((3))$ 
with a single vertex $v$ of weight $3$ and a loop. 
Indeed, there is 
a relatively minimal diagram 
\eqref{triangular diagram}, where  $\Gamma$ is
the circular graph $((-1,-1))$ with 
two vertices $v_1$ and $v_2$. 
It dominates, for $i=1,2$,
the minimal graph
$\Gamma_i=((3))$  with a single vertex 
$v_{3-i}$
via the blowdown $p_i$ of $v_i$. 
\end{examples}
\begin{lemma}\label{lem:isoms} A minimal weighted 
graph $\Gamma_1$ is 
birationally rigid if and only if for 
every relatively minimal diagram
\eqref{triangular diagram} 
with a minimal graph $\Gamma_2$,
the birational morphisms $p_1$ and $p_2$ 
are isomorphisms.
\end{lemma}
\begin{proof}
The ``if'' part is immediate. Let us prove the ``only if'' part. 
Suppose on the contrary that $\Gamma_1$ is birationally rigid
and there exists a $(-1)$-vertex 
$v\in\Vert(\Gamma)\setminus\Vert(\Gamma_2)$. 
By the assumption of birational rigidity we have 
$v\in\Vert(\Gamma)\setminus\Vert(\Gamma_2)
\subset \Vert(\Gamma)\setminus\Vert(\Gamma_1)$.
It follows that $v$ is blown  down by 
both $p_1$ and $p_2$.
The latter contradicts the  relative 
minimality assumption.
Thus, $p_2$ is an isomorphism, and so 
$\Gamma\cong\Gamma_2$ is minimal. 
Therefore, $p_1$
is an isomorphism too. 
\end{proof}
We also have the following criterion 
of  birational rigidity;
cf. \cite[Corollaries A.3 and A.4]{FZ}. 
Notice that  our birationally rigid graphs correspond 
to absolutely minimal graphs of \cite{FZ}.
For the reader's convenience we provide a proof. 
\begin{proposition}\label{abs-graph} 
A minimal weighted graph $\Gamma_1$ 
 is birationally rigid 
if and only if it is admissible. 
\end{proposition}
\begin{proof} Assume to  the contrary that  
all segments of $\Gamma_1$ 
are admissible, but
$\Gamma_1$ is not  birationally rigid. Then there is
a relatively minimal diagram \eqref{triangular diagram} 
with a minimal graph $\Gamma_2$ and a 
vertex of $\Gamma_1$ that is 
contracted in $\Gamma_2$. 
Since $p_2$ is not an isomorphism, there is 
a $(-1)$-vertex $v\in\Vert(\Gamma)$
 blown down the first under the morphism 
$p_2\colon\Gamma\to \Gamma_2$. 
Due to the relative minimality of  
diagram \eqref{triangular diagram}, 
$v$ is not contracted under 
$p_1\colon\Gamma\to\Gamma_1$, 
that is,  $v\in \Vert(\Gamma_1) \subset\Vert(\Gamma)$. 

Since $v$ is at most linear vertex of $\Gamma$
and $\Br(\Gamma_1)\subset \Br(\Gamma)$,
there is a segment $S_1$ 
of $\Gamma_1$ that contains $v$.
Let $w_1(v)$ be the weight of $v$ in $\Gamma_1$ 
and $w(v)$ be its weight in $\Gamma$.  
Clearly,  $-1=w(v)\le w_1(v)$. 
If the admissible segment $S_1$  
is either linear of any length, 
or circular with at least two vertices, 
then we have $-1=w(v)\le-2$, 
which gives a contradiction. 

Let now $S_1$ be a circular segment of $\Gamma_1$ 
with a single 
vertex $v$. Since $S_1$ is admissible
we have $w_1(v)\le 2$, 
see Definition \ref{def:admissible}. Let $S$ 
be the total preimage
of $S_1$ in $\Gamma$. 
Assume that $S$ 
has at least two vertices. 
Since $v$ is at most linear in $\Gamma$,
 no outer blowup was done at $v$ under $p_1^{-1}$.
If the loop of $S_1$ at
$v$
was blown  under $p_1^{-1}$ up, 
then again $-1=w(v)\le -2$ hold, a contradiction.
 Otherwise $S$ contains a single vertex 
$v$ and has an incident 
loop in $\Gamma$, hence $v$ cannot be contracted 
in $\Gamma_2$.
Once again, this gives  a contradiction. 

To show the converse assume that $\Gamma_1$ 
has a non-admissible
segment $S_1$.  
Both $\Gamma_1$ and $S_1$ are minimal, 
so they are not contractible. 
Being non-admissible  $S_1$ has a vertex $v$ of
weight $a$, where $a\ge 0$ if $S_1$ is different 
from a circular segment with a single vertex and 
$a\ge 3$ otherwise. Let us construct a relatively 
minimal diagram \eqref{triangular diagram} 
with a minimal graph $\Gamma_2$ such that 
$v$ is contracted in $\Gamma_2$.

Blowing $\Gamma_1$ up 
successively we can reduce the weight of $v$ to $-1$ 
while keeping it at most linear. In more detail, first 
we blowup at an edge $ [u,v]$
incident to $v$ in $\Gamma_1$ unless $\Gamma_1$ 
consists of a single isolated vertex $v$; in the latter 
case we start with an outer blowup at $v$. Let $e_1$ 
be the $(-1)$-vertex created on the first step. On the second 
step, we blowup the edge $[e_1,v]$ creating 
a new vertex $e_2$, etc. On
 the $i$th step we blowup 
the edge $[e_{i-1},v]$ by adding a new vertex $e_i$. 
After 
$k>0$ successive blowups, where $k=a+1$  if $S_1$ 
is different from a circular segment with a single vertex 
and $k=a-2$ otherwise, we obtain a weighted graph 
$\Gamma$ with  the vertices  $e_1,\ldots,e_k$ 
appearing in this order under our procedure.
For $i=1,\ldots,k-1$ the weight of $e_i$ in $\Gamma$ 
equals $-2$ while $e_k$ and $v$ are 
at most linear $(-1)$-vertices of $\Gamma$. 
Notice that $e_k$ is the unique $(-1)$-vertex 
of $\Gamma$ blown down under the contraction 
$p_1\colon \Gamma\to \Gamma_1$ of $e_k,\ldots,e_1$ 
in this order. 

 Starting with the blowdown of $v$ in $\Gamma$ 
 we continue to blow down
the new at most linear $(-1)$-vertices 
until we reach 
a minimal graph $\Gamma_2$; this yields a birational 
morphism $p_2\colon \Gamma\to \Gamma_2$. 
The graph $\Gamma$ and the minimal graphs 
$\Gamma_i$, $i=1,2$ along with the contractions
$p_i\colon\Gamma\to\Gamma_i$ form 
a diagram \eqref{triangular diagram} 
(cf. e.g. Example \ref{exa:couples}). 

 We claim that this diagram is relatively minimal.
 Indeed, after the first blowdown of $v$ under $p_2$ 
 the weight
 of $e_k$ becomes $0$ and remains non-negative 
 during the successive blowdowns which form $p_2$. 
 Thus,  the only $(-1)$-vertex  $e_k$ of $\Gamma$ 
 blown under $p_1$ down is not blown down under $p_2$. 
This proves our claim. Since the vertex $v$ of $\Gamma_1$ 
 is contracted in $\Gamma_2$, the graph
$\Gamma_1$ is not birationally rigid, 
see Definition \ref{def:bir-rigid}. 
\end{proof}
Summarizing we get the following proposition.
\begin{proposition}\label{cor:graph} Assume 
we are given a relatively minimal diagram 
 \eqref{triangular diagram}  with minimal weighted 
 graphs $\Gamma_1$ and $\Gamma_2$ dominated 
 by a graph $\Gamma$. 
 Then the following hold.
\begin{enumerate}\item[$1.$] The birational map 
$\phi:=p_2\circ p_1^{-1}\colon \Gamma_1
\dasharrow\Gamma_2$ 
induces
\begin{itemize}
\item[(a)] a bijection between the sets 
of branch vertices 
$\Br(\Gamma_1)\simeq \Br(\Gamma_2)$;
\item[(b)] a bijection between
the sets of segments of $\Gamma_1$ and $\Gamma_2$
that preserves the subsets of  admissible segments, 
of inner (extremal, respectively) linear segments 
and of circular segments;
\item[(c)] isomorphisms between every pair of 
corresponding admissible linear (resp., admissible circular) 
segments;
\item[(d)] an inner birational transformation between 
every pair of corresponding non-admissible inner
linear segments;
\item[(e)]  a birational transformation between every pair 
of corresponding non-admissible extremal linear segments.
\end{itemize}
\item[$2.$]  If the $\Gamma_i$ 
in diagram \eqref{triangular diagram}  
are circular graphs, 
then $\Gamma$
is also circular and the birational morphisms $p_i$ are inner. 
\item[$3.$]  If all extremal linear segments in either 
$\Gamma_1$ or $\Gamma_2$ 
are admissible, then the birational transformation 
$\phi=p_2\circ p_1^{-1}\colon\Gamma_1\dasharrow\Gamma_2$ 
is inner  and induces a bijection between 
the sets of tips 
$\Tip(\Gamma_1)\simeq \Tip(\Gamma_2)$.
\item[$4.$] If $\phi$ is inner and the $\Gamma_i$ 
are linear graphs, then  also $\Gamma$ is a linear graph. 
\item[$5.$] If $\phi$ is inner and the $\Gamma_i$ 
are linear graphs with two vertices, then 
the $p_i\colon\Gamma\stackrel{\cong}
{\longrightarrow}\Gamma_i$ are isomorphisms 
of weighted graphs. 
\end{enumerate}
\end{proposition}
\begin{proof}
Statement 1(a) follows from Graph 
Lemma~\ref{graph-lemma}($a$). 
By Graph Lemma~\ref{graph-lemma}($b$)
we may restrict $\phi$ to any segment of 
$\Gamma_1$ extended by its incident edges 
and adjacent branch vertices. 
Such a restriction includes all the blowups and 
blowdowns in $\phi$ that happen on this segment, 
its incident edges 
and their successive images. 
It yields again a relatively minimal diagram 
\eqref{triangular diagram} 
with minimal dominated graphs. 
Thus, 1(b) resp. 1(c) follows from Graph 
Lemma~\ref{graph-lemma}($b$) resp. from
Lemma \ref{lem:isoms}
and Proposition~\ref{abs-graph}. 

To prove  1(d) we have to show that a relatively 
minimal birational transformation 
between non-admissible inner
linear segments is inner. 
Assuming the contrary, the connected component of 
$\Gamma\setminus B$ corresponding to these segments 
contains a branch vertex 
$v\in \Br(\Gamma)\setminus B$.  Then by 
Graph Lemma~\ref{graph-lemma}($c_1$) 
these segments are 
extremal,  contrary to our assumption.

See \cite[Lemma 2.7]{FKZ-graphs} for statement 2.  
Statement 3 is immediate from 1(c), 1(d) and statement 2.  
 To show statement 4
 it suffices to observe that an inner blowup cannot create 
 a new branch point. 
 
To show statement 5 suppose 
 that the linear segment 
 $\Gamma$ in diagram \eqref{triangular diagram} contains 
 more than 2 vertices. Since $p_i$ is inner,
 it cannot contract 
 a tip of $\Gamma$. Hence, it contracts all vertices 
 in $\Vert(\Gamma)\setminus\Tip(\Gamma)$. 
Therefore, there is a  $(-1)$-vertex 
$v\in\Vert(\Gamma)\setminus\Tip(\Gamma)$ contracted 
in both $\Gamma_1$ and $\Gamma_2$. 
The latter
 is impossible since our  diagram 
 \eqref{triangular diagram} is relatively minimal. 
 Hence $\Gamma$ contains just $2$ vertices, and so 
 the $p_i$ are isomorphisms of weighted graphs. 
\end{proof}
\begin{remark} Given a relatively minimal diagram 
\eqref{triangular diagram}, where $\Gamma_1$ 
and $\Gamma_2$ are minimal 
non-admissible linear graphs, the morphisms $p_i$ 
do not need 
to be inner, in general, 
see e.g. Example \ref{exa:non-linear domination} and 
\cite[Examples 2.6 and 3.30]{FKZ-graphs}.
\end{remark} 
Proposition~\ref{cor:graph} leads to the following theorem.
\begin{theorem}\label{th:bir=inn}
Let $(X,D)$ be an NC-pair with minimal dual graph 
$\Gamma(D)$ such that all extremal linear segments 
of $\Gamma(D)$ are admissible.
Then $\Bir(X,D)=\Inn(X,D)$.
\end{theorem}
\begin{proof}
By Lemma \ref{lem:rel-min-pairs} any $\Phi\in\Bir(X,D)$ 
fits in a relatively minimal commutative diagram 
\eqref{diagr:bir-pairs} of NC-pairs.
Hence $\Phi=P_2\circ P_1^{-1}$, where $P_i^{-1}\colon
 (X,  D)\dasharrow (\widetilde X,\widetilde D)$ for $i=1,2$
  is composed of a sequence of blowups 
of $X$ in points of $D$ and infinitesimally near 
points followed by an isomorphism. Let 
$p_i\colon\Gamma(\widetilde D)\to\Gamma(D)$ 
be the birational morphism of dual graphs 
induced by $P_i$,  $i=1,2$. 
 Then the induced birational transformation  
  of dual graphs $\phi=p_2\circ p_1^{-1}\colon 
  \Gamma(D)\dasharrow\Gamma(D)$ 
  (see Definition \ref{def:dual}) 
 fits in a relatively minimal diagram 
 \eqref{triangular diagram}.
 According to Proposition~\ref{cor:graph}  
 $\phi$ is an inner birational transformation of 
 weighted graphs. 
 Hence $\Phi$ is an inner birational transformation 
 of NC-pairs, i.e. $\Phi\in \Inn(X,D)$. 
\end{proof}
\section{$\Ga$-actions, affine rulings and dual graphs}
\label{sec:actions}
The following lemma is well known. 
\begin{lemma}[{\rm\cite[Corollary 1.2]{FZ-LND}}] 
\label{lem: Ga-actions}
If an affine variety $Y$ of dimension $\ge 2$ admits 
a nontrivial $\Ga$-action, then 
$\Aut (Y)$ contains a connected infinite-dimensional 
commutative unipotent ind-subgroup. 
In particular, it contains algebraic unipotent subgroups 
of arbitrarily large dimensions. 
\end{lemma}
\begin{proof} Indeed, let $H\subset \Aut^{\0} (Y)$ 
be the one-parameter unipotent subgroup
which corresponds to the given $\Ga$-action. 
Then $H$ can be written as $H=\exp(\K\partial)$, 
where $\partial$ is a nonzero locally nilpotent derivation 
of the algebra $\mathcal{O}(Y)$, see \cite{Fr}. 
The  ring of invariants $\ker\partial=\mathcal{O}(Y)^H$ 
is an infinite dimensional vector space, and 
$\exp((\ker\partial)\partial)$ is 
a connected infinite-dimensional 
commutative unipotent subgroup
of $\Aut^{\0} (Y)$. For any finite-dimensional 
subspace $V\subset \ker\partial$, $\exp(V\partial)$ is
an algebraic unipotent subgroup of $\Aut^{\0} (Y)$.
\end{proof}
\begin{remark} 
The assumption in Lemma 
\ref{lem: Ga-actions} that $Y$ is affine is important.  
Indeed, an open surface  $Y=\PP^1\times \AA^1$
admits an effective $\Ga$-action, 
and at the same time
$\Aut (Y)$ is an algebraic group;
see \cite{La} 
for a description of $\Aut (Y)$. 
\end{remark}
Recall that an {\em $\A^1$-fibration}, or an 
{\em affine ruling}, on a normal affine 
surface $Y$ is a morphism $Y\to B$ to a smooth 
curve with a general fiber 
isomorphic to the affine line $\A^1$.
The following lemma is well known, see e.g.
 \cite[Lemma 1.6]{FKZ-deformations}; cf.  
 the proof of  Proposition \ref{affine-fibration}.
\begin{lemma}\label{lemstd}
Let $Y$ be a normal affine surface.
Given an  $\A^{1}$-fibration $\pi:Y\to B$ 
over a smooth affine curve 
$B$ there exists an  SNC-completion $(X_0,D_0)$ 
of $Y$
with the following properties:
 \begin{itemize}
 \item  $\pi$ extends to a $\PP^1$-fibration 
 $\bar\pi:X_0\to \bar B$ over a smooth 
completion $\bar B$ of $B$;
 \item there is a unique component $S$ of $D_0$ 
 with $S^{2}=0$ which is a section of $\bar\pi$;
 \item all other components of $D_0$ are smooth
  rational curves contained in fibers of $\bar\pi$;
 \item the dual graph $\Gamma(D_0)$ is a minimal tree;
 \item the $(0)$-vertices of $\Gamma(D_0)$ different 
 from $S$ are tips;  
 they correspond to 
 the irreducible fibers of $\bar\pi$ contained in
$D_0$.
\end{itemize}
\end{lemma} 
The following proposition proves Corollary \ref{cor}. 
The equivalence (i)$\Leftrightarrow$(ii) 
is contained in \cite[Lemma 1.6]
{FZ-LND}; see also \cite[Remark 1.7]{FKZ-deformations}. 
For convenience of the reader 
we provide a  proof.
\begin{proposition}\label{affine-fibration} 
Let $Y$ be a normal affine surface 
and $(X,D)$ be a minimal NC-completion of $Y$.  
Then the following conditions are equivalent:
\begin{enumerate}\item[(i)] $Y$
admits an effective $\Ga$-action; \item[(ii)] $Y$ admits 
an $\A^1$-fibration over a smooth affine curve; 
\item[(iii)] the dual graph $\Gamma(D)$ has a non-admissible 
extremal linear segment. 
\end{enumerate}
\end{proposition}
\begin{proof} Assuming (iii), $\Gamma(D)$ 
 is birationally equivalent to a
graph $\Gamma_0$
with a rational ($0$)-vertex $v$ of degree 1, see e.g. 
\cite[Examples 2.11]{FKZ-graphs}. 
By Proposition~\ref{prop:transf-repr}
we can replace the original NC-completion $(X,D)$ 
with a new one $(X_0,D_0)$ 
such that $\Gamma(D_0)=\Gamma_0$. 
The vertex $v$ corresponds
 to a smooth rational component, say, 
$C_0$ of $D_0$ with ${C_0}^2=0$. 
Resolving singularities of $X_0$ 
we obtain a smooth projective surface $\tilde X_0$ 
with the exceptional divisor $E$ and 
a smooth rational curve $C_0'$ 
disjoint with $E$ (namely, the proper transform of $C_0$ 
on $\tilde X_0$) such that ${C_0'}^2=0$. 
Indeed, the singular points of $X_0$ being points of 
$Y=X_0\setminus\Supp D_0$ 
the resolution does not affect the divisor $D_0$. 
Such a curve $C_0'$ on $\tilde X_0$ is a reduced member 
of a linear pencil of rational curves
with no base point, 
see e.g. \cite[Proposition~V.4.3]{BHPV}. 
This pencil defines a $\PP^1$-fibration 
$\tilde\pi\colon \tilde X_0\to \bar B$ 
over a smooth projective curve
$\bar B$ such that $C_0'$ is a reduced fiber of  
$\tilde\pi$. Since $E\cdot C_0'=0$, 
each connected component of $E$ is properly 
contained in a fiber of $\tilde\pi$. 
The contraction of $E$  yields a $\PP^1$-fibration 
$\pi_0\colon X_0\to\bar B$ with a reduced fiber $C_0$
that lies in the smooth locus of $X_0$. 

The surface $Y=X_0\setminus\supp D_0$ being affine, 
$\supp D_0$ is connected and supports an effective 
ample divisor, 
see Lemma \ref{lem:Hartogs}(a). 
Furthermore, $Y$ 
contains no complete curve, 
in particular, no entire fiber of $\pi_0$. 
It follows that $\supp D_0\neq C_0$.
Since $\deg(v)=1$, $C_0$ meets transversally
 just one component
$S\neq C_0$ of $D_0$. Since $C_0\cdot S=1$, $S$ is 
a section of $\pi_0$. Hence 
$\pi_0|_Y\colon Y\to B$ 
is an $\A^1$-fibration over 
a smooth affine curve  
$B\subset \bar B\setminus\{\pi_0(C_0)\}$. 
This proves the implication (iii)$\Rightarrow$(ii). 

Assume now that (ii) holds. By Lemma \ref{lemstd} 
the $\A^1$-fibration $\pi\colon Y\to B$ over 
a smooth affine curve  $B$ admits an extension 
$\bar\pi\colon X_0\to\bar B$ 
to an SNC-completion $(X_0,D_0)$ of $Y$, 
where $\bar B$ is a smooth completion of $B$. 
The places at infinity of general $\A^1$-fibers of $\pi$ 
lie on a section, 
say, $S$ of $\bar \pi$. 

A contraction of at most linear 
$(-1)$-vertex of $\Gamma(D_0)$ 
corresponds to a contraction of an exceptional 
$(-1)$-component of $D_0$ preserving the NC-condition. 
Every branch $W$ of $\Gamma(D_0)$ 
at the vertex $[S]$ corresponds to 
the intersection $D_0(P):=D_0\cap \bar\pi^{-1}(P)$
 for a point $P\in\bar B$. 
Contracting subsequently at most linear 
$(-1)$-vertices of $W$ we transform the divisor 
$D_0(P)$ into a minimal divisor, say, $C_P$.
If $P\in \bar B\setminus B$, then $D_0(P)$
coincides with the full fiber 
$ \bar\pi^{-1}(P)$ and the corresponding 
divisor $C_P$ is a reduced $\PP^1$-curve
 that corresponds to  
a $(0)$-vertex of the dual graph. We have 
$C_P\cdot S=1$.

After all such birational transformations 
in the fibers of $\bar\pi$
that carry components of $D_0$,
 it might happen that $[S]$ becomes
an at most linear $(-1)$-vertex of 
the resulting dual graph. 
In this case we perform 
an elementary transformation at 
a point of a full fiber 
$C_P$ with $P\in \bar B\setminus B$, 
see Remark \ref{rem:up-and-down}. 
This changes the weight of $[S]$ by one. 
In this way we can arrive finally
at a minimal NC-completion $(X_0',D_0')$ 
of $Y$ along with 
a $\PP^1$-fibration $\bar\pi'\colon X_0'\to \bar B$  
that restricts to $\pi$ on $Y$.
The birationally equivalent minimal dual graphs 
$\Gamma_1=\Gamma(D)$ and 
$\Gamma_2=\Gamma(D_0')$
fit in a relatively minimal diagram 
\eqref{triangular diagram}, 
see Proposition \ref{prop:domination}.  

For $P\in \bar B\setminus B$ the $\PP^1$-component 
$C_P$ of $D_0'$
 corresponds to 
a ($0$)-vertex $v$ of the dual graph 
$\Gamma(D_0')$ that 
is a tip of a non-admissible extremal linear segment $L'$
of $\Gamma(D_0')$. By Proposition \ref{cor:graph}.1(b), 
$L'$ corresponds to 
a non-admissible extremal linear 
segment $L$ of $\Gamma(D)$.
This proves the equivalence  (iii) $\Leftrightarrow$ (ii).

Suppose that (i) holds, that is,  $Y$ admits an effective 
$\Ga$-action. Let $\delta$ be the nonzero locally nilpotent 
derivation on 
$\cO_Y(Y)$ generating the given $\Ga$-action. 
The subalgebra  $\ker\delta\subset \cO_Y(Y)$ of $\Ga$-invariants is 
finitely generated, 
see \cite[Lemma 1.1]{Fie}, and $B=\Spec(\ker\delta)$
is a smooth affine curve.
The embedding $\ker\delta\hookrightarrow \cO_Y(Y)$ 
yields an $\A^1$-fibration
$Y\to B$ along the orbits of the $\Ga$-action.
Thus, (ii) holds.

 Conversely, suppose that (ii)  holds, and let $\pi\colon Y\to B$  
 be an $\A^1$-fibration 
 over a smooth affine curve $B$. Consider
 as before an SNC-completion $(X_0,D_0)$ of $Y$ and 
 an extension $\bar\pi\colon X_0\to\bar{B}$ to 
 a $\PP^1$-fibration on $X_0$ with a section $S\subset D_0$. 
 Shrinking 
the base $B$ suitably we obtain a principal Zariski open cylinder 
$U\cong Z\times\A^1$ on $Y$ with an affine base $Z\subset B$, 
where $U=\bar\pi^{-1}(Z)\setminus S$. 
By \cite[Proposition 3.1.5]{KPZ} there is an effective 
$\Ga$-action on $Y$. 
This shows the equivalence  (i) $\Leftrightarrow$ (ii).  
\end{proof}
\begin{remark}\label{rem:tree} Under the assumptions 
of Proposition \ref{affine-fibration}  suppose that $\Gamma(D)$
has a non-admissible extremal linear segment. 
We claim that in this case 
$\Gamma(D)$ is a tree with at most one non-rational 
vertex. 
Indeed, let $(X_0,D_0)$ and $S$ be 
as in the proof of the proposition. 
Then $\Gamma(D_0)$ is a tree with a root $[S]$
whose branches $C_P$ at $[S]$ are trees with only rational vertices.
Since $\Gamma(D)$ and $\Gamma(D_0)$ 
are birationally equivalent, the same holds for $\Gamma(D)$ as well. 
\end{remark}
\section{Main results}\label{sec:main}
In this section we prove our main Theorem \ref{intro-fin-dim}.
In addition, we show that $\Bir(X,D)=\Inn(X,D)$ if 
and only if the surface $Y=X\setminus \supp D$ admits 
no $\Ga$-action. 
In the latter case, the identity component 
$\Aut^\0(X)$
is an algebraic torus.
The next theorem yields 
Theorem \ref{intro-fin-dim} from the Introduction.
\begin{theorem}\label{thm:Ga-inn}  
 Let $(X,D)$ be a minimal NC-completion of a normal 
 affine surface $Y$. Then the following are equivalent:
 \begin{enumerate}[(a)]
 \item all extremal linear segments of the dual graph 
 $\Gamma(D)$ are admissible;
 \item $Y$ admits no effective $\Ga$-action;
 \item $\Bir(X,D)=\Inn(X,D)$;
 \item $\Aut^{\0}(Y)$ is an algebraic torus 
 of rank $\le 2$.
 \end{enumerate} 
\end{theorem}
\begin{proof}
Properties (a) and (b) are equivalent by 
Proposition~\ref{affine-fibration}. 
The implication (a)$\implies$(c) follows from 
Theorem~\ref{th:bir=inn}.

To deduce the implication (c) $\implies$ (b)  
assume that (c) holds. If 
on the contrary $Y$ admits 
an effective $\Ga$-action, then 
$\Bir(X,D)=\Aut(Y)$ contains connected 
algebraic subgroups 
of arbitrary dimension, see 
Lemma~\ref{lem: Ga-actions}. 
This contradicts the fact that
every connected algebraic subgroup $G$ 
of $\Bir(X,D)=\Inn(X,D)$ 
is contained in the algebraic group 
$\Aut^{\0}(X,D)$, see 
Proposition \ref{prop:ind-sbgrp}. 
Thus, the conditions 
(a), (b) and (c) are equivalent.

The implication $(d)\implies(b)$ is immediate. 
Suppose now that one of the equivalent 
conditions (a)--(c) is fulfilled. 
By (c),  $\Aut(Y)=\Bir(X,D)=\Inn(X,D)$. 
By Proposition \ref{prop:ind-sbgrp} 
the closed connected ind-subgroup
$\Aut^{\0}(Y)=\Inn^{\0}(X,D)$ of $\Bir(X,D)$
coincides with the connected affine 
algebraic subgroup 
$\Aut^{\0}(X,D)$. Due to (b), the latter group
contains no $\Ga$-subgroup, 
hence no unipotent element. Therefore, $\Aut^{\0}(Y)$
is an algebraic torus. 
 This proves the converse implication $(b)\implies(d)$. 
\end{proof}
Combining Theorem \ref{thm:Ga-inn}, 
Proposition \ref{affine-fibration} 
and its proof we arrive at the following conclusion. 
It implies Corollary 
\ref{cor} from the Introduction;  
cf. \cite[Proposition 3.34]{FKZ-graphs}. 
\begin{corollary}\label{cor:bir=inn} 
For  a normal affine surface
$Y$ the following are equivalent.
\begin{itemize}
\item[(i)] $Y$ admits an effective $\Ga$-action; 
\item[(ii)] $Y$ admits an $\AA^1$-fibration 
over a smooth affine curve; 
\item[(iii)] for every minimal NC-completion 
$(X,D)$ of $Y$ 
the dual graph $\Gamma(D)$ has 
a non-admissible extremal segment;
\item[(iv)] there exists a minimal NC-completion 
$(X_0,D_0)$ of $Y$ 
such that $\Gamma(D_0)$ has a $(0)$-vertex 
of degree $1$;
\item[(v)] $\Bir(X,D)\neq\Inn(X,D)$ for every minimal 
NC-completion $(X,D)$ of $Y$;
\item[(vi)] there exists a minimal NC-completion $(X_0,D_0)$ 
of $Y$ such that $\Bir(X_0,D_0)\neq\Inn(X_0,D_0)$.
\end{itemize}
\end{corollary}
\begin{example}\label{exa:plane-curve}  
Consider the affine surface 
$Y=\PP^2\setminus \supp D$, where $D$ is a nonzero 
reduced effective divisor on $\PP^2$ with only nodes 
as singularities. 
Let us show, 
as a simple application 
of  Corollary \ref{cor:bir=inn}, that $Y$ 
is rigid if and only if $\deg(D)\ge 3$. 

First of all, consider the case $\deg(D)\le 2$. 
In this case $\supp D$ 
is either a projective line, or a pair of lines, 
or a smooth conic.
 In all three cases there exists 
 a pencil  of  conics $\mathcal{L}$ 
 on $\PP^2$ which includes $D$ 
 (resp. $2D$ in the former case) 
 and has a unique base point. 
 The restriction $\mathcal{L}|_Y$ 
 yields an $\A^1$-fibration over $\A^1$. 
 Thus,  $Y$ admits an effective 
 $\Ga$-action, and so is not rigid, 
 see Corollary \ref{cor:bir=inn}(i)--(ii). 

Suppose further that $\deg(D)\ge 3$. 
Using the Bezout theorem and 
analyzing separately the cases where 
$D$ has 1, 2 and at least 3 
components, one can easily conclude 
that every component $C$ of $D$ 
corresponds either to a branch vertex 
of $\Gamma(D)$, or to 
a vertex sitting on a cycle. Anyway, 
$\Tip(\Gamma(D))=\emptyset$.
Therefore, 
$\Gamma(D)$ has no 
extremal linear segment. 
By Corollary \ref{cor:bir=inn} 
this implies the rigidity. 

As an alternative proof one can use the fact 
that $\bar k(Y)=-\infty$ provided $Y$ 
admits an effective $\Ga$-action, 
where $\bar k$ stands for the logarithmic 
Kodaira dimension. 
Since $D$ is nodal one has 
$K_{\PP^2}+D=(\deg(D)-3)H$, where 
$H\in\Pic(\PP^2)$ is the class of a line. 
Therefore, $\bar k(Y)=-\infty$ 
if and only if $\deg(D)\le 2$.

A more direct approach is as follows 
\footnote{We thank 
Shulim Kaliman, who
suggested this approach.}. 
Suppose  $Y$ admits an 
effective $\Ga$-action. Consider an 
SNC-completion$(X',D')$ as in Lemma \ref{lemstd}.
Since $X$ is rational, the dual graph 
$\Gamma(D')$ is a tree with 
only rational vertices, see Remark \ref{rem:tree}. 
The latter properties of 
$\Gamma(D')$ are preserved under 
birational transformations between NC-pairs. 
Hence $\Gamma(D)$ is also a tree with 
only rational vertices. 
By the Bezout theorem this implies 
as before that $\deg(D)\le 2$. 
\end{example}
\section{Appendix: Minimal models 
of weighted graphs}\label{app}
All weighted graphs 
considered in this section 
are supposed to have 
only rational vertices. 
Among these graphs, we classify 
those that have a unique up 
to isomorphism minimal model, 
and show that any other weighted graph 
with only rational vertices
has an infinite number of 
non-isomorphic minimal models, 
see Theorem \ref{prop:classification}. 
 Recall that $[[w_1,\ldots,w_n]]$ (resp. $((w_1,\ldots,w_n))$)
 represents the linear (resp. circular) graph with 
the ordered (resp. cyclically ordered) 
sequence  of weights $(w_1,\ldots,w_n)$. 
We abbreviate by $a_k$ the sequence 
$(\underbrace{a,\ldots,a}_{k})$.

The following corollary of Propositions 
\ref{abs-graph} and 
\ref{cor:graph}.1(a)--(c) is immediate. 
\begin{corollary}\label{cor:min-model} Let $\Gamma$  
be a birationally rigid minimal weighted graph. 
Then up to isomorphism $\Gamma$ is 
a unique minimal graph 
in its class of birational equivalence.
\end{corollary}
We provide below examples 
of non-birationally rigid minimal 
 weighted graphs with a unique 
minimal model, see Proposition \ref{prop:counter-ex}.  
We use the following definitions. 
\begin{definition}[\textit{Triangulation of a circular graph}]
\label{def:triangulation}
Let $\C$ be a simplicial $2$-complex 
homeomorphic to a disc, 
with all vertices lying on the boundary of the disc. 
In this case the incidence graph 
of triangles in $\C$ is a tree. 
One  associates with $\C$ 
a weighted circular graph that consists of
the $1$-complex $B=\partial\C$
by attributing the weight $-k+1$ 
to a vertex of degree $k$ of $\C$. 
We call $\C$ a \textit{triangulation} 
of the weighted circular graph $B$, 
and we say that $B$ is 
\textit{triangulable}. 
\end{definition}
\begin{definition}
We say that an inner segment $S$
 of a weighted graph $\Gamma$ 
is a \textit{charm earring} if $S$ 
consists of a single rational 
$(0)$-vertex of degree 2
adjacent to a branch vertex via two edges. 
We say that a minimal graph 
$\Gamma$ is \textit{admissible 
modulo charm earrings} 
if every non-admissible segment of 
$\Gamma$ is a charm earring, 
see Definition \ref{def:admissible}.
\end{definition}
\begin{proposition} \label{prop:counter-ex} 
Let $\Gamma_1$ be a graph that
is either admissible modulo
charm earring,
or coincides with one of the graphs
\[[[0]], \quad ((3)),\quad ((4)),\quad ((0,0))
\quad\text{and}\quad ((0,m))\quad\text{with}
\quad m\le -2.\] 
Then up to isomorphism $\Gamma_1$ 
is a unique minimal graph in its 
birational equivalence class.
\end{proposition}
\begin{proof} Assume that $\Gamma_2$ 
is a minimal graph 
birationally equivalent to $\Gamma_1$ 
and not isomorphic to 
$\Gamma_1$. 
By Lemma \ref{lem:rel-min} there exists 
a relatively minimal diagram 
\eqref{triangular diagram} 
with $\Gamma$ that dominates
$\Gamma_1$ and $\Gamma_2$ 
via  birational morphisms $p_1$ resp. $p_2$. 
We may suppose that $p_1$ is composed of $n\ge1$ 
blowdowns and $p_2$ 
also is composed of blowdowns. 
Let $p_1^{-1}\colon \Gamma_1
\dasharrow\Gamma$ 
add vertices $v_1,\ldots,v_n$ in this order, where $n\ge 1$. 
By the relative minimality assumption  
$p_2$ cannot contract any $(-1)$-vertex of $\Gamma$ 
among $v_1,\ldots,v_n$. In particular, 
it does not contract 
the $(-1)$-vertex $v_n$ of $\Gamma$. 

{\bf  Case 1: $\Gamma_1=[[0]]$.} 
Let $\Vert(\Gamma_1)=\{v_0\}$. 
We prove by induction that for every $k=1,\ldots, n$ 
the following hold.
\begin{enumerate}[(i)] 
\item[(i$_k$)] \textit{The first $k$ blowups in $p_1^{-1}$ 
are outer 
and yield the graph $[[-1,-2_{k-1},-1]]$, 
with vertices $v_0,\ldots,v_{k}$ 
appearing in this order;}
\item[(ii$_k$)] \textit{the further blowups in 
$p_1^{-1}$ do not 
change the weights of $v_0,\ldots,v_{k-1}$;}
\item[(iii$_k)$] \textit{the first $k$ blowdowns in $p_2$ 
contract the vertices $v_0,\ldots,v_{k-1}$ of 
$\Gamma$ in this order.}
\end{enumerate}
For $k=1$, (i$_1$)--(iii$_1$) hold because 
$p_2$ contracts no $(-1)$-vertex of 
$\Gamma$ different from $v_0$  
and $\Gamma\ominus\{v_0\}$ is not minimal,
 while $\Gamma_2$ is.
Suppose by induction that for some 
$k\in \{1,\ldots,n-1\}$, (i$_{k}$)--(iii$_{k}$) hold. 
By (ii$_{k}$), the $(k+1)$st blowup 
in $p_1^{-1}$ must be 
the outer blowup at $v_{k}$, 
hence (i$_{k+1}$) holds.
Let $\Gamma_{2,k}$ be
the graph obtained from $\Gamma$ under 
the first $k$ blowdowns in $p_2$.
By (ii$_{k}$) and (iii$_{k}$), 
we have a natural isomorphism of weighted graphs
$\Gamma_{2,k}\ominus\{v_{k}\}\cong
\Gamma\ominus\{v_0,\ldots,v_{k}\}$. The
latter graph contains
at most linear $(-1)$-vertex $v_n$. 
Hence it
is not minimal 
and no of its $(-1)$-vertices 
is contracted by $p_2$. 
Thus, $v_{k}$ is the only vertex of 
$\Gamma_{2,k}$ that can be 
contracted on the next blowdown 
 in $p_2$. 
Its weight is $-1$ in $\Gamma_{2,k}$ and 
$-2$ in $\Gamma$. 
Now (ii$_{k+1}$) and (iii$_{k+1}$) follow.
By recursion (i$_n$) holds, i.e. 
$\Gamma=[[-1,-2_{n-1},-1]]$. 
So $\Gamma_2= [[0]]$ due to 
the minimality of $\Gamma_2$.

{\bf  Case 2:} $\Gamma_1=((m,0))$ with $m\le-2$.
Since $\Gamma_1$ is circular, $\Gamma$ 
is also circular,
and $p_1,p_2$ consist of inner blowdowns, see 
Proposition \ref{cor:graph}(2). 
Let $\Vert(\Gamma_1)=\{u,v_0\}$, 
where $v_0$ is the $(0)$-vertex. 
The same recursion  as in Case 1 shows that for 
$k=1,\ldots,n$ we have:
\begin{enumerate}[(i)]
\item[(i$_k$)] \textit{ the first $k$ blowups in $p_1^{-1}$ 
yield the graph $((m-k,-1,-2_{k-1},-1))$ 
with vertices $u,v_0,\ldots,v_{k}$  
appearing in this order;}
\item[(ii$_k$)] \textit{ the further blowups in $p_1^{-1}$ 
do not change the weights of $v_0,\ldots,v_{k-1}$;}
\item[(iii$_k$)] \textit{ the first $k$ blowdowns in $p_2$ 
contract vertices $v_0,\ldots,v_{k-1}$.}
\end{enumerate}
The only additional observation that we need 
is the following: 
the weight of $u$ in $\Gamma_{2,k-1}$ is 
at most $m-k+(k-1)\le-3$, 
so $u$ cannot be contracted on the next step. 
Thus, we conclude as before that 
$\Gamma_2\cong((m,0))$.

{\bf  Case 3: $\Gamma_1=((3))$.} 
After the first blowup we get 
$\Gamma_1^\prime=((-1,-1))$  
with vertices $v_0$ and $v_1$, where $v_1$ 
is contracted in $\Gamma_1$. 
The second blowup would change the weight of 
$v_0$ to $-2$. The latter is not possible, 
since the vertex
$v_0$ must be 
contracted the first  under $p_2$.
So $\Gamma_2\cong((3))$.

{\bf  Case 4: $\Gamma_1=((4))$.} 
After the first blowup we get 
$\Gamma_1^\prime=((-1,0))$. 
Letting now $m=-1$, the argument from 
Case 2 can be  applied mutatis mutandis to 
$\Gamma_1^\prime$ instead of $\Gamma_1$.

{\bf  Case 5:} $\Gamma_1$ is admissible 
modulo charm earrings. 
By Proposition~\ref{cor:graph}, 
$p_2\circ p_1^{-1}\colon\Gamma_1
\dasharrow\Gamma_2$ 
induces a bijection between $\Br(\Gamma_1)$ 
and $\Br(\Gamma_2)$, an isomorphism between 
the corresponding admissible segments of 
$\Gamma_1$ and $\Gamma_2$ 
and a birational transformation between every 
segment $[[0]]$ coming from 
a charm earring $S$ of $\Gamma_1$ 
and the corresponding linear
segment of $\Gamma_2$. 
Let $v_0$ be the $(0)$-vertex of $S$ adjacent to 
a branch vertex $u\in\Br(\Gamma_1)$. 
Since $u$ cannot be contracted by $p_2$, 
the argument used in Case 1 can be applied to
the linear segment $[[0]]$ with vertex $v_0$.
This shows that $p_2\circ p_1^{-1}$ sends 
every charm earring of $\Gamma_1$ 
isomorphically onto a charm earring of $\Gamma_2$ 
and induces an isomorphism of weighted graphs
$\Gamma_1\cong\Gamma_2$.

{\bf Case 6:} $\Gamma_1=((0,0))$.
According to 
Proposition \ref{cor:graph}.2
the graphs  that appear
on successive steps of the birational transformation 
$\phi'=p_2\circ p_1^{-1}$ are circular. 
We claim that such a circular graph $B$
is triangulable, 
except for the initial graph $\Gamma_1=((0,0))$, 
see Definition \ref{def:triangulation}.
Indeed, the first blowup in $p_1^{-1}$ yields 
the circular graph $((-1,-1,-1))$ which has a triangulation 
with a single triangle.
Suppose that a circular graph $B$ is equipped 
with a triangulation $\C$. 
Then the inner blowup of $B$ 
at the edge $[v_i,v_{i+1}]$ 
that introduces
a new vertex $u$ results in gluing 
to $\C$ a new triangle
with vertices
$v_i,v_{i+1}$ and $u$. 
This adds a new 
leaf to the incidence tree of triangles. 
Conversely, a blowdown of a 
$(-1)$-vertex $v$ of $B$ 
results in removing the unique triangle in $\C$ 
with vertex $v$, provided $B$ 
is different from $((-1,-1,-1))$.

So $\Gamma_2$ is either $((0,0))$ or triangulable. 
In the latter case $\Gamma_2$ is not minimal. 
Indeed, the incidence tree of triangles associated 
with the triangulation $\C$ of $\Gamma_2$ 
contains a leaf that corresponds 
to a triangle with a vertex of weight $-1$.
\end{proof}
\begin{remark} An alternative approach is based 
on the following observation. Consider 
again a relatively minimal diagram 
\eqref{triangular diagram} 
with $\Gamma$ that dominates minimal graphs
$\Gamma_1$ and $\Gamma_2$, 
where $\Gamma_1$ as in 
Proposition \ref{prop:counter-ex} is not 
admissible modulo charm earrings 
and $\Gamma\neq ((-1,-1,-1))$. 
We claim that $\Gamma$ contains exactly two 
$(-1)$-vertices. 
Indeed, for the inertia indices of $\Gamma_1$
we have 
$i_0(\Gamma_1)+i_+(\Gamma_1)=1$. Let $v$ be 
a $(-1)$-vertex of $\Gamma$ that belongs to 
$\Vert(\Gamma_1)$. 
Then $v$ is the unique $(-1)$-vertex of $\Gamma$ 
contracted under $p_2$. 
Assume on the contrary that $\Gamma$ contains 
two distinct $(-1)$-vertices $u_1$ and $u_2$ 
different from $v$. 
They are not adjacent, see  Corollary 
\ref{cor:contractibility}(b), 
are not contained among the vertices of $\Gamma_1$ 
and are not contracted under $p_2$. 
Since $\Gamma_2$ is minimal, the weights 
of $u_1$ and $u_2$ 
in $\Gamma_2$ are non-negative. 
One can show that $u_1$ and $u_2$ in $\Gamma_2$ 
cannot be adjacent. Then the intersection form 
$I(\Gamma_2)$ 
is semipositive definite on the vector subspace spanned 
by the mutually orthogonal vectors $u_1$ and $u_2$. 
However, this contradicts the fact
that $i_0(\Gamma_2)+i_+(\Gamma_2)=
i_0(\Gamma_1)+i_+(\Gamma_1)=1$, 
see Corollary \ref{cor:inertia-indices}.
\end{remark}
\begin{remark} 
A linear segment is called \textit{standard}
 if it is one of the following:
\[ [[0_{2k}, w_1, \ldots, w_n]]\quad\text{and}
\quad [[0_{2k+1}]] \]
where $k, n \ge 0$ and $w_i \le -2$ $\forall i$. 
Similarly, a circular segment is called \textit{standard} if  
it is one of the following:
\[((0_{2k}, w_1, \ldots, w_n)),\quad ((0_{l}, w))
 \quad\text{and}\quad  ((0_{2k}, -1, -1))\]
where $k,l\ge 0$, $n>0$, $w\le 0$ and 
$w_i \le -2$ $\forall i$, see \cite[Definition 2.13]{FKZ-graphs}. 
Notice that for $w=0$ the second circular graph 
above becomes $((0_{l+1}))$. 
Any minimal segment is birationally equivalent to 
a standard one, see \cite[Theorem 2.15(b)]{FKZ-graphs}. 
Moreover, for a minimal linear segment $L$ 
its birational equivalence class contains at most 
two standard graphs related by a \textit{reversion}
\[ [[0_{2k},w_1,\ldots,w_n]] \mapsto
 [[0_{2k},w_n,\ldots,w_1]].\] 
For a minimal circular segment $C$,
 the standard graph is unique in 
 the birational equivalence class 
up to a cyclic permutation of its nonzero weights 
and reversion, see 
\cite[Corrigendum, Corollary 3.33]{FKZ-graphs}. 

For instance, the circular graphs with sequences 
$((-1,-1))$ and $((0,-1))$, 
respectively, are the unique standard graphs 
in the respective birational equivalence classes 
of $((3))$ and $((4))$, while $[[0]]$ is  
the unique standard graph in its birational 
equivalence class.
\end{remark}
\begin{lemma}\label{lem:min-mod} Let $\Gamma$ 
be a   (not necessarily minimal)  non-contractible 
connected
weighted graph.
Assume that  $\Gamma$
has a $(0)$-vertex $v$ of degree $1$ 
or $2$ with no incident loop and multiple edges.
Then $\Gamma$ admits an infinite number 
of non-isomorphic minimal models.
\end{lemma}
\begin{proof}
By our assumption $\Gamma$ contains an edge 
$[v,u]$ isomorphic to $[[0,c]]$ for some $c\in\ZZ$.
Applying iteratively elementary transformations at $v$,
\begin{equation}\label{eq:outer-elem-transf} [[0,c]]
\leadsto [[-1,-1,c]]\leadsto [[0,c+1]],\end{equation}
which is inner if $\deg(v)=2$ and outer otherwise, 
see Remark \ref{rem:up-and-down}, 
we obtain an infinite sequence
of graphs $\{\Gamma_n\}$ 
from the birational equivalence class of $\Gamma$
 with the same number of vertices, where $\Gamma_n$ 
 has a vertex of weight $n$.
Since $\Gamma_n$
 is not contractible, it dominates a minimal graph 
$\Gamma_n^\prime$. We claim that among the 
$\Gamma_n^\prime$ 
there exists an infinite number of non-isomorphic 
minimal graphs.
 Indeed, the contraction $\Gamma_n\to\Gamma_n^\prime$ 
 consists of at most $N={\rm card}(\Vert(\Gamma))$ 
 blowdowns.
Hence it drops the maximal weight of vertices in $\Gamma$
 at most by $4N$. 
Therefore, the range of maximal weights of the graphs
$\Gamma_n^\prime$ is unbounded, which proves our claim.
\end{proof}
\begin{theorem}\label{prop:classification}
Let $\Gamma$ be a minimal connected weighted graph. 
Then the birational equivalence class of $\Gamma$ 
contains  an infinite number of non-isomorphic 
minimal models if and only if 
 $\Gamma$ is not admissible modulo charm earrings 
 and is different from the graphs 
 $[[0]], ((3)), ((4))$, $((0,m))$ with $m\le 0$. 
 In the opposite case,  up to isomorphism $\Gamma$ is 
 a unique minimal graph in its birational equivalence class.
\end{theorem}

\begin{proof} The second assertion follows 
immediately from 
Proposition \ref{prop:counter-ex}. 
To show the first assertion,
consider first the case where $\Gamma$ 
is a non-admissible circular segment 
 with $N$ vertices.  
If $N\ge 3$, then there is a  
vertex $v$ of weight $a\ge 0$
that has degree $2$ and
two distinct neighbors. 
Performing $a$ inner blowups near $v$ we drop 
the weight of $v$ to $0$, and then 
Lemma~\ref{lem:min-mod} gives the result.

Let now $N=2$. If $\Gamma$ has a vertex $v$
of weight $a>0$, then after an inner blowup 
near $v$ we return to the previous case $N\ge 3$. 
Otherwise, $\Gamma=((0,m))$ with $m\le 0$, 
which is excluded by our assumption.
Finally, if $N=1$ and $\Gamma=((m))$ where $m\ge5$ 
by our assumption, then an inner blowup yields 
the 2-cycle $((m-4,-1))$ 
with a positive weight, that returns us to a previous case.

Finally, let $\Gamma$ be non-circular. 
Since by assumption $\Gamma$ is not admissible 
modulo charm earrings, 
$\Gamma\ominus\Br(\Gamma)$ 
contains a non-admissible linear segment with 
a vertex $v$ of weight $a\ge 0$, where $a>0$ if 
either $\Gamma=[[a]]$ or $v$ has two incident 
edges $[v,u]$ with the same branch vertex 
$u\in\Br(\Gamma)$.  
Performing $a$ blowups near $v$ 
we reduce the setup to the one of 
Lemma~\ref{lem:min-mod}, 
which implies the assertion. 
\end{proof}
%



\begin{thebibliography}{}
%
\bibitem{AG} 
I. V. Arzhantsev and S. A. Gaifullin, 
\textit{The automorphism group of 
a rigid affine variety}, 
Math. Nachr. 290 (2017), 662--671. 
%
\bibitem{BHPV} 
W.~P.~Barth, K.~Hulek, 
C.~A.~M.~Peters, and A.~Van de Ven,  
\textit{Compact complex surfaces}. 
Second edition. Springer-Verlag, Berlin, 2004.
%
\bibitem{Be} 
A.~Beauville, 
\textit{Complex algebraic surfaces}. 
Second edition. 
London Mathematical Society Student Texts. 34. 
Cambridge: Cambridge Univ. Press, 1996.
%
\bibitem{BF13}
J.~Blanc and J.-Ph.~Furter,
\textit{Topologies and structures of the Cremona groups},
Ann. Math. 178 (2013), 1173--1198.
%
\bibitem{Boldyrev-Gaifullin}
I.~Boldyrev and S.~Gaifullin,
\textit{Automorphisms of nonnormal toric varieties},
Math. Notes 110  (2021), 872--886. 
%
\bibitem{Borovik-Gaifullin}
V.~Borovik and S.~Gaifullin,
\textit{Isolated torus invariants and 
automorphism groups of rigid varieties},
J. Algebra 666 (2025), 821--839. 
%
\bibitem{Cantat-Xie}
S.~Cantat and J.~Xie, 
\textit{On degrees of birational mappings},
Math. Res. Lett. 27:2 (2020), 319--337.
%
\bibitem{Dai1} D.~Daigle, {\em Classification
of weighted graphs up to blowing-up and blowing-down}.
math.AG/0305029, 2003, 47 pp.
%
\bibitem{Dai2} D.~Daigle,
{\em Classification of linear weighted graphs
up to blowing-up and blowing-down}. 
Canad.\ J.\ Math.\ 60 (2008),
64--87.
%
\bibitem{DG}
V.~I.~Danilov and M.~H.~Gizatullin, 
{\em Automorphisms of affine
surfaces}. 
I. Math.\ USSR Izv.\ 9 (1975), 493--534; II. {\it ibid.}
11 (1977), 51--98.
%
\bibitem{Dub04}
A.~Dubouloz, 
\textit{Completions of normal affine surfaces 
with a trivial Makar-Limanov invariant}, 
Michigan Math. J. 52 (2004), 289--308. 
%
\bibitem{EN} 
D.~Eisenbud and W.~D.~Neumann,
\textit{Three-dimensional link theory and
invariants of plane curve singularities}, 
Annals of Math. Studies, Princeton Univ. Press, 1985.
%
\bibitem{Fie}
K.-H.~Fieseler.
{\em On complex affine surfaces with $\CC^+$-action}. 
Comment. Math. Helv. 69 (1994), 5--27.
%
\bibitem{FKZ-graphs} 
H.~Flenner, S.~Kaliman, and M.~Zaidenberg, 
\textit{ Birational
transformations of weighted graphs}, 
in: Affine algebraic geometry,  107--147.
Osaka Univ.\ Press, 2007.  
\textit{Corrigendum}, in: Affine algebraic geometry, 123--163. 
Centre de Recherches
Math\'ematiques. CRM Proc.\ Lecture Notes, 54, 
Amer.\ Math.\ Soc., Providence, RI, 2011. 
%
\bibitem{FKZ-DGthm} 
H.~Flenner, S.~Kaliman, and M.~Zaidenberg, 
\textit{On the Danilov-Gizatullin isomorphism theorem},
 Enseign. Math. (2) 55 (2009), 275--283.
%
\bibitem{FKZ-deformations} 
H.~Flenner, S.~Kaliman, and M.~Zaidenberg, 
\textit{Deformation equivalence of affine ruled surfaces},  
arXiv:1305.5366 (2013), 35p.
%
\bibitem{FZ}  H.~Flenner and M.~Zaidenberg, 
\textit{$\QQ$-acyclic surfaces and their deformations}, 
Classification of algebraic varieties (L'Aquila, 1992), 143--208. 
Contemp. Math.\ 162, Amer. Math. Soc., 
Providence, RI, 1994.
%
\bibitem{FZ-LND}  H.~Flenner and M.~Zaidenberg, 
\textit{Locally nilpotent derivations on 
affine surfaces with a ${\CC}^*$-action}, 
Osaka J.\ Math.\ 42 (2005), 931--974.
%
\bibitem{Fr} G.~Freudenburg, 
\textit {Algebraic theory
of locally nilpotent derivations}, Encyclopaedia of Mathematical
Sciences,  136. Invariant Theory and Algebraic Transformation
Groups, VII. Springer-Verlag, Berlin, 2006.
%
\bibitem{Fu} T. Fujita, 
\textit {On the topology of non-complete algebraic surfaces},
J. Fac. Sci., Univ. Tokyo, Sect. I A 29  (1982), 503--566.
%
\bibitem{FK}
J.-Ph.~Furter and H.~Kraft, 
\textit{On the geometry of automorphism
groups of affine varieties},  arXiv:1809.04175 (2018), 1--179.
%
\bibitem{Gi1} 
M.~H.~Gizatullin, 
\textit{On affine surfaces that can be 
completed by a smooth rational curve},
Izv. Math.\ 4 (1970), 787--810.
%
\bibitem{Gi2} 
M.~H.~Gizatullin, 
\textit{Invariants of incomplete algebraic surfaces 
obtained by means of completions}, 
Izv. Math.\ 5 (1971), 503--515 (1972).
%
\bibitem{Giz71b}
M.~H.~Gizatullin,
\textit{Quasihomogeneous affine surfaces}, 
Math. USSR Izv. 5 (1971), 1057--1081. 
%
\bibitem{Go69} 
J.~E.~Goodman, \textit{Affine open subsets 
of algebraic varieties 
and ample divisors.} Ann. Math. 89 (1969), 160--183.
%
\bibitem{Har70}
R.~Hartshorne, 
\textit{Ample subvarieties of algebraic varieties. 
Notes written in collaboration with C. Musili.}
Lecture Notes in Math. vol. 156. Springer-Verlag, 
Berlin-New York 1970.
%
\bibitem{Hi} F.~Hirzebruch, 
\textit{The topology of normal singularities 
of an algebraic surface (after D. Mumford)}, 
S\'eminaire Bourbaki, Vol. 8, Exp. No. 250, 129--137. 
Soc. Math. France, Paris, 1995.
%
\bibitem{Ii} S.~Iitaka, 
\textit{On logarithmic Kodaira dimension of
algebraic varieties}, Compl.\ Anal.\ and Algebr.\ Geom., 
Tokyo, Iwanami Shoten, 1977, 175--189.
%
\bibitem{Je} Z.~Jelonek,
 \textit{On the group of automorphisms of a quasi-affine variety}, 
Math. Ann. 362  (2015), 569--578. 

\bibitem{JL} Z.~Jelonek and T.~Lenarcik,
 \textit{Automorphisms of affine smooth varieties}, 
Proc.\ Amer.\ Math.\ Soc.\ 142 (2014), 1157--1163.
%
\bibitem{KPZ} T.~Kishimoto, Yu.~Prokhorov, and M.~Zaidenberg, 
\textit{Group actions on affine cones}, 
Affine algebraic geometry, 123--163, 
CRM Proc.\ Lecture Notes, 54, Amer.\ Math.\ Soc., Providence, RI, 2011.
%
\bibitem{Kr} H.~Kraft, 
\textit{Automorphism groups of affine varieties 
and a characterization of affine $n$-space},
Trans. Mosc. Math. Soc. 78 (2017), 171--186. 
Translated from: Tr. Mosk. Mat. O.-va 78 (2017), 209--226.
%
\bibitem{La} S.~Lamy,   
\textit{Sur la structure du groupe d'automorphismes 
de certaines surfaces affines}, 
Publ.\ Mat.\ 49 (2005),  3--20. 
%
\bibitem{Laz04}
R.~Lazarsfeld.
\textit{Positivity
in Algebraic Geometry I.
Classical Setting:
Line Bundles and Linear Series.}
Ergebnisse der Mathematik und ihrer Grenzgebiete
3. Folge, vol. 48.
A Series of Modern Surveys in Mathematics. 
Springer-Verlag, Berlin Heidelberg 2004.
%
\bibitem{MO} H.~Matsumura and F.~Oort,  
\textit{Representability of group functors, 
and automorphisms of algebraic schemes}, 
Invent.\ Math.\ 4 (1967), 1--25.
%
\bibitem{Ma} T.~Matsusaka, 
\textit{Polarized varieties, fields of moduli and generalized Kummer
varieties of polarized abelian varieties}, 
Amer.\ J.\ Math.\ 80 (1958), 45--82.
%
\bibitem{Mu} D. Mumford,
\textit{Algebraic Geometry. I: Complex projective varieties},
Classics in Mathematics. Springer-Verlag, Berlin,  1995.
%
\bibitem{Ne}
W.~D.~Neumann, \textit{On bilinear forms represented by trees}, 
Bull. Austral. Math. Soc. 40 (1989), 303--321.
\bibitem{OZ}
S.~Y.~Orevkov and  M.~G.~Zaidenberg, 
\textit{Some estimates for plane cuspidal curves}, in:
\textit{S\'eminaire d'alg\`ebre et g\'eom\'etrie, journ\'ees singuli\`eres et jacobiennes}. 
Cours de l'institut Fourier, S24 (1993), 93--116. 
http://www.numdam.org/item/CIF$\_$1993$\_\!\!\_$S24$\_\!\!\_$93$\_$0.pdf
%
\bibitem{PR1}
A.~Perepechko and A.~Regeta,
\textit{When is the automorphism group of an affine variety nested?},
Transform. groups 28  (2023), 401--412. 
%
\bibitem{PR2}
A.~Perepechko and A.~Regeta,
\textit{Automorphism groups of affine varieties 
without non-algebraic elements}, 
Proc. Amer. Math. Soc. 152  (2024), No. 6, 2377--2383.
%
\bibitem{Ra} C.~P.~Ramanujam,  \textit{A note on automorphism groups of algebraic varieties}, 
Math.\ Ann.\ 156 (1964), 25--33.
%
\bibitem{Ra1} C.~P.~Ramanujam,  \textit{A topological characterization 
of the affine plane as an algebraic variety}, 
Ann. Math. 94 (1971), 69--88. 
%
\bibitem{Ru}
P.~Russell, 
\textit{Some formal aspects of the theorems of Mumford-Ramanujam},
Parimala, R. (ed.), Proceedings of the international colloquium on algebra, arithmetic and geometry, 
Mumbai, India, 2000. Part I and II. New Delhi: Narosa Publishing House. Stud. Math., 
Tata Inst. Fundam. Res. 16  (2002), 557--584.
%
\bibitem{SP-CA} The {Stacks project authors}, 
\textit{The Stacks project}, 
\url{https://stacks.math.columbia.edu}, 2024.
%
\end{thebibliography}
\end{document}